%% file: FirstOrderFEMMomentModelsPart1.tex
\documentclass[preprint,3p,sort&compress]{elsarticle}
\journal{{\tt arXiv.org}}

\usepackage[dvips]{epsfig}
\usepackage{graphicx}
\usepackage{pdfpages}
\usepackage{latexsym}
\usepackage{verbatim}
\usepackage{amsmath,amsthm}
\usepackage{amssymb}
\usepackage{bbm}
\usepackage{mathtools}
\usepackage{fonts}
\usepackage{enumerate}
\usepackage{placeins}
\usepackage{standalone}
\usepackage{paralist}
\usepackage{subcaption}
\usepackage{blindtext}

\definecolor{greenyellow}   {cmyk}{0.15, 0   , 0.69, 0   }
\definecolor{yellow}        {cmyk}{0   , 0   , 1   , 0   }
\definecolor{goldenrod}     {cmyk}{0   , 0.10, 0.84, 0   }
\definecolor{dandelion}     {cmyk}{0   , 0.29, 0.84, 0   }
\definecolor{apricot}       {cmyk}{0   , 0.32, 0.52, 0   }
\definecolor{peach}         {cmyk}{0   , 0.50, 0.70, 0   }
\definecolor{melon}         {cmyk}{0   , 0.46, 0.50, 0   }
\definecolor{yelloworange}  {cmyk}{0   , 0.42, 1   , 0   }
\definecolor{orange}        {cmyk}{0   , 0.61, 0.87, 0   }
\definecolor{burntorange}   {cmyk}{0   , 0.51, 1   , 0   }
\definecolor{bittersweet}   {cmyk}{0   , 0.75, 1   , 0.24}
\definecolor{redorange}     {cmyk}{0   , 0.77, 0.87, 0   }
\definecolor{mahogany}      {cmyk}{0   , 0.85, 0.87, 0.35}
\definecolor{maroon}        {cmyk}{0   , 0.87, 0.68, 0.32}
\definecolor{brickred}      {cmyk}{0   , 0.89, 0.94, 0.28}
\definecolor{red}           {cmyk}{0   , 1   , 1   , 0   }
\definecolor{orangered}     {cmyk}{0   , 1   , 0.50, 0   }
\definecolor{rubinered}     {cmyk}{0   , 1   , 0.13, 0   }
\definecolor{wildstrawberry}{cmyk}{0   , 0.96, 0.39, 0   }
\definecolor{salmon}        {cmyk}{0   , 0.53, 0.38, 0   }
\definecolor{carnationpink} {cmyk}{0   , 0.63, 0   , 0   }
\definecolor{magenta}       {cmyk}{0   , 1   , 0   , 0   }
\definecolor{violetred}     {cmyk}{0   , 0.81, 0   , 0   }
\definecolor{rhodamine}     {cmyk}{0   , 0.82, 0   , 0   }
\definecolor{mulberry}      {cmyk}{0.34, 0.90, 0   , 0.02}
\definecolor{redviolet}     {cmyk}{0.07, 0.90, 0   , 0.34}
\definecolor{fuchsia}       {cmyk}{0.47, 0.91, 0   , 0.08}
\definecolor{lavender}      {cmyk}{0   , 0.48, 0   , 0   }
\definecolor{thistle}       {cmyk}{0.12, 0.59, 0   , 0   }
\definecolor{orchid}        {cmyk}{0.32, 0.64, 0   , 0   }
\definecolor{darkorchid}    {cmyk}{0.40, 0.80, 0.20, 0   }
\definecolor{purple}        {cmyk}{0.45, 0.86, 0   , 0   }
\definecolor{plum}          {cmyk}{0.50, 1   , 0   , 0   }
\definecolor{violet}        {cmyk}{0.79, 0.88, 0   , 0   }
\definecolor{royalpurple}   {cmyk}{0.75, 0.90, 0   , 0   }
\definecolor{blueviolet}    {cmyk}{0.86, 0.91, 0   , 0.04}
\definecolor{periwinkle}    {cmyk}{0.57, 0.55, 0   , 0   }
\definecolor{cadetblue}     {cmyk}{0.62, 0.57, 0.23, 0   }
\definecolor{cornflowerblue}{cmyk}{0.65, 0.13, 0   , 0   }
\definecolor{midnightblue}  {cmyk}{0.98, 0.13, 0   , 0.43}
\definecolor{navyblue}      {cmyk}{0.94, 0.54, 0   , 0   }
\definecolor{royalblue}     {cmyk}{1   , 0.50, 0   , 0   }
\definecolor{blue}          {cmyk}{1   , 1   , 0   , 0   }
\definecolor{cerulean}      {cmyk}{0.94, 0.11, 0   , 0   }
\definecolor{cyan}          {cmyk}{1   , 0   , 0   , 0   }
\definecolor{processblue}   {cmyk}{0.96, 0   , 0   , 0   }
\definecolor{skyblue}       {cmyk}{0.62, 0   , 0.12, 0   }
\definecolor{turquoise}     {cmyk}{0.85, 0   , 0.20, 0   }
\definecolor{tealblue}      {cmyk}{0.86, 0   , 0.34, 0.02}
\definecolor{aquamarine}    {cmyk}{0.82, 0   , 0.30, 0   }
\definecolor{bluegreen}     {cmyk}{0.85, 0   , 0.33, 0   }
\definecolor{emerald}       {cmyk}{1   , 0   , 0.50, 0   }
\definecolor{junglegreen}   {cmyk}{0.99, 0   , 0.52, 0   }
\definecolor{seagreen}      {cmyk}{0.69, 0   , 0.50, 0   }
\definecolor{green}         {cmyk}{1   , 0   , 1   , 0   }
\definecolor{forestgreen}   {cmyk}{0.91, 0   , 0.88, 0.12}
\definecolor{pinegreen}     {cmyk}{0.92, 0   , 0.59, 0.25}
\definecolor{limegreen}     {cmyk}{0.50, 0   , 1   , 0   }
\definecolor{yellowgreen}   {cmyk}{0.44, 0   , 0.74, 0   }
\definecolor{springgreen}   {cmyk}{0.26, 0   , 0.76, 0   }
\definecolor{olivegreen}    {cmyk}{0.64, 0   , 0.95, 0.40}
\definecolor{rawsienna}     {cmyk}{0   , 0.72, 1   , 0.45}
\definecolor{sepia}         {cmyk}{0   , 0.83, 1   , 0.70}
\definecolor{brown}         {cmyk}{0   , 0.81, 1   , 0.60}
\definecolor{tan}           {cmyk}{0.14, 0.42, 0.56, 0   }
\definecolor{gray}          {cmyk}{0   , 0   , 0   , 0.50}
\definecolor{black}         {cmyk}{0   , 0   , 0   , 1   }
\definecolor{white}         {cmyk}{0   , 0   , 0   , 0   }

 \usepackage[]{hyperref}

\usepackage{pgfplots}
\pgfplotsset{compat=newest}

\newcommand{\externaltikz}[2]{\includegraphics{Externals/#1}} 
\newtheorem{theorem}{Theorem}[section]
\newtheorem{definition}[theorem]{Definition}
\newtheorem{remark}[theorem]{Remark}
\newtheorem{example}[theorem]{Example}

\newtheorem{lemma}[theorem]{Lemma}
\newtheorem{corollary}[theorem]{Corollary}

\newcounter{tikzsubfigcounter}[figure]
\renewcommand{\thetikzsubfigcounter}{\thesection.\the\numexpr\value{figure}+1\relax\alph{tikzsubfigcounter}}

\newcounter{tikzsubfigcounterinvisible}[figure]
\renewcommand{\thetikzsubfigcounterinvisible}{\the\numexpr\value{figure}+1\relax\alph{tikzsubfigcounterinvisible}}

\newcommand{\settikzlabel}[1]{ %
\refstepcounter{tikzsubfigcounterinvisible} \label{#1}
}

\numberwithin{equation}{section}

\newcommand{\bdm}{\begin{displaymath}}
\newcommand{\edm}{\end{displaymath}}
\newcommand{\beq}{\begin{equation}}
\newcommand{\eeq}{\end{equation}}
\newcommand{\beqa}{\begin{eqnarray}}
\newcommand{\eeqa}{\end{eqnarray}}

\title{First-order continuous- and discontinuous-Galerkin moment models for a linear kinetic
equation: model derivation and realizability theory}
\author[fs]{Florian Schneider}
\address[fs]{Fachbereich Mathematik, TU Kaiserslautern, Erwin-Schr\"odinger-Str., 67663 Kaiserslautern, Germany, {\tt schneider@mathematik.uni-kl.de}}
\author[tl]{Tobias Leibner}
\address[tl]{Fachbereich Mathematik und Informatik, WWU M\"unster, Einsteinstrasse 62, 48149 M\"unster, {\tt tobias.leibner@uni-muenster.de}}
\date{}

\input{usermacros}

\begin{document}

\begin{abstract}
We provide two new classes of moment models for linear kinetic equations in slab and three-dimensional geometry.
They are based on classical finite elements and low-order discontinuous-Galerkin approximations on the unit sphere.
We investigate their realizability conditions and other basic properties.
Numerical tests show that these models are more efficient than classical full-moment models in a space-homogeneous test,
when the analytical solution is not smooth.
\end{abstract}
\maketitle

\noindent


\def\tikzpath{Images/}

\input{Sections/introduction}
\input{Sections/modelling}
\input{Sections/momentmodels}

\input{Sections/realizabilitySlab}
\input{Sections/results}
\input{Sections/outlook}

\bibliographystyle{siamplain}
\bibliography{bibliography}

\end{document}

%% file: usermacros.tex
\usepgfplotslibrary{groupplots}
\usepackage{booktabs} 
\newlength{\figureheight}
\newlength{\figurewidth}
\pgfplotsset{%
	,compat=1.11
	,colormap={parula}{%
rgb(0pt)=(0.2081,0.1663,0.5292);
rgb(1pt)=(0.208355,0.16778,0.532238);
rgb(2pt)=(0.208611,0.169261,0.535275);
rgb(3pt)=(0.208866,0.170741,0.538313);
rgb(4pt)=(0.209121,0.172222,0.54135);
rgb(5pt)=(0.209376,0.173702,0.544388);
rgb(6pt)=(0.209632,0.175183,0.547425);
rgb(7pt)=(0.209887,0.176663,0.550463);
rgb(8pt)=(0.210134,0.178144,0.553505);
rgb(9pt)=(0.210338,0.179624,0.556568);
rgb(10pt)=(0.210542,0.181105,0.559631);
rgb(11pt)=(0.210746,0.182585,0.562694);
rgb(12pt)=(0.210944,0.184066,0.565763);
rgb(13pt)=(0.211123,0.185546,0.568852);
rgb(14pt)=(0.211302,0.187027,0.57194);
rgb(15pt)=(0.21148,0.188507,0.575029);
rgb(16pt)=(0.211642,0.189996,0.578117);
rgb(17pt)=(0.21177,0.191502,0.581206);
rgb(18pt)=(0.211897,0.193008,0.584295);
rgb(19pt)=(0.212025,0.194514,0.587383);
rgb(20pt)=(0.212132,0.19602,0.590472);
rgb(21pt)=(0.212208,0.197526,0.59356);
rgb(22pt)=(0.212285,0.199032,0.596649);
rgb(23pt)=(0.212361,0.200538,0.599738);
rgb(24pt)=(0.212413,0.202044,0.602839);
rgb(25pt)=(0.212438,0.20355,0.605953);
rgb(26pt)=(0.212464,0.205056,0.609067);
rgb(27pt)=(0.212489,0.206562,0.612181);
rgb(28pt)=(0.212471,0.208083,0.61531);
rgb(29pt)=(0.21242,0.209614,0.61845);
rgb(30pt)=(0.212368,0.211146,0.621589);
rgb(31pt)=(0.212317,0.212677,0.624729);
rgb(32pt)=(0.212216,0.214209,0.627868);
rgb(33pt)=(0.212088,0.215741,0.631008);
rgb(34pt)=(0.211961,0.217272,0.634148);
rgb(35pt)=(0.211833,0.218804,0.637287);
rgb(36pt)=(0.211668,0.220354,0.640446);
rgb(37pt)=(0.211489,0.221911,0.643611);
rgb(38pt)=(0.21131,0.223468,0.646776);
rgb(39pt)=(0.211132,0.225025,0.649941);
rgb(40pt)=(0.210848,0.226603,0.653107);
rgb(41pt)=(0.210541,0.228186,0.656272);
rgb(42pt)=(0.210235,0.229768,0.659437);
rgb(43pt)=(0.209929,0.231351,0.662602);
rgb(44pt)=(0.209553,0.232934,0.665767);
rgb(45pt)=(0.20917,0.234516,0.668932);
rgb(46pt)=(0.208787,0.236099,0.672098);
rgb(47pt)=(0.208405,0.237681,0.675263);
rgb(48pt)=(0.20787,0.239289,0.678453);
rgb(49pt)=(0.207334,0.240897,0.681644);
rgb(50pt)=(0.206798,0.242505,0.684835);
rgb(51pt)=(0.206255,0.244114,0.688025);
rgb(52pt)=(0.205617,0.245722,0.691216);
rgb(53pt)=(0.204979,0.24733,0.694407);
rgb(54pt)=(0.204341,0.248938,0.697597);
rgb(55pt)=(0.203675,0.250554,0.700792);
rgb(56pt)=(0.202858,0.252213,0.704008);
rgb(57pt)=(0.202041,0.253872,0.707224);
rgb(58pt)=(0.201225,0.255531,0.710441);
rgb(59pt)=(0.200372,0.257184,0.713657);
rgb(60pt)=(0.199402,0.258818,0.716873);
rgb(61pt)=(0.198432,0.260452,0.720089);
rgb(62pt)=(0.197462,0.262085,0.723305);
rgb(63pt)=(0.196419,0.263735,0.726522);
rgb(64pt)=(0.195219,0.26542,0.729738);
rgb(65pt)=(0.19402,0.267105,0.732954);
rgb(66pt)=(0.19282,0.268789,0.73617);
rgb(67pt)=(0.191549,0.270474,0.739386);
rgb(68pt)=(0.19017,0.272159,0.742603);
rgb(69pt)=(0.188792,0.273843,0.745819);
rgb(70pt)=(0.187414,0.275528,0.749035);
rgb(71pt)=(0.1859,0.277237,0.752264);
rgb(72pt)=(0.184241,0.278973,0.755505);
rgb(73pt)=(0.182581,0.280709,0.758747);
rgb(74pt)=(0.180922,0.282444,0.761989);
rgb(75pt)=(0.179133,0.284209,0.765245);
rgb(76pt)=(0.177244,0.285996,0.768512);
rgb(77pt)=(0.175356,0.287783,0.77178);
rgb(78pt)=(0.173467,0.289569,0.775047);
rgb(79pt)=(0.171363,0.291406,0.778314);
rgb(80pt)=(0.169142,0.293269,0.781581);
rgb(81pt)=(0.166922,0.295132,0.784849);
rgb(82pt)=(0.164701,0.296996,0.788116);
rgb(83pt)=(0.162238,0.298934,0.791365);
rgb(84pt)=(0.159686,0.300899,0.794606);
rgb(85pt)=(0.157133,0.302865,0.797848);
rgb(86pt)=(0.15458,0.30483,0.80109);
rgb(87pt)=(0.151738,0.306858,0.804352);
rgb(88pt)=(0.148828,0.3089,0.80762);
rgb(89pt)=(0.145918,0.310942,0.810887);
rgb(90pt)=(0.143008,0.312984,0.814154);
rgb(91pt)=(0.139687,0.31514,0.81733);
rgb(92pt)=(0.136318,0.31731,0.820495);
rgb(93pt)=(0.132949,0.319479,0.82366);
rgb(94pt)=(0.129579,0.321649,0.826826);
rgb(95pt)=(0.125811,0.323918,0.829841);
rgb(96pt)=(0.122033,0.32619,0.832853);
rgb(97pt)=(0.118256,0.328462,0.835865);
rgb(98pt)=(0.114458,0.330737,0.838862);
rgb(99pt)=(0.110349,0.333059,0.841619);
rgb(100pt)=(0.106239,0.335382,0.844376);
rgb(101pt)=(0.102129,0.337705,0.847132);
rgb(102pt)=(0.0979874,0.340021,0.849835);
rgb(103pt)=(0.093648,0.342292,0.852209);
rgb(104pt)=(0.0893087,0.344564,0.854583);
rgb(105pt)=(0.0849694,0.346836,0.856957);
rgb(106pt)=(0.08063,0.349091,0.859234);
rgb(107pt)=(0.0762907,0.351286,0.861174);
rgb(108pt)=(0.0719514,0.353481,0.863114);
rgb(109pt)=(0.067612,0.355676,0.865053);
rgb(110pt)=(0.0633195,0.357817,0.866853);
rgb(111pt)=(0.0591333,0.359833,0.868333);
rgb(112pt)=(0.0549471,0.36185,0.869814);
rgb(113pt)=(0.050761,0.363866,0.871294);
rgb(114pt)=(0.0466838,0.365823,0.872626);
rgb(115pt)=(0.0427784,0.367687,0.873724);
rgb(116pt)=(0.038873,0.36955,0.874821);
rgb(117pt)=(0.0349676,0.371414,0.875919);
rgb(118pt)=(0.0315066,0.373217,0.876872);
rgb(119pt)=(0.0285456,0.374953,0.877664);
rgb(120pt)=(0.0255847,0.376688,0.878455);
rgb(121pt)=(0.0226237,0.378424,0.879246);
rgb(122pt)=(0.0202132,0.380061,0.879868);
rgb(123pt)=(0.0182477,0.381618,0.880353);
rgb(124pt)=(0.0162823,0.383175,0.880838);
rgb(125pt)=(0.0143168,0.384732,0.881323);
rgb(126pt)=(0.0127892,0.386241,0.881695);
rgb(127pt)=(0.0115129,0.387721,0.882001);
rgb(128pt)=(0.0102366,0.389202,0.882307);
rgb(129pt)=(0.00896036,0.390682,0.882614);
rgb(130pt)=(0.00812372,0.392089,0.88281);
rgb(131pt)=(0.00746006,0.393468,0.882963);
rgb(132pt)=(0.0067964,0.394846,0.883116);
rgb(133pt)=(0.00613273,0.396224,0.883269);
rgb(134pt)=(0.00581622,0.397562,0.88332);
rgb(135pt)=(0.00558649,0.398889,0.883346);
rgb(136pt)=(0.00535676,0.400217,0.883371);
rgb(137pt)=(0.00512703,0.401544,0.883397);
rgb(138pt)=(0.00516757,0.402804,0.883332);
rgb(139pt)=(0.00524414,0.404054,0.883256);
rgb(140pt)=(0.00532072,0.405305,0.883179);
rgb(141pt)=(0.0053973,0.406556,0.883103);
rgb(142pt)=(0.00572012,0.407757,0.882952);
rgb(143pt)=(0.00605195,0.408957,0.882799);
rgb(144pt)=(0.00638378,0.410157,0.882646);
rgb(145pt)=(0.00672643,0.411355,0.882489);
rgb(146pt)=(0.00728799,0.412529,0.882259);
rgb(147pt)=(0.00784955,0.413704,0.88203);
rgb(148pt)=(0.00841111,0.414878,0.8818);
rgb(149pt)=(0.00898919,0.416045,0.881564);
rgb(150pt)=(0.00967838,0.417168,0.881283);
rgb(151pt)=(0.0103676,0.418292,0.881002);
rgb(152pt)=(0.0110568,0.419415,0.880721);
rgb(153pt)=(0.011773,0.420532,0.880435);
rgb(154pt)=(0.0125898,0.42163,0.880129);
rgb(155pt)=(0.0134066,0.422728,0.879823);
rgb(156pt)=(0.0142234,0.423825,0.879516);
rgb(157pt)=(0.0150703,0.424915,0.879195);
rgb(158pt)=(0.0159892,0.425987,0.878838);
rgb(159pt)=(0.0169081,0.427059,0.87848);
rgb(160pt)=(0.017827,0.428132,0.878123);
rgb(161pt)=(0.0187748,0.429194,0.877746);
rgb(162pt)=(0.0197703,0.430241,0.877338);
rgb(163pt)=(0.0207658,0.431287,0.876929);
rgb(164pt)=(0.0217613,0.432334,0.876521);
rgb(165pt)=(0.0227802,0.43338,0.876113);
rgb(166pt)=(0.0238267,0.434427,0.875704);
rgb(167pt)=(0.0248733,0.435473,0.875296);
rgb(168pt)=(0.0259198,0.43652,0.874887);
rgb(169pt)=(0.0269802,0.437553,0.874451);
rgb(170pt)=(0.0280523,0.438574,0.873992);
rgb(171pt)=(0.0291243,0.439595,0.873532);
rgb(172pt)=(0.0301964,0.440616,0.873073);
rgb(173pt)=(0.0312844,0.441621,0.872614);
rgb(174pt)=(0.032382,0.442616,0.872154);
rgb(175pt)=(0.0334796,0.443612,0.871695);
rgb(176pt)=(0.0345772,0.444607,0.871235);
rgb(177pt)=(0.0357108,0.445603,0.870758);
rgb(178pt)=(0.0368595,0.446598,0.870273);
rgb(179pt)=(0.0380081,0.447594,0.869788);
rgb(180pt)=(0.0391568,0.448589,0.869303);
rgb(181pt)=(0.0402652,0.449565,0.868798);
rgb(182pt)=(0.0413628,0.450535,0.868287);
rgb(183pt)=(0.0424604,0.451505,0.867777);
rgb(184pt)=(0.043558,0.452474,0.867266);
rgb(185pt)=(0.0445889,0.453444,0.866756);
rgb(186pt)=(0.0456099,0.454414,0.866245);
rgb(187pt)=(0.0466309,0.455384,0.865735);
rgb(188pt)=(0.047652,0.456354,0.865224);
rgb(189pt)=(0.0486,0.457324,0.864714);
rgb(190pt)=(0.0495444,0.458294,0.864203);
rgb(191pt)=(0.0504889,0.459264,0.863692);
rgb(192pt)=(0.0514315,0.460234,0.863181);
rgb(193pt)=(0.0523249,0.461204,0.862645);
rgb(194pt)=(0.0532183,0.462174,0.862109);
rgb(195pt)=(0.0541117,0.463144,0.861573);
rgb(196pt)=(0.0549991,0.464111,0.861034);
rgb(197pt)=(0.0558414,0.465056,0.860472);
rgb(198pt)=(0.0566838,0.466,0.859911);
rgb(199pt)=(0.0575261,0.466944,0.859349);
rgb(200pt)=(0.0583532,0.467889,0.858793);
rgb(201pt)=(0.0591189,0.468833,0.858257);
rgb(202pt)=(0.0598847,0.469778,0.857721);
rgb(203pt)=(0.0606505,0.470722,0.857185);
rgb(204pt)=(0.0614018,0.471667,0.856641);
rgb(205pt)=(0.0621165,0.472611,0.85608);
rgb(206pt)=(0.0628312,0.473556,0.855518);
rgb(207pt)=(0.0635459,0.4745,0.854957);
rgb(208pt)=(0.064242,0.475444,0.854405);
rgb(209pt)=(0.0649057,0.476389,0.853868);
rgb(210pt)=(0.0655694,0.477333,0.853332);
rgb(211pt)=(0.066233,0.478278,0.852796);
rgb(212pt)=(0.0668625,0.479222,0.852249);
rgb(213pt)=(0.0674495,0.480167,0.851687);
rgb(214pt)=(0.0680366,0.481111,0.851126);
rgb(215pt)=(0.0686237,0.482056,0.850564);
rgb(216pt)=(0.0691838,0.483,0.850003);
rgb(217pt)=(0.0697198,0.483944,0.849441);
rgb(218pt)=(0.0702559,0.484889,0.84888);
rgb(219pt)=(0.0707919,0.485833,0.848318);
rgb(220pt)=(0.0712967,0.486778,0.847772);
rgb(221pt)=(0.0717817,0.487722,0.847236);
rgb(222pt)=(0.0722667,0.488667,0.8467);
rgb(223pt)=(0.0727517,0.489611,0.846164);
rgb(224pt)=(0.0732012,0.490573,0.845628);
rgb(225pt)=(0.0736351,0.491543,0.845092);
rgb(226pt)=(0.0740691,0.492513,0.844556);
rgb(227pt)=(0.074503,0.493483,0.84402);
rgb(228pt)=(0.0748973,0.494433,0.843484);
rgb(229pt)=(0.0752802,0.495378,0.842948);
rgb(230pt)=(0.0756631,0.496322,0.842412);
rgb(231pt)=(0.0760459,0.497267,0.841876);
rgb(232pt)=(0.0763631,0.498233,0.841362);
rgb(233pt)=(0.0766694,0.499203,0.840851);
rgb(234pt)=(0.0769757,0.500173,0.840341);
rgb(235pt)=(0.077282,0.501143,0.83983);
rgb(236pt)=(0.0775162,0.502137,0.83932);
rgb(237pt)=(0.0777459,0.503132,0.838809);
rgb(238pt)=(0.0779757,0.504128,0.838298);
rgb(239pt)=(0.0782042,0.505123,0.837789);
rgb(240pt)=(0.0783829,0.506093,0.837304);
rgb(241pt)=(0.0785616,0.507063,0.836819);
rgb(242pt)=(0.0787402,0.508033,0.836334);
rgb(243pt)=(0.0789135,0.509008,0.835851);
rgb(244pt)=(0.0790411,0.510029,0.835392);
rgb(245pt)=(0.0791688,0.51105,0.834932);
rgb(246pt)=(0.0792964,0.512071,0.834473);
rgb(247pt)=(0.0794048,0.513092,0.834018);
rgb(248pt)=(0.0794303,0.514113,0.833584);
rgb(249pt)=(0.0794559,0.515134,0.83315);
rgb(250pt)=(0.0794814,0.516155,0.832717);
rgb(251pt)=(0.0794862,0.517183,0.832289);
rgb(252pt)=(0.0794351,0.51823,0.831881);
rgb(253pt)=(0.0793841,0.519276,0.831473);
rgb(254pt)=(0.079333,0.520323,0.831064);
rgb(255pt)=(0.079255,0.521369,0.830665);
rgb(256pt)=(0.0791273,0.522416,0.830282);
rgb(257pt)=(0.0789997,0.523462,0.829899);
rgb(258pt)=(0.0788721,0.524509,0.829516);
rgb(259pt)=(0.0786889,0.525589,0.829156);
rgb(260pt)=(0.0784336,0.526712,0.828824);
rgb(261pt)=(0.0781784,0.527835,0.828492);
rgb(262pt)=(0.0779231,0.528958,0.82816);
rgb(263pt)=(0.077615,0.530081,0.827868);
rgb(264pt)=(0.0772577,0.531205,0.827613);
rgb(265pt)=(0.0769003,0.532328,0.827357);
rgb(266pt)=(0.0765429,0.533451,0.827102);
rgb(267pt)=(0.0761243,0.534589,0.826862);
rgb(268pt)=(0.0756649,0.535738,0.826632);
rgb(269pt)=(0.0752054,0.536886,0.826403);
rgb(270pt)=(0.0747459,0.538035,0.826173);
rgb(271pt)=(0.0742168,0.539219,0.825961);
rgb(272pt)=(0.0736553,0.540418,0.825756);
rgb(273pt)=(0.0730937,0.541618,0.825552);
rgb(274pt)=(0.0725321,0.542818,0.825348);
rgb(275pt)=(0.0718925,0.544037,0.825183);
rgb(276pt)=(0.0712288,0.545262,0.82503);
rgb(277pt)=(0.0705652,0.546487,0.824877);
rgb(278pt)=(0.0699015,0.547713,0.824723);
rgb(279pt)=(0.0691514,0.548938,0.824614);
rgb(280pt)=(0.0683856,0.550163,0.824511);
rgb(281pt)=(0.0676198,0.551388,0.824409);
rgb(282pt)=(0.0668541,0.552614,0.824307);
rgb(283pt)=(0.0660408,0.553886,0.824205);
rgb(284pt)=(0.065224,0.555162,0.824103);
rgb(285pt)=(0.0644072,0.556439,0.824001);
rgb(286pt)=(0.0635892,0.557715,0.823899);
rgb(287pt)=(0.0626703,0.558991,0.823848);
rgb(288pt)=(0.0617514,0.560268,0.823797);
rgb(289pt)=(0.0608324,0.561544,0.823746);
rgb(290pt)=(0.0599087,0.56282,0.823693);
rgb(291pt)=(0.0589387,0.564096,0.823616);
rgb(292pt)=(0.0579688,0.565373,0.82354);
rgb(293pt)=(0.0569988,0.566649,0.823463);
rgb(294pt)=(0.0560243,0.567925,0.823386);
rgb(295pt)=(0.0550288,0.569202,0.82331);
rgb(296pt)=(0.0540333,0.570478,0.823233);
rgb(297pt)=(0.0530378,0.571754,0.823157);
rgb(298pt)=(0.0520423,0.57303,0.82308);
rgb(299pt)=(0.0510468,0.574307,0.823004);
rgb(300pt)=(0.0500514,0.575583,0.822927);
rgb(301pt)=(0.0490559,0.576859,0.82285);
rgb(302pt)=(0.0480604,0.578127,0.822756);
rgb(303pt)=(0.0470649,0.579377,0.822629);
rgb(304pt)=(0.0460694,0.580628,0.822501);
rgb(305pt)=(0.0450739,0.581879,0.822374);
rgb(306pt)=(0.0441,0.583119,0.822235);
rgb(307pt)=(0.0431556,0.584344,0.822082);
rgb(308pt)=(0.0422111,0.585569,0.821929);
rgb(309pt)=(0.0412667,0.586795,0.821776);
rgb(310pt)=(0.0403351,0.58802,0.821597);
rgb(311pt)=(0.0394162,0.589245,0.821392);
rgb(312pt)=(0.0384973,0.59047,0.821188);
rgb(313pt)=(0.0375784,0.591695,0.820984);
rgb(314pt)=(0.0367495,0.592891,0.820735);
rgb(315pt)=(0.0359838,0.594065,0.820454);
rgb(316pt)=(0.035218,0.595239,0.820173);
rgb(317pt)=(0.0344523,0.596413,0.819892);
rgb(318pt)=(0.0337721,0.597553,0.819595);
rgb(319pt)=(0.0331339,0.598676,0.819288);
rgb(320pt)=(0.0324958,0.599799,0.818982);
rgb(321pt)=(0.0318577,0.600923,0.818676);
rgb(322pt)=(0.0312964,0.602026,0.818312);
rgb(323pt)=(0.0307604,0.603124,0.817929);
rgb(324pt)=(0.0302243,0.604222,0.817546);
rgb(325pt)=(0.0296883,0.605319,0.817163);
rgb(326pt)=(0.0292375,0.606395,0.816738);
rgb(327pt)=(0.0288036,0.607468,0.816304);
rgb(328pt)=(0.0283697,0.60854,0.81587);
rgb(329pt)=(0.0279357,0.609612,0.815436);
rgb(330pt)=(0.0275721,0.610637,0.814955);
rgb(331pt)=(0.0272147,0.611658,0.81447);
rgb(332pt)=(0.0268574,0.612679,0.813985);
rgb(333pt)=(0.0265,0.6137,0.8135);
rgb(334pt)=(0.0262447,0.614695,0.812964);
rgb(335pt)=(0.0259895,0.615691,0.812428);
rgb(336pt)=(0.0257342,0.616686,0.811892);
rgb(337pt)=(0.0254853,0.61768,0.811352);
rgb(338pt)=(0.0253066,0.61865,0.810765);
rgb(339pt)=(0.0251279,0.61962,0.810177);
rgb(340pt)=(0.0249492,0.62059,0.80959);
rgb(341pt)=(0.024779,0.621551,0.808995);
rgb(342pt)=(0.0246514,0.62247,0.808357);
rgb(343pt)=(0.0245237,0.623389,0.807719);
rgb(344pt)=(0.0243961,0.624308,0.80708);
rgb(345pt)=(0.0242748,0.625221,0.80643);
rgb(346pt)=(0.0241727,0.626114,0.805741);
rgb(347pt)=(0.0240706,0.627008,0.805051);
rgb(348pt)=(0.0239685,0.627901,0.804362);
rgb(349pt)=(0.0238832,0.628786,0.803656);
rgb(350pt)=(0.0238321,0.629654,0.802916);
rgb(351pt)=(0.0237811,0.630522,0.802176);
rgb(352pt)=(0.02373,0.631389,0.801435);
rgb(353pt)=(0.023679,0.632247,0.800685);
rgb(354pt)=(0.0236279,0.633089,0.799919);
rgb(355pt)=(0.0235769,0.633932,0.799153);
rgb(356pt)=(0.0235258,0.634774,0.798387);
rgb(357pt)=(0.0234748,0.635604,0.797596);
rgb(358pt)=(0.0234237,0.63642,0.79678);
rgb(359pt)=(0.0233727,0.637237,0.795963);
rgb(360pt)=(0.0233216,0.638054,0.795146);
rgb(361pt)=(0.0232706,0.638856,0.794329);
rgb(362pt)=(0.0232195,0.639647,0.793512);
rgb(363pt)=(0.0231685,0.640439,0.792695);
rgb(364pt)=(0.0231174,0.64123,0.791879);
rgb(365pt)=(0.0230832,0.642005,0.791011);
rgb(366pt)=(0.0230577,0.64277,0.790118);
rgb(367pt)=(0.0230321,0.643536,0.789225);
rgb(368pt)=(0.0230066,0.644302,0.788331);
rgb(369pt)=(0.0229811,0.645049,0.787438);
rgb(370pt)=(0.0229556,0.645789,0.786544);
rgb(371pt)=(0.02293,0.646529,0.785651);
rgb(372pt)=(0.0229045,0.647269,0.784758);
rgb(373pt)=(0.022858,0.64801,0.783843);
rgb(374pt)=(0.0228069,0.64875,0.782924);
rgb(375pt)=(0.0227559,0.64949,0.782005);
rgb(376pt)=(0.0227048,0.65023,0.781086);
rgb(377pt)=(0.0227,0.650947,0.780144);
rgb(378pt)=(0.0227,0.651662,0.7792);
rgb(379pt)=(0.0227,0.652377,0.778256);
rgb(380pt)=(0.0227,0.653092,0.777311);
rgb(381pt)=(0.0228261,0.653781,0.776341);
rgb(382pt)=(0.0229538,0.65447,0.775371);
rgb(383pt)=(0.0230814,0.655159,0.774402);
rgb(384pt)=(0.0232108,0.655849,0.77343);
rgb(385pt)=(0.023364,0.656538,0.772434);
rgb(386pt)=(0.0235171,0.657227,0.771439);
rgb(387pt)=(0.0236703,0.657916,0.770443);
rgb(388pt)=(0.0238312,0.658602,0.769444);
rgb(389pt)=(0.0240354,0.659265,0.768423);
rgb(390pt)=(0.0242396,0.659929,0.767402);
rgb(391pt)=(0.0244438,0.660592,0.766381);
rgb(392pt)=(0.0247021,0.661256,0.765354);
rgb(393pt)=(0.025136,0.66192,0.764307);
rgb(394pt)=(0.02557,0.662583,0.763261);
rgb(395pt)=(0.0260039,0.663247,0.762214);
rgb(396pt)=(0.0264541,0.663911,0.761168);
rgb(397pt)=(0.026939,0.664574,0.760121);
rgb(398pt)=(0.027424,0.665238,0.759074);
rgb(399pt)=(0.027909,0.665902,0.758028);
rgb(400pt)=(0.028445,0.666555,0.756971);
rgb(401pt)=(0.0290577,0.667193,0.755899);
rgb(402pt)=(0.0296703,0.667832,0.754827);
rgb(403pt)=(0.0302829,0.66847,0.753755);
rgb(404pt)=(0.030994,0.669095,0.752683);
rgb(405pt)=(0.0318108,0.669708,0.751611);
rgb(406pt)=(0.0326276,0.670321,0.750539);
rgb(407pt)=(0.0334444,0.670933,0.749467);
rgb(408pt)=(0.0343045,0.67156,0.748366);
rgb(409pt)=(0.0351979,0.672198,0.747243);
rgb(410pt)=(0.0360913,0.672837,0.74612);
rgb(411pt)=(0.0369847,0.673475,0.744996);
rgb(412pt)=(0.0380432,0.674096,0.743873);
rgb(413pt)=(0.0391919,0.674709,0.74275);
rgb(414pt)=(0.0403405,0.675322,0.741627);
rgb(415pt)=(0.0414892,0.675934,0.740504);
rgb(416pt)=(0.0427123,0.676528,0.739381);
rgb(417pt)=(0.0439631,0.677115,0.738258);
rgb(418pt)=(0.0452138,0.677702,0.737135);
rgb(419pt)=(0.0464646,0.678289,0.736011);
rgb(420pt)=(0.0477153,0.678897,0.734868);
rgb(421pt)=(0.0489661,0.67951,0.733719);
rgb(422pt)=(0.0502168,0.680123,0.73257);
rgb(423pt)=(0.0514676,0.680735,0.731422);
rgb(424pt)=(0.0529237,0.681325,0.73025);
rgb(425pt)=(0.0544042,0.681912,0.729076);
rgb(426pt)=(0.0558847,0.682499,0.727902);
rgb(427pt)=(0.0573652,0.683086,0.726728);
rgb(428pt)=(0.0587709,0.683673,0.725553);
rgb(429pt)=(0.0601748,0.68426,0.724379);
rgb(430pt)=(0.0615787,0.684847,0.723205);
rgb(431pt)=(0.0629946,0.685435,0.722028);
rgb(432pt)=(0.0646027,0.686022,0.720803);
rgb(433pt)=(0.0662108,0.686609,0.719577);
rgb(434pt)=(0.0678189,0.687196,0.718352);
rgb(435pt)=(0.069427,0.687779,0.717131);
rgb(436pt)=(0.0710351,0.688341,0.715931);
rgb(437pt)=(0.0726432,0.688902,0.714731);
rgb(438pt)=(0.0742514,0.689464,0.713532);
rgb(439pt)=(0.0758709,0.690026,0.712326);
rgb(440pt)=(0.07753,0.690587,0.711101);
rgb(441pt)=(0.0791892,0.691149,0.709876);
rgb(442pt)=(0.0808483,0.69171,0.70865);
rgb(443pt)=(0.0825387,0.692272,0.707417);
rgb(444pt)=(0.0843,0.692833,0.706167);
rgb(445pt)=(0.0860613,0.693395,0.704916);
rgb(446pt)=(0.0878225,0.693956,0.703665);
rgb(447pt)=(0.089564,0.694518,0.702405);
rgb(448pt)=(0.0912742,0.69508,0.701128);
rgb(449pt)=(0.0929844,0.695641,0.699852);
rgb(450pt)=(0.0946946,0.696203,0.698576);
rgb(451pt)=(0.0965009,0.696752,0.697299);
rgb(452pt)=(0.0984153,0.697288,0.696023);
rgb(453pt)=(0.10033,0.697824,0.694747);
rgb(454pt)=(0.102244,0.69836,0.693471);
rgb(455pt)=(0.10413,0.698896,0.69218);
rgb(456pt)=(0.105994,0.699432,0.690878);
rgb(457pt)=(0.107857,0.699968,0.689577);
rgb(458pt)=(0.10972,0.700505,0.688275);
rgb(459pt)=(0.111632,0.701041,0.686973);
rgb(460pt)=(0.113572,0.701577,0.685671);
rgb(461pt)=(0.115512,0.702113,0.684369);
rgb(462pt)=(0.117452,0.702649,0.683068);
rgb(463pt)=(0.119429,0.703185,0.681747);
rgb(464pt)=(0.12142,0.703721,0.68042);
rgb(465pt)=(0.123411,0.704257,0.679093);
rgb(466pt)=(0.125402,0.704793,0.677765);
rgb(467pt)=(0.127372,0.705308,0.676438);
rgb(468pt)=(0.129338,0.705819,0.675111);
rgb(469pt)=(0.131303,0.706329,0.673783);
rgb(470pt)=(0.133269,0.70684,0.672456);
rgb(471pt)=(0.135369,0.70735,0.671084);
rgb(472pt)=(0.137488,0.707861,0.669705);
rgb(473pt)=(0.139607,0.708371,0.668327);
rgb(474pt)=(0.141725,0.708882,0.666949);
rgb(475pt)=(0.143795,0.709392,0.665595);
rgb(476pt)=(0.145862,0.709903,0.664242);
rgb(477pt)=(0.14793,0.710414,0.662889);
rgb(478pt)=(0.150003,0.710924,0.661534);
rgb(479pt)=(0.152198,0.711435,0.66013);
rgb(480pt)=(0.154394,0.711945,0.658726);
rgb(481pt)=(0.156589,0.712456,0.657322);
rgb(482pt)=(0.158784,0.712963,0.655922);
rgb(483pt)=(0.160979,0.713448,0.654543);
rgb(484pt)=(0.163174,0.713933,0.653165);
rgb(485pt)=(0.16537,0.714418,0.651786);
rgb(486pt)=(0.16757,0.714908,0.650397);
rgb(487pt)=(0.169791,0.715419,0.648968);
rgb(488pt)=(0.172012,0.715929,0.647538);
rgb(489pt)=(0.174232,0.71644,0.646109);
rgb(490pt)=(0.176483,0.716935,0.64468);
rgb(491pt)=(0.178806,0.717395,0.64325);
rgb(492pt)=(0.181129,0.717854,0.641821);
rgb(493pt)=(0.183452,0.718314,0.640391);
rgb(494pt)=(0.185755,0.718783,0.638952);
rgb(495pt)=(0.188027,0.719268,0.637497);
rgb(496pt)=(0.190299,0.719753,0.636042);
rgb(497pt)=(0.192571,0.720238,0.634587);
rgb(498pt)=(0.194913,0.720711,0.633132);
rgb(499pt)=(0.197338,0.72117,0.631677);
rgb(500pt)=(0.199762,0.72163,0.630223);
rgb(501pt)=(0.202187,0.722089,0.628768);
rgb(502pt)=(0.204612,0.722549,0.627299);
rgb(503pt)=(0.207037,0.723008,0.625818);
rgb(504pt)=(0.209462,0.723468,0.624338);
rgb(505pt)=(0.211887,0.723927,0.622857);
rgb(506pt)=(0.214328,0.724386,0.621377);
rgb(507pt)=(0.216778,0.724846,0.619896);
rgb(508pt)=(0.219229,0.725305,0.618416);
rgb(509pt)=(0.221679,0.725765,0.616935);
rgb(510pt)=(0.224202,0.726188,0.615455);
rgb(511pt)=(0.226754,0.726597,0.613974);
rgb(512pt)=(0.229307,0.727005,0.612494);
rgb(513pt)=(0.231859,0.727414,0.611014);
rgb(514pt)=(0.234392,0.727842,0.609513);
rgb(515pt)=(0.236919,0.728276,0.608007);
rgb(516pt)=(0.239446,0.72871,0.606501);
rgb(517pt)=(0.241973,0.729144,0.604995);
rgb(518pt)=(0.244611,0.729556,0.603467);
rgb(519pt)=(0.247266,0.729964,0.601935);
rgb(520pt)=(0.24992,0.730372,0.600404);
rgb(521pt)=(0.252575,0.730781,0.598872);
rgb(522pt)=(0.25523,0.731189,0.597365);
rgb(523pt)=(0.257884,0.731598,0.595859);
rgb(524pt)=(0.260539,0.732006,0.594353);
rgb(525pt)=(0.263194,0.732414,0.592846);
rgb(526pt)=(0.265848,0.732796,0.591314);
rgb(527pt)=(0.268503,0.733179,0.589783);
rgb(528pt)=(0.271158,0.733562,0.588251);
rgb(529pt)=(0.27383,0.733945,0.58672);
rgb(530pt)=(0.276638,0.734328,0.585188);
rgb(531pt)=(0.279446,0.734711,0.583657);
rgb(532pt)=(0.282254,0.735094,0.582125);
rgb(533pt)=(0.285051,0.735471,0.580594);
rgb(534pt)=(0.287808,0.735829,0.579062);
rgb(535pt)=(0.290565,0.736186,0.577531);
rgb(536pt)=(0.293322,0.736544,0.575999);
rgb(537pt)=(0.2961,0.736894,0.574468);
rgb(538pt)=(0.298933,0.737226,0.572936);
rgb(539pt)=(0.301767,0.737557,0.571405);
rgb(540pt)=(0.3046,0.737889,0.569873);
rgb(541pt)=(0.307452,0.738221,0.568351);
rgb(542pt)=(0.310336,0.738553,0.566845);
rgb(543pt)=(0.313221,0.738885,0.565339);
rgb(544pt)=(0.316105,0.739217,0.563833);
rgb(545pt)=(0.318978,0.739537,0.562315);
rgb(546pt)=(0.321837,0.739843,0.560784);
rgb(547pt)=(0.324696,0.74015,0.559252);
rgb(548pt)=(0.327555,0.740456,0.557721);
rgb(549pt)=(0.330468,0.740749,0.556216);
rgb(550pt)=(0.333429,0.741029,0.554736);
rgb(551pt)=(0.336389,0.74131,0.553255);
rgb(552pt)=(0.33935,0.741591,0.551775);
rgb(553pt)=(0.342296,0.741872,0.550279);
rgb(554pt)=(0.345231,0.742153,0.548773);
rgb(555pt)=(0.348167,0.742433,0.547267);
rgb(556pt)=(0.351102,0.742714,0.545761);
rgb(557pt)=(0.354038,0.742977,0.54429);
rgb(558pt)=(0.356973,0.743232,0.542835);
rgb(559pt)=(0.359908,0.743488,0.54138);
rgb(560pt)=(0.362844,0.743743,0.539925);
rgb(561pt)=(0.365839,0.743959,0.53847);
rgb(562pt)=(0.368851,0.744163,0.537015);
rgb(563pt)=(0.371863,0.744367,0.53556);
rgb(564pt)=(0.374875,0.744571,0.534105);
rgb(565pt)=(0.377843,0.744775,0.532672);
rgb(566pt)=(0.380804,0.74498,0.531243);
rgb(567pt)=(0.383765,0.745184,0.529814);
rgb(568pt)=(0.386726,0.745388,0.528384);
rgb(569pt)=(0.389711,0.745568,0.527003);
rgb(570pt)=(0.392697,0.745747,0.525624);
rgb(571pt)=(0.395684,0.745926,0.524246);
rgb(572pt)=(0.39867,0.746104,0.522868);
rgb(573pt)=(0.401657,0.746257,0.521489);
rgb(574pt)=(0.404643,0.74641,0.520111);
rgb(575pt)=(0.40763,0.746563,0.518732);
rgb(576pt)=(0.410611,0.746716,0.517359);
rgb(577pt)=(0.413546,0.746869,0.516032);
rgb(578pt)=(0.416482,0.747023,0.514705);
rgb(579pt)=(0.419417,0.747176,0.513377);
rgb(580pt)=(0.422357,0.747319,0.512055);
rgb(581pt)=(0.425318,0.747421,0.510753);
rgb(582pt)=(0.428279,0.747523,0.509451);
rgb(583pt)=(0.43124,0.747626,0.50815);
rgb(584pt)=(0.43418,0.747735,0.506848);
rgb(585pt)=(0.437065,0.747862,0.505546);
rgb(586pt)=(0.439949,0.74799,0.504244);
rgb(587pt)=(0.442834,0.748117,0.502942);
rgb(588pt)=(0.445727,0.748227,0.501659);
rgb(589pt)=(0.448637,0.748304,0.500408);
rgb(590pt)=(0.451547,0.74838,0.499157);
rgb(591pt)=(0.454457,0.748457,0.497906);
rgb(592pt)=(0.457333,0.748522,0.496667);
rgb(593pt)=(0.460167,0.748573,0.495441);
rgb(594pt)=(0.463,0.748624,0.494216);
rgb(595pt)=(0.465833,0.748675,0.492991);
rgb(596pt)=(0.468667,0.748726,0.491779);
rgb(597pt)=(0.4715,0.748777,0.490579);
rgb(598pt)=(0.474333,0.748829,0.48938);
rgb(599pt)=(0.477167,0.74888,0.48818);
rgb(600pt)=(0.479969,0.748931,0.486995);
rgb(601pt)=(0.482752,0.748982,0.485821);
rgb(602pt)=(0.485534,0.749033,0.484647);
rgb(603pt)=(0.488316,0.749084,0.483473);
rgb(604pt)=(0.491081,0.7491,0.482316);
rgb(605pt)=(0.493838,0.7491,0.481168);
rgb(606pt)=(0.496595,0.7491,0.480019);
rgb(607pt)=(0.499351,0.7491,0.47887);
rgb(608pt)=(0.502069,0.74912,0.477722);
rgb(609pt)=(0.504775,0.749145,0.476573);
rgb(610pt)=(0.50748,0.749171,0.475424);
rgb(611pt)=(0.510186,0.749196,0.474276);
rgb(612pt)=(0.512892,0.7492,0.47317);
rgb(613pt)=(0.515598,0.7492,0.472073);
rgb(614pt)=(0.518303,0.7492,0.470975);
rgb(615pt)=(0.521009,0.7492,0.469877);
rgb(616pt)=(0.523644,0.749176,0.46878);
rgb(617pt)=(0.526273,0.749151,0.467682);
rgb(618pt)=(0.528902,0.749125,0.466585);
rgb(619pt)=(0.531531,0.7491,0.465487);
rgb(620pt)=(0.53416,0.749074,0.464415);
rgb(621pt)=(0.536789,0.749049,0.463343);
rgb(622pt)=(0.539418,0.749023,0.462271);
rgb(623pt)=(0.542043,0.748998,0.461199);
rgb(624pt)=(0.544621,0.748972,0.460127);
rgb(625pt)=(0.547199,0.748947,0.459055);
rgb(626pt)=(0.549777,0.748921,0.457983);
rgb(627pt)=(0.55235,0.748891,0.45692);
rgb(628pt)=(0.554903,0.74884,0.455899);
rgb(629pt)=(0.557456,0.748789,0.454878);
rgb(630pt)=(0.560008,0.748738,0.453857);
rgb(631pt)=(0.562554,0.748687,0.452829);
rgb(632pt)=(0.565081,0.748636,0.451783);
rgb(633pt)=(0.567608,0.748585,0.450736);
rgb(634pt)=(0.570135,0.748534,0.449689);
rgb(635pt)=(0.572653,0.748474,0.44866);
rgb(636pt)=(0.575155,0.748397,0.447665);
rgb(637pt)=(0.577656,0.748321,0.446669);
rgb(638pt)=(0.580158,0.748244,0.445674);
rgb(639pt)=(0.582649,0.748168,0.444678);
rgb(640pt)=(0.585125,0.748091,0.443683);
rgb(641pt)=(0.587601,0.748014,0.442687);
rgb(642pt)=(0.590077,0.747938,0.441692);
rgb(643pt)=(0.59254,0.747861,0.440709);
rgb(644pt)=(0.59499,0.747785,0.439739);
rgb(645pt)=(0.597441,0.747708,0.438769);
rgb(646pt)=(0.599891,0.747632,0.437799);
rgb(647pt)=(0.602311,0.747555,0.436814);
rgb(648pt)=(0.604711,0.747478,0.435819);
rgb(649pt)=(0.60711,0.747402,0.434823);
rgb(650pt)=(0.60951,0.747325,0.433828);
rgb(651pt)=(0.611909,0.747232,0.432867);
rgb(652pt)=(0.614308,0.747129,0.431922);
rgb(653pt)=(0.616708,0.747027,0.430978);
rgb(654pt)=(0.619107,0.746925,0.430033);
rgb(655pt)=(0.621487,0.746823,0.429089);
rgb(656pt)=(0.623861,0.746721,0.428144);
rgb(657pt)=(0.626235,0.746619,0.4272);
rgb(658pt)=(0.628609,0.746517,0.426256);
rgb(659pt)=(0.630962,0.746393,0.425311);
rgb(660pt)=(0.63331,0.746266,0.424367);
rgb(661pt)=(0.635658,0.746138,0.423422);
rgb(662pt)=(0.638007,0.746011,0.422478);
rgb(663pt)=(0.640332,0.745906,0.421557);
rgb(664pt)=(0.642654,0.745804,0.420638);
rgb(665pt)=(0.644977,0.745702,0.419719);
rgb(666pt)=(0.6473,0.7456,0.4188);
rgb(667pt)=(0.649623,0.745472,0.417881);
rgb(668pt)=(0.651946,0.745345,0.416962);
rgb(669pt)=(0.654268,0.745217,0.416043);
rgb(670pt)=(0.656587,0.745089,0.415124);
rgb(671pt)=(0.658859,0.744962,0.414205);
rgb(672pt)=(0.661131,0.744834,0.413286);
rgb(673pt)=(0.663402,0.744707,0.412368);
rgb(674pt)=(0.665674,0.744579,0.411453);
rgb(675pt)=(0.667946,0.744451,0.410559);
rgb(676pt)=(0.670218,0.744324,0.409666);
rgb(677pt)=(0.672489,0.744196,0.408773);
rgb(678pt)=(0.674755,0.744062,0.407879);
rgb(679pt)=(0.677001,0.743909,0.406986);
rgb(680pt)=(0.679247,0.743756,0.406092);
rgb(681pt)=(0.681494,0.743603,0.405199);
rgb(682pt)=(0.68374,0.743458,0.404306);
rgb(683pt)=(0.685986,0.74333,0.403412);
rgb(684pt)=(0.688232,0.743203,0.402519);
rgb(685pt)=(0.690479,0.743075,0.401626);
rgb(686pt)=(0.692704,0.742937,0.400732);
rgb(687pt)=(0.694899,0.742784,0.399839);
rgb(688pt)=(0.697094,0.742631,0.398945);
rgb(689pt)=(0.699289,0.742477,0.398052);
rgb(690pt)=(0.701497,0.742324,0.397171);
rgb(691pt)=(0.703718,0.742171,0.396303);
rgb(692pt)=(0.705939,0.742018,0.395435);
rgb(693pt)=(0.708159,0.741865,0.394568);
rgb(694pt)=(0.710351,0.741712,0.3937);
rgb(695pt)=(0.71252,0.741559,0.392832);
rgb(696pt)=(0.71469,0.741405,0.391964);
rgb(697pt)=(0.71686,0.741252,0.391096);
rgb(698pt)=(0.719029,0.741082,0.390228);
rgb(699pt)=(0.721199,0.740904,0.38936);
rgb(700pt)=(0.723369,0.740725,0.388492);
rgb(701pt)=(0.725538,0.740546,0.387625);
rgb(702pt)=(0.727708,0.740386,0.386757);
rgb(703pt)=(0.729878,0.740233,0.385889);
rgb(704pt)=(0.732047,0.74008,0.385021);
rgb(705pt)=(0.734217,0.739927,0.384153);
rgb(706pt)=(0.736366,0.739753,0.383285);
rgb(707pt)=(0.73851,0.739574,0.382417);
rgb(708pt)=(0.740654,0.739395,0.38155);
rgb(709pt)=(0.742798,0.739217,0.380682);
rgb(710pt)=(0.744919,0.739038,0.379837);
rgb(711pt)=(0.747038,0.738859,0.378995);
rgb(712pt)=(0.749156,0.738681,0.378152);
rgb(713pt)=(0.751275,0.738502,0.37731);
rgb(714pt)=(0.753394,0.738323,0.376442);
rgb(715pt)=(0.755512,0.738145,0.375574);
rgb(716pt)=(0.757631,0.737966,0.374707);
rgb(717pt)=(0.75975,0.737789,0.373841);
rgb(718pt)=(0.761868,0.737636,0.372998);
rgb(719pt)=(0.763987,0.737483,0.372156);
rgb(720pt)=(0.766105,0.73733,0.371314);
rgb(721pt)=(0.76822,0.737169,0.370471);
rgb(722pt)=(0.770313,0.736965,0.369629);
rgb(723pt)=(0.772406,0.73676,0.368786);
rgb(724pt)=(0.774499,0.736556,0.367944);
rgb(725pt)=(0.776592,0.736358,0.367096);
rgb(726pt)=(0.778686,0.736179,0.366228);
rgb(727pt)=(0.780779,0.736001,0.36536);
rgb(728pt)=(0.782872,0.735822,0.364492);
rgb(729pt)=(0.784957,0.735643,0.363632);
rgb(730pt)=(0.787024,0.735465,0.36279);
rgb(731pt)=(0.789092,0.735286,0.361948);
rgb(732pt)=(0.791159,0.735107,0.361105);
rgb(733pt)=(0.793227,0.734929,0.360263);
rgb(734pt)=(0.795295,0.73475,0.359421);
rgb(735pt)=(0.797362,0.734571,0.358578);
rgb(736pt)=(0.79943,0.734392,0.357736);
rgb(737pt)=(0.801485,0.734214,0.356881);
rgb(738pt)=(0.803527,0.734035,0.356014);
rgb(739pt)=(0.805569,0.733856,0.355146);
rgb(740pt)=(0.807611,0.733678,0.354278);
rgb(741pt)=(0.809668,0.733499,0.353424);
rgb(742pt)=(0.811735,0.73332,0.352582);
rgb(743pt)=(0.813803,0.733142,0.35174);
rgb(744pt)=(0.81587,0.732963,0.350897);
rgb(745pt)=(0.817921,0.732784,0.350038);
rgb(746pt)=(0.819963,0.732606,0.349171);
rgb(747pt)=(0.822005,0.732427,0.348303);
rgb(748pt)=(0.824047,0.732248,0.347435);
rgb(749pt)=(0.826071,0.73207,0.346567);
rgb(750pt)=(0.828087,0.731891,0.345699);
rgb(751pt)=(0.830104,0.731712,0.344831);
rgb(752pt)=(0.83212,0.731534,0.343963);
rgb(753pt)=(0.834158,0.731355,0.343095);
rgb(754pt)=(0.8362,0.731176,0.342228);
rgb(755pt)=(0.838242,0.730998,0.34136);
rgb(756pt)=(0.840284,0.730819,0.340492);
rgb(757pt)=(0.842303,0.73064,0.339624);
rgb(758pt)=(0.84432,0.730462,0.338756);
rgb(759pt)=(0.846336,0.730283,0.337888);
rgb(760pt)=(0.848353,0.730104,0.33702);
rgb(761pt)=(0.850369,0.729926,0.336153);
rgb(762pt)=(0.852386,0.729747,0.335285);
rgb(763pt)=(0.854402,0.729568,0.334417);
rgb(764pt)=(0.856419,0.729391,0.333546);
rgb(765pt)=(0.858435,0.729238,0.332627);
rgb(766pt)=(0.860452,0.729085,0.331708);
rgb(767pt)=(0.862468,0.728932,0.330789);
rgb(768pt)=(0.864481,0.728778,0.329874);
rgb(769pt)=(0.866472,0.728625,0.32898);
rgb(770pt)=(0.868463,0.728472,0.328087);
rgb(771pt)=(0.870454,0.728319,0.327194);
rgb(772pt)=(0.872445,0.728166,0.326295);
rgb(773pt)=(0.874436,0.728013,0.325376);
rgb(774pt)=(0.876427,0.727859,0.324457);
rgb(775pt)=(0.878418,0.727706,0.323538);
rgb(776pt)=(0.880417,0.727561,0.322619);
rgb(777pt)=(0.882433,0.727433,0.3217);
rgb(778pt)=(0.88445,0.727306,0.320781);
rgb(779pt)=(0.886466,0.727178,0.319862);
rgb(780pt)=(0.888463,0.72705,0.318933);
rgb(781pt)=(0.890429,0.726923,0.317989);
rgb(782pt)=(0.892394,0.726795,0.317044);
rgb(783pt)=(0.894359,0.726668,0.3161);
rgb(784pt)=(0.896337,0.726552,0.315132);
rgb(785pt)=(0.898328,0.72645,0.314136);
rgb(786pt)=(0.900319,0.726348,0.313141);
rgb(787pt)=(0.90231,0.726246,0.312145);
rgb(788pt)=(0.904301,0.726158,0.31115);
rgb(789pt)=(0.906292,0.726081,0.310154);
rgb(790pt)=(0.908283,0.726005,0.309159);
rgb(791pt)=(0.910274,0.725928,0.308163);
rgb(792pt)=(0.912249,0.725851,0.307151);
rgb(793pt)=(0.914214,0.725775,0.30613);
rgb(794pt)=(0.91618,0.725698,0.305109);
rgb(795pt)=(0.918145,0.725622,0.304088);
rgb(796pt)=(0.920111,0.7256,0.303031);
rgb(797pt)=(0.922076,0.7256,0.301959);
rgb(798pt)=(0.924041,0.7256,0.300886);
rgb(799pt)=(0.926007,0.7256,0.299814);
rgb(800pt)=(0.927972,0.7256,0.298722);
rgb(801pt)=(0.929938,0.7256,0.297624);
rgb(802pt)=(0.931903,0.7256,0.296527);
rgb(803pt)=(0.933869,0.7256,0.295429);
rgb(804pt)=(0.935812,0.725668,0.294264);
rgb(805pt)=(0.937752,0.725744,0.29309);
rgb(806pt)=(0.939692,0.725821,0.291916);
rgb(807pt)=(0.941632,0.725897,0.290741);
rgb(808pt)=(0.943571,0.726023,0.289518);
rgb(809pt)=(0.945511,0.726151,0.288293);
rgb(810pt)=(0.947451,0.726278,0.287068);
rgb(811pt)=(0.949389,0.726411,0.285839);
rgb(812pt)=(0.951278,0.726641,0.284537);
rgb(813pt)=(0.953167,0.72687,0.283235);
rgb(814pt)=(0.955056,0.7271,0.281933);
rgb(815pt)=(0.956938,0.72734,0.280622);
rgb(816pt)=(0.958776,0.727646,0.279243);
rgb(817pt)=(0.960614,0.727952,0.277865);
rgb(818pt)=(0.962451,0.728259,0.276486);
rgb(819pt)=(0.964273,0.728597,0.275086);
rgb(820pt)=(0.966034,0.729057,0.273606);
rgb(821pt)=(0.967795,0.729516,0.272126);
rgb(822pt)=(0.969557,0.729976,0.270645);
rgb(823pt)=(0.971288,0.730473,0.269135);
rgb(824pt)=(0.972947,0.73106,0.267552);
rgb(825pt)=(0.974606,0.731647,0.265969);
rgb(826pt)=(0.976265,0.732234,0.264387);
rgb(827pt)=(0.977857,0.732879,0.262785);
rgb(828pt)=(0.979338,0.733619,0.261151);
rgb(829pt)=(0.980818,0.734359,0.259518);
rgb(830pt)=(0.982299,0.735099,0.257884);
rgb(831pt)=(0.983697,0.73591,0.256227);
rgb(832pt)=(0.984999,0.736803,0.254542);
rgb(833pt)=(0.986301,0.737697,0.252858);
rgb(834pt)=(0.987603,0.73859,0.251173);
rgb(835pt)=(0.988753,0.739566,0.249474);
rgb(836pt)=(0.989774,0.740613,0.247764);
rgb(837pt)=(0.990795,0.741659,0.246054);
rgb(838pt)=(0.991816,0.742706,0.244344);
rgb(839pt)=(0.992677,0.743816,0.242681);
rgb(840pt)=(0.993443,0.744965,0.241048);
rgb(841pt)=(0.994209,0.746114,0.239414);
rgb(842pt)=(0.994975,0.747262,0.23778);
rgb(843pt)=(0.995578,0.748465,0.236165);
rgb(844pt)=(0.996114,0.74969,0.234557);
rgb(845pt)=(0.99665,0.750915,0.232949);
rgb(846pt)=(0.997186,0.752141,0.231341);
rgb(847pt)=(0.997562,0.753386,0.229813);
rgb(848pt)=(0.997893,0.754637,0.228307);
rgb(849pt)=(0.998225,0.755887,0.226801);
rgb(850pt)=(0.998557,0.757138,0.225295);
rgb(851pt)=(0.998711,0.758433,0.223856);
rgb(852pt)=(0.998839,0.759735,0.222426);
rgb(853pt)=(0.998966,0.761037,0.220997);
rgb(854pt)=(0.999094,0.762339,0.219567);
rgb(855pt)=(0.999076,0.763641,0.218186);
rgb(856pt)=(0.99905,0.764942,0.216808);
rgb(857pt)=(0.999025,0.766244,0.21543);
rgb(858pt)=(0.998995,0.767546,0.214054);
rgb(859pt)=(0.998868,0.768848,0.212752);
rgb(860pt)=(0.99874,0.77015,0.21145);
rgb(861pt)=(0.998613,0.771451,0.210149);
rgb(862pt)=(0.998473,0.772756,0.208856);
rgb(863pt)=(0.998243,0.774083,0.207631);
rgb(864pt)=(0.998014,0.775411,0.206405);
rgb(865pt)=(0.997784,0.776738,0.20518);
rgb(866pt)=(0.997539,0.77806,0.20396);
rgb(867pt)=(0.997232,0.779362,0.20276);
rgb(868pt)=(0.996926,0.780664,0.201561);
rgb(869pt)=(0.99662,0.781966,0.200361);
rgb(870pt)=(0.996299,0.783268,0.199168);
rgb(871pt)=(0.995942,0.784569,0.197994);
rgb(872pt)=(0.995584,0.785871,0.19682);
rgb(873pt)=(0.995227,0.787173,0.195646);
rgb(874pt)=(0.994842,0.788475,0.19449);
rgb(875pt)=(0.994408,0.789777,0.193367);
rgb(876pt)=(0.993974,0.791078,0.192244);
rgb(877pt)=(0.99354,0.79238,0.191121);
rgb(878pt)=(0.993083,0.793671,0.190021);
rgb(879pt)=(0.992598,0.794947,0.188949);
rgb(880pt)=(0.992113,0.796223,0.187877);
rgb(881pt)=(0.991628,0.797499,0.186805);
rgb(882pt)=(0.99113,0.798789,0.185732);
rgb(883pt)=(0.990619,0.800091,0.18466);
rgb(884pt)=(0.990109,0.801393,0.183588);
rgb(885pt)=(0.989598,0.802695,0.182516);
rgb(886pt)=(0.989072,0.803996,0.18146);
rgb(887pt)=(0.988536,0.805298,0.180413);
rgb(888pt)=(0.988,0.8066,0.179367);
rgb(889pt)=(0.987464,0.807902,0.17832);
rgb(890pt)=(0.98691,0.809186,0.177291);
rgb(891pt)=(0.986349,0.810462,0.17627);
rgb(892pt)=(0.985787,0.811738,0.175249);
rgb(893pt)=(0.985226,0.813015,0.174228);
rgb(894pt)=(0.984644,0.814311,0.173207);
rgb(895pt)=(0.984057,0.815613,0.172186);
rgb(896pt)=(0.98347,0.816914,0.171165);
rgb(897pt)=(0.982883,0.818216,0.170144);
rgb(898pt)=(0.982296,0.819518,0.169145);
rgb(899pt)=(0.981709,0.82082,0.16815);
rgb(900pt)=(0.981122,0.822122,0.167154);
rgb(901pt)=(0.980535,0.823423,0.166159);
rgb(902pt)=(0.979947,0.824725,0.165163);
rgb(903pt)=(0.97936,0.826027,0.164168);
rgb(904pt)=(0.978773,0.827329,0.163172);
rgb(905pt)=(0.978186,0.828631,0.162177);
rgb(906pt)=(0.977599,0.829932,0.161181);
rgb(907pt)=(0.977012,0.831234,0.160186);
rgb(908pt)=(0.976425,0.832536,0.15919);
rgb(909pt)=(0.975838,0.833841,0.158195);
rgb(910pt)=(0.975251,0.835168,0.157199);
rgb(911pt)=(0.974664,0.836495,0.156204);
rgb(912pt)=(0.974077,0.837823,0.155208);
rgb(913pt)=(0.973489,0.83915,0.154213);
rgb(914pt)=(0.972902,0.840477,0.153217);
rgb(915pt)=(0.972315,0.841805,0.152222);
rgb(916pt)=(0.971728,0.843132,0.151226);
rgb(917pt)=(0.971155,0.844466,0.150224);
rgb(918pt)=(0.970619,0.845819,0.149203);
rgb(919pt)=(0.970083,0.847172,0.148182);
rgb(920pt)=(0.969547,0.848525,0.147161);
rgb(921pt)=(0.96902,0.849886,0.14614);
rgb(922pt)=(0.968509,0.851265,0.145119);
rgb(923pt)=(0.967999,0.852643,0.144098);
rgb(924pt)=(0.967488,0.854022,0.143077);
rgb(925pt)=(0.967,0.855411,0.142056);
rgb(926pt)=(0.966541,0.856815,0.141035);
rgb(927pt)=(0.966081,0.858219,0.140014);
rgb(928pt)=(0.965622,0.859623,0.138992);
rgb(929pt)=(0.965189,0.86104,0.137945);
rgb(930pt)=(0.96478,0.862469,0.136873);
rgb(931pt)=(0.964372,0.863899,0.135801);
rgb(932pt)=(0.963963,0.865328,0.134729);
rgb(933pt)=(0.96357,0.866773,0.133657);
rgb(934pt)=(0.963187,0.868228,0.132585);
rgb(935pt)=(0.962805,0.869683,0.131513);
rgb(936pt)=(0.962422,0.871138,0.130441);
rgb(937pt)=(0.962091,0.87261,0.129351);
rgb(938pt)=(0.961785,0.874091,0.128253);
rgb(939pt)=(0.961478,0.875571,0.127156);
rgb(940pt)=(0.961172,0.877052,0.126058);
rgb(941pt)=(0.960885,0.878571,0.124961);
rgb(942pt)=(0.960605,0.880103,0.123863);
rgb(943pt)=(0.960324,0.881634,0.122765);
rgb(944pt)=(0.960043,0.883166,0.121668);
rgb(945pt)=(0.959849,0.884719,0.120549);
rgb(946pt)=(0.95967,0.886276,0.119426);
rgb(947pt)=(0.959491,0.887833,0.118302);
rgb(948pt)=(0.959313,0.88939,0.117179);
rgb(949pt)=(0.959181,0.890995,0.116032);
rgb(950pt)=(0.959054,0.892603,0.114884);
rgb(951pt)=(0.958926,0.894211,0.113735);
rgb(952pt)=(0.958799,0.895819,0.112587);
rgb(953pt)=(0.958748,0.897453,0.111464);
rgb(954pt)=(0.958697,0.899086,0.110341);
rgb(955pt)=(0.958646,0.90072,0.109217);
rgb(956pt)=(0.958602,0.902359,0.108089);
rgb(957pt)=(0.958628,0.904043,0.106915);
rgb(958pt)=(0.958653,0.905728,0.105741);
rgb(959pt)=(0.958679,0.907413,0.104567);
rgb(960pt)=(0.958718,0.909102,0.103393);
rgb(961pt)=(0.95882,0.910812,0.102219);
rgb(962pt)=(0.958922,0.912522,0.101044);
rgb(963pt)=(0.959024,0.914232,0.0998703);
rgb(964pt)=(0.959153,0.915962,0.0986961);
rgb(965pt)=(0.959357,0.917749,0.0975219);
rgb(966pt)=(0.959561,0.919536,0.0963477);
rgb(967pt)=(0.959765,0.921323,0.0951736);
rgb(968pt)=(0.959996,0.923118,0.0939907);
rgb(969pt)=(0.960277,0.924931,0.092791);
rgb(970pt)=(0.960557,0.926743,0.0915913);
rgb(971pt)=(0.960838,0.928555,0.0903916);
rgb(972pt)=(0.961151,0.930378,0.0891919);
rgb(973pt)=(0.961509,0.932216,0.0879922);
rgb(974pt)=(0.961866,0.934054,0.0867925);
rgb(975pt)=(0.962223,0.935892,0.0855928);
rgb(976pt)=(0.96262,0.937768,0.0843802);
rgb(977pt)=(0.963053,0.939683,0.083155);
rgb(978pt)=(0.963487,0.941597,0.0819297);
rgb(979pt)=(0.963921,0.943512,0.0807045);
rgb(980pt)=(0.964415,0.945426,0.0794643);
rgb(981pt)=(0.964951,0.947341,0.0782135);
rgb(982pt)=(0.965487,0.949255,0.0769628);
rgb(983pt)=(0.966023,0.951169,0.075712);
rgb(984pt)=(0.966594,0.953118,0.0744441);
rgb(985pt)=(0.967181,0.955083,0.0731679);
rgb(986pt)=(0.967768,0.957049,0.0718916);
rgb(987pt)=(0.968355,0.959014,0.0706153);
rgb(988pt)=(0.96898,0.960999,0.0693006);
rgb(989pt)=(0.969619,0.96299,0.0679733);
rgb(990pt)=(0.970257,0.964981,0.0666459);
rgb(991pt)=(0.970895,0.966972,0.0653186);
rgb(992pt)=(0.971554,0.968984,0.0639486);
rgb(993pt)=(0.972218,0.971001,0.0625703);
rgb(994pt)=(0.972882,0.973017,0.0611919);
rgb(995pt)=(0.973545,0.975034,0.0598135);
rgb(996pt)=(0.974232,0.97705,0.058318);
rgb(997pt)=(0.974922,0.979067,0.056812);
rgb(998pt)=(0.975611,0.981083,0.055306);
rgb(999pt)=(0.9763,0.9831,0.0538)
}
}

\usetikzlibrary{arrows,shapes,positioning}
\usetikzlibrary{decorations.markings}
\tikzstyle arrowstyle=[scale=1]
\tikzstyle directed=[postaction={decorate,decoration={markings,
		mark=at position .65 with {\arrow[arrowstyle]{stealth}}}}]
\tikzstyle reverse directed=[postaction={decorate,decoration={markings,
		mark=at position .65 with {\arrowreversed[arrowstyle]{stealth};}}}]

\newcommand{\secref}[1]{Section~\ref{#1}}
\newcommand{\thmref}[1]{Theorem~\ref{#1}}

\newcommand{\lemref}[1]{Lemma~\ref{#1}}
\newcommand{\defnref}[1]{Definition~\ref{#1}}

\newcommand{\figref}[1]{Figure~\ref{#1}}

\newcommand{\abs}[1]{\ensuremath{\left| #1 \right|}}
\newcommand{\norm}[2]{\ensuremath{\left\Vert #1 \right\Vert_{#2}}}

\newcommand{\R}{\mathbb{R}}
\newcommand{\Rpos}{\R_{\geq 0}}
\newcommand{\scattering}{\ensuremath{\sigma_s}}
\newcommand{\absorption}{\ensuremath{\sigma_a}}

\newcommand{\source}{\ensuremath{Q}}

\newcommand{\ceil}[1]{\ensuremath{\left\lceil #1 \right\rceil}}

\newcommand{\sphere}[1][2]{\ensuremath{\mathcal{S}^{#1}}}

\newcommand{\interior}[1]{\ensuremath{\operatorname{int}\left(#1\right)}}
\newcommand{\closure}[1]{\ensuremath{\overline{#1}}}
\newcommand{\convexhull}[1]{\ensuremath{\operatorname{conv}\left(#1\right)}}

\newcommand{\volume}[1]{\ensuremath{\operatorname{vol}\left(#1\right)}}

\newcommand{\distribution}[1][ ]{\ensuremath{\psi_{#1}}}
\newcommand{\distributiontzero}{\ensuremath{\distribution[\timevar=0]}}
\newcommand{\distributionboundary}{\ensuremath{\distribution[b]}}

\newcommand{\ansatz}[1][ ]{\ensuremath{\hat{\psi}_{#1}}}

\newcommand{\momentorder}{\ensuremath{N}}
\newcommand{\momentnumber}{\ensuremath{n}}
\newcommand{\refinementnumber}{\ensuremath{r}}
\newcommand{\basis}[1][ ]{{\ensuremath{\bb_{#1}}}} 
\newcommand{\basisind}{\ensuremath{i}} 
\newcommand{\basisindvar}{\ensuremath{j}} 
\newcommand{\basiscomp}[1][\basisind]{\ensuremath{b_{#1}}} 
\newcommand{\SHl}{{\ensuremath{l}}} 

\newcommand{\mmbasis}[1][\momentorder]{\ensuremath{\bm_{#1}}} 
\newcommand{\fmbasis}[1][\momentorder]{\ensuremath{\bff_{#1}}} 
\newcommand{\pmbasis}[1][\momentnumber]{\ensuremath{\bp_{#1}}} 
\newcommand{\pmbasisIk}[1][\hankelhalfind]{\ensuremath{\bp_{\partition[#1]}}}
\newcommand{\pmbasisTh}{\ensuremath{\bp_{\TriangulationSphere}}}
\newcommand{\hfbasis}[1][\momentnumber]{\ensuremath{\bh_{#1}}} 
\newcommand{\hfbasisIk}[1][\hankelhalfind]{\ensuremath{\bh_{\partition[#1]}}}
\newcommand{\hfbasisTh}{\ensuremath{\bh_{\TriangulationSphere}}}
\newcommand{\hfbasiscomp}[1][\basisind]{\ensuremath{h_{#1}}} 
\newcommand{\moments}[1][ ]{\ensuremath{\bu_{#1}}} 

\newcommand{\momentcomp}[1]{\ensuremath{u_{#1}}} 
\newcommand{\momentsposblock}[1]{\ensuremath{\bu_{#1}^{+}}} 
\newcommand{\momentsposblockcomp}[2]{\ensuremath{\left(\momentsposblock{#1}\right)_{#2}}} 

\newcommand{\density}{\ensuremath{\rho}}

\newcommand{\isotropicmoment}{\moments[\text{iso}]}

\newcommand{\massmatrix}{\ensuremath{\bM}}
\DeclareMathOperator{\spanop}{span}
\DeclareMathOperator{\ord}{ord}

\newcommand{\normalizedmoments}[1][ ]{\ensuremath{\bsphi_{#1}}} 

\newcommand{\multipliers}[1][ ]{\ensuremath{\bsalpha_{#1}}} 
\newcommand{\multipliersone}[1][ ]{\ensuremath{\bsalpha_{\text{one}}}}
\newcommand{\multiplierscomp}[1]{\ensuremath{\alpha_{#1}}} 

\newcommand{\SC}{\ensuremath{\Omega}} 
\newcommand{\SCheight}{\ensuremath{\mu}} 
\newcommand{\SCangle}{\ensuremath{\varphi}} 
\newcommand{\Domain}{\ensuremath{X}} 
\newcommand{\outernormal}{\ensuremath{\bn}} 
\newcommand{\spatialVariable}{\ensuremath{\bx}} 
\newcommand{\timeint}{\ensuremath{T}} 
\newcommand{\tf}{\ensuremath{t_f}} 
\newcommand{\timevar}{\ensuremath{t}} 

\newcommand{\ints}[1]{\ensuremath{\left<#1\right>}}
\newcommand{\intA}[2]{\ensuremath{\left<#1\right>_{#2}}}

\newcommand{\collisionop}{\ensuremath{\cC}}
\newcommand{\collision}[1]{\ensuremath{\collisionop\left(#1\right)}}

\newcommand{\collisionkernel}{\ensuremath{K}}

\newcommand{\dirac}{\ensuremath{\delta}}
\newcommand{\diracmd}{\ensuremath{\bsdelta}}

\newcommand{\indicator}[1]{\ensuremath{\mathbbm{1}_{#1}}}

\newcommand{\Lp}[1]{\ensuremath{L_{#1}}}
\newcommand{\RD}[2]{\ensuremath{\mathcal{R}_{#1}^{#2}}}
\newcommand{\RDdist}[1]{\ensuremath{\widetilde{\mathcal{R}}_{#1}}}
\newcommand{\RDpos}[1]{\ensuremath{\mathcal{R}^{+}_{#1}}}

\newcommand{\RQ}[1]{\RD{#1}{\mathcal{Q}}}

\newcommand{\AnsatzSpace}[1][\basis]{\ensuremath{\cA_{#1}}}

\newcommand{\PN}[1][\momentorder]{\ensuremath{\text{P}_{#1}}}
\newcommand{\MN}[1][\momentorder]{\ensuremath{\text{M}_{#1}}}

\newcommand{\HFMN}[1][\momentnumber]{\ensuremath{\text{HFM}_{#1}}}
\newcommand{\HFMIk}[1][\hankelhalfind]{\ensuremath{\text{HFM}_{\partition[#1]}}}
\newcommand{\HFMTh}{\ensuremath{\text{HFM}_{\TriangulationSphere}}}
\newcommand{\HFPN}[1][\momentnumber]{\ensuremath{\text{HFP}_{#1}}}
\newcommand{\HFPIk}[1][\hankelhalfind]{\ensuremath{\text{HFP}_{\partition[#1]}}}
\newcommand{\HFPTh}{\ensuremath{\text{HFP}_{\TriangulationSphere}}}

\newcommand{\PMMN}[1][\momentnumber]{\ensuremath{\text{PMM}_{#1}}}
\newcommand{\PMMIk}[1][\hankelhalfind]{\ensuremath{\text{PMM}_{\partition[#1]}}}
\newcommand{\PMMTh}{\ensuremath{\text{PMM}_{\TriangulationSphere}}}
\newcommand{\PMPN}[1][\momentnumber]{\ensuremath{\text{PMP}_{#1}}}
\newcommand{\PMPIk}[1][\hankelhalfind]{\ensuremath{\text{PMP}_{\partition[#1]}}}
\newcommand{\PMPTh}{\ensuremath{\text{PMP}_{\TriangulationSphere}}}

\newcommand{\Flux}{\ensuremath{\bF}}
\newcommand{\Source}{\ensuremath{\bs}}

\DeclareMathOperator*{\argmin}{argmin}
\newcommand{\ld}[1]{\ensuremath{{#1}_*}} 
\newcommand{\entropy}{\ensuremath{\eta}} 
\newcommand{\entropyFunctional}{\ensuremath{\mathcal{H}}} 

\def\quand{\quad \mbox{and} \quad}

\newcommand{\x}{\ensuremath{x}}
\newcommand{\y}{\ensuremath{y}}
\newcommand{\z}{\ensuremath{z}}
\newcommand{\dx}{\partial_{\x}}
\newcommand{\dy}{\partial_{\y}}
\newcommand{\dz}{\partial_{\z}}
\newcommand{\dt}{\partial_\timevar}

\newcommand{\rank}[2]{\mathrm{rank}_{#1}\hspace{-2pt}\left(#2\right)}

\newcommand{\posblockstartset}{\ensuremath{S}}
\newcommand{\posblockendset}{\ensuremath{E}}
\newcommand{\posblockorder}{\ensuremath{\ord}}

\newcommand{\hankelA}{\mathbf{A}}
\newcommand{\hankelB}{\mathbf{B}}
\newcommand{\hankelC}{\mathbf{C}}

\newcommand{\hankelrank}{r}

\newcommand{\hankelhalfind}{k}

\newcommand{\partition}[1][\hankelhalfind]{\ensuremath{\mathcal{I}_{#1}}}
\newcommand{\generalpartition}{\ensuremath{\mathcal{P}}}
\newcommand{\generalpart}[1][\hankelhalfind]{\ensuremath{V_{#1}}}

\newcommand{\cellind}{\ensuremath{j}}

\newcommand{\cell}[1]{\ensuremath{I_{#1}}}
\newcommand{\ccell}[1]{\cell{{#1}}}

\newcommand{\SpaceOfPolynomials}[1]{\ensuremath{P^{#1}}}

\newcommand{\angularDomain}{\ensuremath{V}}
\newcommand{\angularQuadrature}{\ensuremath{\mathcal{Q}}}

\newcommand{\angularQuadratureIndex}{\ensuremath{i}}
\newcommand{\angularQuadratureNumber}[1][\angularQuadrature]
{\ensuremath{{n_{#1}}}}

\newcommand{\angularQuadratureNodes}[1]{\ensuremath{\hat{\SCheight}_{#1}}}
\newcommand{\angularQuadratureNodesS}[1]{\ensuremath{\SC_{#1}}}
\newcommand{\angularQuadratureWeights}[1]{\ensuremath{w_{#1}}}

\newcommand{\intQ}[1]{\ensuremath{\intA{#1}{\angularQuadrature}}}

\newcommand{\Dx}{\nabla_\spatialVariable}
\newcommand{\momentstens}[2][ ]{\ensuremath{\bu_{#1}^{\abs{#2}}}} 
\newcommand{\normalizedmomentstens}[2][
]{\ensuremath{\normalizedmoments[#1]^{\abs{#2}}}} 

\newcommand{\SCx}{\ensuremath{\SC_\x}}
\newcommand{\SCy}{\ensuremath{\SC_\y}}
\newcommand{\SCz}{\ensuremath{\SC_\z}}

\makeatletter
\DeclareFontFamily{U}{tipa}{}
\DeclareFontShape{U}{tipa}{m}{n}{<->tipa10}{}
\newcommand{\arc@char}{{\usefont{U}{tipa}{m}{n}\symbol{62}}}%

\newcommand{\arc}[1]{\mathpalette\arc@arc{#1}}

\newcommand{\arc@arc}[2]{%
  \sbox0{$\m@th#1#2$}%
  \vbox{
    \hbox{\resizebox{\wd0}{\height}{\arc@char}}
    \nointerlineskip
    \box0
  }%
}
\makeatother

\newcommand{\TriangulationSphere}{\ensuremath{\mathcal{T}_h}}
\newcommand{\sphericaltriangle}{\ensuremath{\arc{K}}}
\newcommand{\flattriangle}{\ensuremath{K}}
\newcommand{\bcA}{\ensuremath{\lambda_{A}}}
\newcommand{\bcB}{\ensuremath{\lambda_{B}}}
\newcommand{\bcC}{\ensuremath{\lambda_{C}}}

\newcommand{\vertexA}[1][]{\ensuremath{A_{#1}}}
\newcommand{\vertexB}[1][]{\ensuremath{B_{#1}}}
\newcommand{\vertexC}[1][]{\ensuremath{C_{#1}}}
\newcommand{\vertex}[1][]{\ensuremath{v_{#1}}}
\newcommand{\nvertex}{\ensuremath{n_v}}
\newcommand{\sphericalprojection}[1]{\ensuremath{g\left(#1\right)}}

\pgfplotscreateplotcyclelist{color parula}{%
	royalblue!70,every mark/.append style={solid,line width = 0pt,fill=royalblue!60!black},mark=ball\\%
	goldenrod!50!yelloworange,every mark/.append style={solid,fill=goldenrod!30!black},mark=*\\%
	green,every mark/.append style={black,fill=yellowgreen},mark=diamond*\\%
	bluegreen,every mark/.append style={fill=black},mark=triangle*\\%
	royalblue!70,densely dashed,every mark/.append style={solid,line width = 0pt,fill=royalblue!60!black},mark=ball\\%
	goldenrod!50!yelloworange,densely dashed,every mark/.append style={solid,black,fill=goldenrod},mark=diamond*\\%
	yellowgreen,densely dashed,every mark/.append style={solid,fill=yellowgreen!60!royalblue},mark=square*, mark size= 1pt\\%
	bluegreen,densely dashed,mark size = 1.5pt,every mark/.append style={solid,fill=white},mark=otimes*\\
	royalblue,densely dashed,mark size = 1.5pt,every mark/.append style={solid,fill=royalblue!30!black},mark=pentagon*\\
	goldenrod!40!yelloworange,every mark/.append style={solid,fill=goldenrod!30!black},mark=|\\%
	}

\definecolor{blueviolet}{rgb}{0.54, 0.17, 0.89}
\pgfplotscreateplotcyclelist{color parula pairwise}{%
	royalblue!70,solid,every mark/.append style={solid,line width =
0pt,fill=royalblue!60!black},mark=ball\\%
	royalblue!70,densely dashed,every mark/.append style={solid,line width =
0pt,fill=royalblue!60!black},mark=*\\%
	goldenrod!50!yelloworange,solid,every mark/.append
style={solid,fill=goldenrod!30!black},mark=diamond*\\%
	goldenrod!50!yelloworange,densely dashed, every mark/.append
style={solid,fill=goldenrod!30!black},mark=triangle*\\%
	green,solid,every mark/.append style={solid, black,fill=yellowgreen},mark=square*\\%
	green,densely dashed,every mark/.append style={solid, black, fill=yellowgreen},mark=pentagon*\\%
        bluegreen,every mark/.append style={fill=black},mark=triangle*\\%
	blueviolet,densely dashed,every mark/.append style={solid,fill=blueviolet},mark=square*, mark size= 1pt\\%
	bluegreen,densely dashed,mark size = 1.5pt,every mark/.append style={solid,fill=white},mark=otimes*\\
	royalblue,densely dashed,mark size = 1.5pt,every mark/.append style={solid,fill=royalblue!30!black},mark=pentagon*\\
	goldenrod!40!yelloworange,every mark/.append style={solid,fill=goldenrod!30!black},mark=|\\%
	}

\usepackage[colorinlistoftodos,obeyFinal]{todonotes}

\pgfplotsset{
    discard if not/.style 2 args={
        x filter/.code={
            \edef\tempa{\thisrow{#1}}
            \edef\tempb{#2}
            \ifx\tempa\tempb
            \else
                
            \fi
        }
    }
}

\DeclarePairedDelimiterX\Set[1]\{\}{%

\,#1\,
}

%% file: Sections/introduction.tex
\section{Introduction}
Moment closures are a type of (non-linear) Galerkin approximation typically used in the context of
kinetic transport equations.
An infinite set of moment equations is defined by taking velocity-
or phase-space averages of the kinetic density with respect to some basis of this space. A
reduced description of the kinetic density is then
achieved by truncating this hierarchy of equations at some finite order.
Unfortunately, in most cases, the remaining equations are not closed and require information from
the equations which were removed.
The specification of this information, the so-called moment closure problem, distinguishes different types of moment models.
In the context of linear radiative transport, the standard spectral method
is commonly referred to as the $\PN$ closure \cite{Lewis-Miller-1984},
where $\momentorder$ is the degree of the highest-order moments in the model. It is basically a
straight-forward Galerkin approximation on the unit sphere. The $\PN$ method is powerful and simple
to implement, but does not take into account that the original function to be
approximated, the kinetic density, must be non-negative. This often leads to physically meaningless
solutions, as the $\PN$ solutions can, e.g., contain negative values for the local
particle density.

In the context of radiative transport, entropy-based moment closures, the so-called $\MN$ models
\cite{Min78,DubFeu99}, have all the properties one would desire in a moment method, namely
positivity of the underlying kinetic density, hyperbolicity of the closed system of equations, and
entropy dissipation \cite{Lev96}. Although there have been a lot of theoretical investigations about
minimum-entropy models throughout the last fifty years, see, e.g.,
\cite{Levermore1984,Mead1984,Cernohorsky1994,DubFeu99,Junk2000}, practical
implementations were limited to very low orders
(e.g. $\momentorder=1$) \cite{BruHol01,Olbrant2012,Chidyagwai2017,Chidyagwai2016} as, generally,
$\MN$ models are too expensive and more difficult to implement compared to, e.g., direct
Monte Carlo or Discrete Ordinates simulations \cite{Garrett2014,Schneider2016,Andreo91}. They require the (numerical) solution
of a non-linear optimization problem at every
point on the space-time grid, which tends to be very time-consuming. However, there has been renewed
interest in these models recently due to their inherent parallelizability \cite{Hauck2010,Alldredge2012}.

The standard $\MN$ models use a polynomial basis on the velocity space. As a consequence,
non-smooth kinetic densities (which very frequently occur in realistic problems) are poorly
captured and often a very high moment order $\momentorder$ is needed for a sufficient
approximation.
To increase the accuracy of the $\MN$ models while maintaining the lower cost for small
$\momentorder$, partial moment and mixed moment models have been developed
\cite{Frank2006,DubKla02,Schneider2014,Ritter2016,Schneider2017}. They are based on a
partition of the velocity space while keeping the moment order fixed, similar to
some h-refinement for, e.g., finite element approximations \cite{Babuska1992}.

While the partial moment models have been extensively studied for special cases (like half- or
quarter-moments in one or two dimensions), we are unaware of any general
investigation, especially in the fully three-dimensional setup.

In this paper, which is the first of two parts, we provide realizability theory for general
first-order piece-wise discontinuous partial moments in one- and three-dimensional geometry, as well
as their continuous analogue. An extensive numerical study using a second-order
realizability-preserving scheme will be done in the second part. Here, the required realizability
can cause problems for numerical methods, as standard high-order numerical solutions (in space and
time) can destroy this property if not handled very carefully
\cite{Zhang2010,Schar1996,Schneider2015a,Schneider2015b,Schneider2016a,Chidyagwai2017,Olbrant2012}.

This paper is organized as follows. First, the transport equation and its moment approximations are
given. Then, the new moment bases are presented and the available realizability theory is
derived in one and three dimensions, followed by some numerical investigations of the approximation
properties of our models. Finally, conclusions and an outlook on future work is given.

%% file: Sections/modelling.tex
\section{Modeling}

We consider the linear transport equation
\begin{subequations}
\label{eq:FokkerPlanckEasy}
\begin{align}
\label{eq:TransportEquation}
\dt\distribution+\SC\cdot\Dx\distribution + \absorption\distribution = \scattering\collision{\distribution}+\source,
\end{align}
which describes the density of particles with speed $\SC\in\sphere$ at position
$\spatialVariable\in\Domain\subseteq\R^3$ and time $\timevar \in \timeint = [0, \tf]$ under the
events of scattering
(proportional to $\scattering\left(\timevar,\spatialVariable\right)$), absorption (proportional to $\absorption\left(\timevar,\spatialVariable\right)$) and emission (proportional
to $\source\left(\timevar,\spatialVariable,\SC\right)$). Collisions are modeled using the BGK-type collision operator
\begin{equation}
 \collision{\distribution} =  \int\limits_{\sphere} \collisionkernel(\SC, \SC^\prime)
  \distribution(\timevar, \spatialVariable, \SC^\prime)~d\SC^\prime
  - \int\limits_{\sphere} \collisionkernel(\SC^\prime, \SC) \distribution(\timevar, \spatialVariable, \SC)~d\SC^\prime.
\label{eq:collisionOperatorR}
\end{equation}
The collision kernel $\collisionkernel$ is assumed to be strictly positive, symmetric (i.e.\, $\collisionkernel(\SC,\SC')=\collisionkernel(\SC',\SC)$) and normalized to
$\int\limits_{\sphere} \collisionkernel(\SC^\prime, \SC) d\SC^\prime~\equiv~1$.  A typical example is
\emph{isotropic scattering}, where $\collisionkernel(\SC, \SC^\prime) \equiv \frac{1}{\abs{\sphere}} = \frac{1}{4\pi}$.

The equation is supplemented with initial condition and Dirichlet boundary conditions:
\begin{align}
\distribution(0,\spatialVariable,\SC) &= \distributiontzero(\spatialVariable,\SC) &\text{for } \spatialVariable\in\Domain, \SC\in\sphere\\
\distribution(\timevar,\spatialVariable,\SC) &= \distributionboundary(\timevar,\spatialVariable,\SC) &\text{for } \timevar\in\timeint, \spatialVariable\in\partial\Domain, \outernormal\cdot\SC<0
\end{align}
where $\outernormal$ is the outward unit normal vector in $\spatialVariable\in\partial\Domain$.
\end{subequations}

Parameterizing $\SC$ in spherical coordinates we obtain
\begin{align}
\label{eq:SphericalCoordinates}
\SC = \left(\sqrt{1-\SCheight^2}\cos(\SCangle),\sqrt{1-\SCheight^2}\sin(\SCangle), \SCheight\right)^T =: \left(\SCx,\SCy,\SCz\right)^T,
\end{align}
where $\SCangle\in[0,2\pi]$ is the azimuthal and $\SCheight\in[-1,1]$ the cosine of the polar angle.

\begin{definition}
The vector of functions $\basis:\sphere\to\R^{\momentnumber}$ consisting of $\momentnumber$ basis functions $\basiscomp[i]$,
$\basisind=1,\ldots\momentnumber$, of maximal
\emph{order} $\momentorder$ (in $\SC$) is called an \emph{angular basis}.
The so-called \emph{moments} $\moments=\left(\momentcomp{0},\ldots,\momentcomp{\momentnumber-1}\right)^T$ of a given density function $\distribution$ are then defined by
\begin{align}
\label{eq:moments}
\moments = \int\limits_{\sphere} {\basis}\distribution~d\SC =: \ints{\basis\distribution},
\end{align}
where the integration is performed component-wise.

Furthermore, the quantity $\density = \density(\moments) := \ints{\distribution}$ is called the \emph{local particle density}. Additionally, $\isotropicmoment = \ints{\basis}$ is called the \emph{isotropic moment}.
\end{definition}

Equations for $\moments$ can then be obtained by multiplying \eqref{eq:FokkerPlanckEasy} with $\basis$ and integration over $\sphere$, yielding
\begin{align*}
\ints{\basis\dt\distribution}+\ints{\basis\Dx\cdot\SC\distribution} + \ints{\basis\absorption\distribution} = \scattering\ints{\basis\collision{\distribution}}+\ints{\basis\source}.
\end{align*}
Collecting known terms, and interchanging integration and differentiation where possible, the moment system has the form
\begin{equation}
\label{eq:MomentSystemUnclosed}
\dt\moments+\ints{\basis\Dx\cdot\SC\distribution} + \absorption\moments =
\scattering\ints{\basis\collision{\distribution}}+\ints{\basis\source}.
\end{equation}
Depending on the choice of $\basis$ the terms $\ints{\SCx \basis\distribution}$, $\ints{\SCy \basis\distribution}$, $\ints{\SCz \basis\distribution}$, and in some cases even $\ints{\basis\collision{\distribution}}$, cannot be given explicitly in terms of $\moments$. Therefore an ansatz $\ansatz$ has to be made for $\distribution$ closing the unknown terms. This is called the \emph{moment-closure problem}.

In this paper the ansatz density $\ansatz$ is reconstructed from the moments $\moments$ by minimizing the entropy-functional
\begin{subequations}
	\label{eq:OptProblem}
 \begin{equation}
 \label{eq:entropyFunctional}
 \entropyFunctional(\distribution) = \ints{\entropy(\distribution)}
 \end{equation}
 under the moment constraints
 \begin{equation}
 \label{eq:MomentConstraints}
 \ints{\basis\distribution} = \moments.
 \end{equation}
 \end{subequations}
The kinetic entropy density $\entropy:\R\to\R$ is strictly convex and twice continuously differentiable and the minimum is simply taken over all functions
$\distribution = \distribution(\SC)$ such that $\entropyFunctional(\distribution)$ is well defined. The obtained ansatz $\ansatz = \ansatz[\moments]$, solving this constrained
optimization problem, is given by
 \begin{equation}
  \ansatz[\moments] = \argmin\limits_{\distribution:\entropy(\distribution)\in\Lp{1}}\left\{\ints{\entropy(\distribution)}
  : \ints{\basis \distribution} = \moments \right\}.
 \label{eq:primal}
 \end{equation}
This problem, which must be solved over the space-time mesh, is typically solved through its strictly convex finite-dimensional dual,
 \begin{equation}
  \multipliers(\moments) := \argmin_{\tilde{\multipliers} \in \R^{\momentnumber}} \ints{\ld{\entropy}(\basis^T
   \tilde{\multipliers})} - \moments^T \tilde{\multipliers},
 \label{eq:dual}
 \end{equation}
where $\ld{\entropy}$ is the Legendre dual of $\entropy$. The first-order necessary conditions for the multipliers $\multipliers(\moments)$ show that the solution to
\eqref{eq:primal} has the form
 \begin{equation}
  \ansatz[\moments] = \ld{\entropy}' \left(\basis^T \multipliers(\moments) \right),
 \label{eq:psiME}
 \end{equation}
where $\ld{\entropy}'$ is the derivative of $\ld{\entropy}$.

Substituting $\distribution$ in \eqref{eq:MomentSystemUnclosed} with $\ansatz[\moments]$ yields a
closed system of equations for $\moments$:
\begin{equation}
\label{eq:MomentSystemClosed}
\dt\moments+\dx\ints{\SCx \basis\ansatz[\moments]}+\dy\ints{\SCy
\basis\ansatz[\moments]}+\dz\ints{\SCz \basis\ansatz[\moments]} + \absorption\moments =
\scattering\ints{\basis\collision{\ansatz[\moments]}}+\ints{\basis\source}.
\end{equation}

This approach is called the \emph{minimum-entropy closure} \cite{Levermore1996,Min78, minerbo1979ment,Friedrichs1971}. The resulting model has many desirable properties: symmetric hyperbolicity, bounded eigenvalues of
the directional flux Jacobian and the direct existence of an entropy-entropy flux pair (compare \cite{Levermore1996,Schneider2016}).

The kinetic entropy density $\entropy$ can be chosen according to the
physics being modeled.

As in \cite{Levermore1996,Hauck2010}, we focus on \emph{Maxwell-Boltzmann entropy}%
\begin{align}
\label{eq:EntropyM}
\entropy(\distribution) = \distribution \log(\distribution) - \distribution,
\end{align}
thus $\ld{\entropy}(p) = \ld{\entropy}'(p) = \exp(p)$. This entropy is used for
non-interacting, classical particles as in an ideal gas. {Other physically relevant
entropies are \cite{Levermore1996}}
\begin{align*}
{\entropy(\distribution) = \distribution \log(\distribution) +
\left(1-\distribution\right)\log\left(1-\distribution\right)}
\end{align*}
{for particles satisfying \emph{Fermi-Dirac} (e.g. fermions) \big($\ld{\entropy}(p) =
\log(\exp(p) + 1)$, $\ld{\entropy}'(p) = \frac{\exp(p)}{1+\exp(p)}$\big) or}
\begin{align*}
{\entropy(\distribution) = \distribution \log(\distribution) -
\left(1+\distribution\right)\log\left(1+\distribution\right)}
\end{align*}
{for particles satisfying \emph{Bose-Einstein statistics} (e.g. bosons)
\big($\ld{\entropy}(p) = - \log(1-\exp(p))$,  $\ld{\entropy}'(p) = \frac{\exp(p)}{1-\exp(p)}$\big).
Although photons are no classical particles but bosons, the approximation using the
Maxwell-Boltzmann entropy is widely used due to its simplicity \cite{Min78}. However, as in our
special case the adaption to the Bose-Einstein case is rather straight-forward, we also include it
in our demonstration code. Note that the modifications needed for the Fermi-Dirac case are more
involved, as the additional requirement $\distribution\leq 1$ changes many properties of the models.
In particular, realizability (see below) has to be treated separately for this entropy.}

Note that these physically relevant entropies all result in positive ansatz densities
\eqref{eq:psiME} (assuming $\basis^T \multipliers(\moments) < 0$ in the Bose-Einstein case). However, positivity is actually not gained for
every entropy-based moment closure. Using the entropy
\begin{align} \label{eq:EntropyP}
\entropy(\distribution) = \frac12\distribution^2,
\end{align}
the linear ansatz
\begin{align}
\label{eq:PnAnsatz}
\ansatz[\moments] = \basis^T \multipliers(\moments)
\end{align}
is obtained, leading to standard continuous/discontinuous-Galerkin approaches. In this case, the optimization problem can be solved analytically yielding
\begin{equation}
\multipliers(\moments) = \massmatrix^{-1} \moments
\end{equation}
where $M_{ij} = \ints{\basiscomp[i] \basiscomp[j]}$ \cite{Schneider2016,Alldredge2014}. If the angular basis is chosen as spherical harmonics of
order $\momentorder$, \eqref{eq:MomentSystemClosed} turns into the classical $\PN$ model
\cite{Blanco1997,Brunner2005b,Seibold2014}, which thus can be considered a minimum-entropy
model that does not guarantee positivity.

For convenience, we write \eqref{eq:MomentSystemClosed} in the standard form of a
{non-linear system of hyperbolic balance laws}:
\begin{equation}
\label{eq:GeneralHyperbolicSystem2D}
\dt\moments+\dx\Flux_1\left(\moments\right)+\dy\Flux_2\left(\moments\right)+\dz\Flux_3\left(\moments\right) = \Source\left(\moments\right),
\end{equation}
where
\begin{subequations}
\label{eq:FluxDefinitions}
\begin{align}
\Flux_1\left(\moments\right) &=  \ints{\SC_x\basis\ansatz[\moments]},\quad \Flux_2\left(\moments\right) =  \ints{\SC_y\basis\ansatz[\moments]},\quad \Flux_3\left(\moments\right) =  \ints{\SC_z\basis\ansatz[\moments]}\in\R^{\momentnumber},\\
\Source\left(\moments\right) &= {\scattering}\ints{\basis \collision{\ansatz[\moments]}}+\ints{\basis\source}-\absorption\moments.
\end{align}
\end{subequations}

To complete the definition of our moment method, we have to choose an angular basis. This
will be done in the next section. Due to the notational complexity of the full three-dimensional
setting we will first derive the models in slab geometry, which is a projection of the sphere onto
the $\z$-axis \cite{Seibold2014}. The transport equation under consideration then has the form
\begin{align}
\label{eq:TransportEquation1D}
\dt\distribution+\SCheight\dz\distribution + \absorption\distribution = \scattering\collision{\distribution}+\source, \qquad \timevar\in\timeint,\z\in\Domain,\SCheight\in[-1,1].
\end{align}
The shorthand notation $\ints{\cdot} = \int\limits_{-1}^1\cdot~d\SCheight$ then denotes integration over $[-1,1]$ instead of $\sphere{}$. Finally, the moment system is given by
\begin{align}
\label{eq:MomentSystemUnclosed1D}
\dt\moments+\dz\ints{\SCheight \basis\ansatz[\moments]} + \absorption\moments = \scattering\ints{\basis\collision{\ansatz[\moments]}}+\ints{\basis\source}.
\end{align}

%% file: Sections/momentmodels.tex
\section{Angular bases in slab geometry}
\begin{definition}
  Let $\basis[]$ be an angular basis. The span of $\basis$ is defined as
    \begin{align*}
       \spanop(\basis[]) = \big\{ f \in \Lp{1}(\angularDomain,\R) \mid f \equiv \sum_{i=0}^{\momentnumber-1} a_i \basiscomp[i], a_i \in \R \big\},
    \end{align*}
  where $\angularDomain = [-1,1]$ in slab geometry and $\angularDomain = \sphere$ in the full three-dimensional setting.
\end{definition}
\begin{definition}
  	Let $\basis[1]$ and $\basis[2]$ be two angular bases. We call them \textit{equivalent} if $\spanop(\basis[1]) = \spanop(\basis[2])$.
\end{definition}
\subsection{Full moments}
The standard Galerkin approach, which results in the classical $\MN$ and $\PN$ models
(using the
entropies \eqref{eq:EntropyM} and \eqref{eq:EntropyP}, respectively), chooses a polynomial basis
$\basis = \fmbasis$ on $[-1,1]$ as angular basis. There are two obvious options for $\fmbasis$,
resulting in equivalent models:
\begin{itemize}
\item the monomial basis
\begin{equation}
\label{eq:monomialbasis}
\fmbasis[\momentorder] = \left(1,\SCheight,\ldots,\SCheight^\momentorder\right)^T;
\end{equation}
\item the Legendre basis
\begin{equation}
\label{eq:legendrebasis}
\fmbasis[\momentorder] = \left(P_0^0,P_1^0,P_2^0\ldots,P_\momentorder^0\right)^T
\end{equation}
with the \emph{Legendre polynomials} $P_\SHl^0$, $\SHl=0,\ldots,\momentorder$.
\end{itemize}
The monomial basis is attractive due to its simplicity, especially for realizability considerations
\cite{Curto1991}. On the other hand, the Legendre polynomials form an orthogonal basis of
$\Lp{2}([-1,1],\R)$ and are additionally eigenfunctions of the isotropic-scattering operator
$\collisionop$ if $\collisionkernel\equiv \frac12$, diagonalizing the
source-term. We thus use Legendre polynomials for the one-dimensional $\MN$ and $\PN$
models in our
numerical experiments.

The full-moment approach is able to accurately approximate smooth density functions, often with
just a few moments. However, it struggles with discontinuous densities which
are abundant in practice. In addition, an obvious problem of full-moment models is that averaging
over the complete angular domain may remove necessary information. In the following, we will thus
investigate basis functions with local support in the angular domain. Although higher-order models
are possible, we will stick to first-order approximation, i.e.\ piecewise linear basis functions.
\subsection{Continuous piecewise-linear angular basis (Hat functions)}
\label{sec:HatFunctions1D}
As mentioned above, the method of moments is nothing else than a Galerkin approximation of the
kinetic equation \eqref{eq:TransportEquation1D} in the angular variable. The full-moment $\PN$
model is equivalent to the finite element method (FEM) with fixed angular mesh-size and increasing
order (this method is often referred to as the $p$-version, related to the maximal order
$p$ of the polynomial basis \cite{Babuska1992}). Another natural approach is the standard finite
element method, where the polynomial order is fixed but the angular mesh is refined
(often referred to as the $h$-version, where $h$ denotes the mesh-size \cite{Babuska1992}).

Let $\partition = \partition(\SCheight_0, \ldots, \SCheight_{\hankelhalfind})$ be the partition of $[-1,1]$ given by a set of $k+1$ angular ``grid'' points
$-1 = \SCheight_0 < \SCheight_1 < \ldots < \SCheight_{\hankelhalfind-1} < \SCheight_{\hankelhalfind}=1$ and
corresponding intervals $\cell{\cellind} = [\SCheight_\cellind,\SCheight_{\cellind+1}]$, $\cellind = 0, \ldots,  \hankelhalfind-1$. Given this partition, the continuous
piecewise-linear basis
functions $\hfbasis[\partition] = \left(\hfbasiscomp[0],\ldots,\hfbasiscomp[\hankelhalfind]\right)^T$ (hat functions or B-splines of first order \cite{Boor1979}) are defined as
\begin{subequations}
\label{eq:hfbasis}
\index{angularbases@\textbf{Angular bases}!Hat-function basis $\hfbasis$}
\begin{align}
\hfbasiscomp[\basisind](\SCheight) =
\begin{cases}
\indicator{\ccell{0}}\cfrac{\SCheight-\SCheight_{1}}{\SCheight_0-\SCheight_{1}} & \text{ if } \basisind=0,\\
\indicator{\ccell{\basisind-1}}\cfrac{\SCheight-\SCheight_{\basisind-1}}{\SCheight_\basisind-\SCheight_{\basisind-1}}
+\indicator{\ccell{\basisind}}\cfrac{\SCheight-\SCheight_{\basisind+1}}{\SCheight_\basisind-\SCheight_{\basisind+1}} & \text{ if } 0<\basisind<\hankelhalfind,\\
\indicator{\ccell{\hankelhalfind-1}}\cfrac{\SCheight-\SCheight_{\hankelhalfind-1}}{\SCheight_{\hankelhalfind}-\SCheight_{\hankelhalfind-1}}& \text{ if } \basisind=\hankelhalfind,\\
\end{cases}
\end{align}
where $\indicator{\ccell{\cellind}}(\SCheight)$ is the indicator function on the interval
$\ccell{\cellind}$.
\end{subequations}

\begin{definition}
\label{def:HFPNHFMN}
The linear ({using entropy} \eqref{eq:EntropyP}) and nonlinear model ({using entropy}
\eqref{eq:EntropyM}) with angular
basis $\basis = \hfbasisIk$ will be called $\HFPIk$ and $\HFMIk$, respectively.
If $\partition = \partition(-1, -1+\frac{2}{k}, \ldots, -1+(k-1)\cdot\frac{2}{k}, 1)$ is the equidistant partition with $\hankelhalfind$ intervals,
we will refer to the basis as $\hfbasis$ and to the models as $\HFPN$ and $\HFMN$, respectively, where $\momentnumber = \hankelhalfind+1$ is the number of basis functions.
\end{definition}
One of the most important properties of the hat functions is the following.
\begin{lemma}
The hat functions form a \emph{partition of unity}, i.e.
$\sum\limits_{\basisind=0}^{\momentnumber-1} \hfbasiscomp[\basisind] \equiv 1$.
\end{lemma}
\begin{corollary}
\label{cor:HFSpecialCases}
\begin{enumerate}[(a)]
\item The full moment basis $\fmbasis[1]$ is equivalent to $\hfbasis[2]$.
\item The mixed moment basis \cite{Frank07,Schneider2014} $\mmbasis[1] = \left(1,\indicator{[0,1]}~\SCheight,\indicator{[-1,0]}~\SCheight\right)^T$ is equivalent to $\hfbasis[3]$.
\end{enumerate}
\end{corollary}

\begin{proof}
The claim follows immediately, observing that
\begin{enumerate}[(a)]
\item $\fmbasis[1] = \left(1,\SCheight\right)^T = \left(\hfbasiscomp[1]+\hfbasiscomp[2],-\hfbasiscomp[1]+\hfbasiscomp[2]\right)^T$,
\item $\mmbasis[1] = \left(1,\indicator{[0,1]}~\SCheight,\indicator{[-1,0]}~\SCheight\right)^T = \left(\hfbasiscomp[1]+\hfbasiscomp[2]+\hfbasiscomp[3],\hfbasiscomp[3],-\hfbasiscomp[1]\right)^T$. \qedhere
\end{enumerate}
\end{proof}
An advantage of the piecewise linear basis over higher order polynomials is that integrals including
the minimum entropy ansatz density (using Maxwell-Boltzmann entropy) $\ansatz[\moments] =
\exp\left(\sum_{\basisind}\hfbasiscomp[\basisind]\multiplierscomp{\basisind}\right)$ can be calculated
analytically. Exemplarily, the $\HFMN$ integrals are given by
\begin{align*}
\intA{\ansatz[\moments]\hfbasiscomp}{\cell{\cellind-1}} &=
 \frac{\left(e^{\multiplierscomp{\basisind-1}} - e^{\multiplierscomp{\basisind}}\right)\, \left(\SCheight_{\basisind}
 - \SCheight_{\basisind-1}\right)}{{\left(\multiplierscomp{\basisind-1} - \multiplierscomp{\basisind}\right)}^2} - \frac{e^{\multiplierscomp{\basisind}}\,
\left(\SCheight_{\basisind} - \SCheight_{\basisind-1}\right)}{\multiplierscomp{\basisind-1} - \multiplierscomp{\basisind}}
= (\SCheight_{\basisind}-\SCheight_{\basisind-1})e^{\multiplierscomp{\basisind}} \sum \limits_{l=2}^{\infty}
\frac{(\multiplierscomp{\basisind-1}-\multiplierscomp{\basisind})^{l-2}}{l!}  \\
\intA{\ansatz[\moments]\hfbasiscomp}{\cell{\cellind}} &=
 \frac{\left(e^{\multiplierscomp{\basisind}} - e^{\multiplierscomp{\basisind+1}}\right)\, \left(\SCheight_{\basisind} -
\SCheight_{\basisind+1}\right)}{{\left(\multiplierscomp{\basisind} - \multiplierscomp{\basisind+1}\right)}^2} - \frac{e^{\multiplierscomp{\basisind}}\, \left(\SCheight_{\basisind}
- \SCheight_{\basisind+1}\right)}{\multiplierscomp{\basisind} - \multiplierscomp{\basisind+1}}
= (\SCheight_{\basisind+1}-\SCheight_{\basisind})e^{\multiplierscomp{\basisind}}\sum \limits_{l=2}^{\infty}
\frac{(\multiplierscomp{\basisind+1}-\multiplierscomp{\basisind})^{l-2}}{l!} \\
\momentcomp{\basisind} &=
\intA{\ansatz[\moments]\hfbasiscomp}{\cell{\cellind}} +
\intA{\ansatz[\moments]\hfbasiscomp}{\cell{\cellind-1}}.
\end{align*}
Note that for $\multiplierscomp{\basisind-1}\approx \multiplierscomp{\basisind}$ or
$\multiplierscomp{\basisind+1}\approx \multiplierscomp{\basisind}$ (which occurs for example in the
isotropic
state) the closed formulas become numerically instable. This can be avoided by replacing them in
this situation with a Taylor expansion of suitable order.
\subsection{Partial moments (discontinuous-Galerkin ansatz)}
\label{sec:PartialMoments1D}

The idea of partial moments is
not to average over the full domain at
once but to partition the sphere (or its projection) into disjoint parts and define moments
separately for every element of the partition \cite{SchFraPin04,Frank2006}. One model of this
class, which has been successfully applied to radiative transfer in one dimension, is the
half-moment approximation \cite{DubKla02}.

Let the partition $\partition$ be defined as in \secref{sec:HatFunctions1D}. For every interval $\cell{\cellind}$ we define moments by
\begin{equation*}
\moments[\cell{\cellind}] = \intA{\pmbasis[\cell{\cellind}]\distribution}{\cell{\cellind}} = \int_{\cell{\cellind}} \pmbasis[\cell{\cellind}]\distribution~d\SCheight.
\end{equation*}
The basis functions $\pmbasis[\cell{\cellind}]$ are monomials on their corresponding intervals and zero elsewhere. Consequently,
\begin{equation*}
  \pmbasis[\cell{\cellind}](\SCheight) = \begin{cases} \left(1,\SCheight\right)^T &\text{ if } \SCheight \in \interior{\cell{\cellind}} \\
    \mathbf{0}                          &\text{ if } \SCheight \in [-1,1]\setminus\ccell{\cellind}
  \end{cases}
\end{equation*}
Note that evaluation on the boundary of $\cell{\cellind}$ is not defined.
The set of basis functions is then given by the vector $\pmbasis[\partition] =
\left(\pmbasis[\cell{0}]^T,\ldots,\pmbasis[\cell{\hankelhalfind-1}]^T\right)^T$.
\begin{definition}
\label{def:PPnPMn}
The linear ({using entropy \eqref{eq:EntropyP}}) and nonlinear model ({using entropy \eqref{eq:EntropyM}})
with angular basis $\basis = \pmbasisIk$ will be called $\PMPIk$ and $\PMMIk$, respectively.
If $\partition$ is the equidistant partition,
we will refer to the basis as $\pmbasis$ and to the models as $\PMPN$ and $\PMMN$, respectively, where $\momentnumber = 2\hankelhalfind$ is the number of basis functions.
\end{definition}
\begin{remark}
	The $\PMMN[2]$ model is equivalent to the full-moment $\MN[1]$ model.
\end{remark}
The advantage of considering only first-order moments is that, similar to the
continuous piecewise-linear ansatz, the moment integrals in \eqref{eq:MomentConstraints} can be
calculated exactly. Assume that the ansatz satisfies
$\left.\ansatz[\moments]\right|_{\cell{\cellind}} =
\exp\left(\multiplierscomp{\cellind,0}+\multiplierscomp{\cellind,1}\SCheight\right)$. Then
\begin{align*}
\momentcomp{0,\cellind} \coloneqq \intA{\ansatz[\moments]}{\cell{\cellind}}
&= -\frac{\mathrm{e}^{\multiplierscomp{\cellind,0}}\, \left(\mathrm{e}^{\multiplierscomp{\cellind,1}\, \SCheight_{\cellind}} - \mathrm{e}^{\multiplierscomp{\cellind,1}
\SCheight_{\cellind+1}}\right)}{\multiplierscomp{\cellind,1}}
= e^{\multiplierscomp{\cellind,0}} \sum \limits_{l=1}^{\infty}
\frac{\multiplierscomp{\cellind,1}^{l-1}(\SCheight_{\cellind}^l-\SCheight_{\cellind+1}^{l})}{l!} \\
\momentcomp{1,\cellind} \coloneqq \intA{\SCheight\ansatz[\moments]}{\cell{\cellind}}
&= \frac{\mathrm{e}^{\multiplierscomp{\cellind,0} + \multiplierscomp{\cellind,1}\,
\SCheight_{\cellind+1}}\, \left(\multiplierscomp{\cellind,1}\, \SCheight_{\cellind+1} -
1\right)}{{\multiplierscomp{\cellind,1}}^2} - \frac{\mathrm{e}^{\multiplierscomp{\cellind,0}
+ \multiplierscomp{\cellind,1}\, \SCheight_{\cellind}}\, \left(\multiplierscomp{\cellind,1}\,
\SCheight_{\cellind} - 1\right)}{{\multiplierscomp{\cellind,1}}^2} \\
&= e^{\multiplierscomp{\cellind,0}} \sum \limits_{l=2}^{\infty}
\frac{\multiplierscomp{\cellind,1}^{l-2}(\SCheight_{\cellind+1}^{l}-\SCheight_{\cellind}^{l})(l-1)}{
l!}.
\end{align*}

\begin{remark} \label{rem:basisspans1d}
	We want to note that our bases satisfy the following relation:
\begin{align*}
\spanop(\fmbasis[1])\subseteq \spanop(\hfbasis[\partition]) \subseteq \spanop(\pmbasis[\partition]).
\end{align*}

\end{remark}
\section{Angular bases in three dimensions}

Like in one dimension, the full moment models are defined using a polynomial basis $\fmbasis$
over the whole velocity space $\sphere$. As a standard choice, we use real spherical harmonics of order
$\momentorder$, resulting in $\momentnumber = (\momentorder + 1)^2$ moments
for the $\PN$ and $\MN$ models.

To define (continuous and discontinuous) locally supported bases, we have to specify a
partition of the domain. Albeit both approaches are not limited to this, we only consider moments
on spherical triangles. A generalization to arbitrary convex spherical polygons is
straightforward.

Let $\TriangulationSphere$ be a spherical triangulation of $\sphere$ and
$\sphericaltriangle\in\TriangulationSphere$ be a spherical triangle with vertices
$\vertexA[\sphericaltriangle]$, $\vertexB[\sphericaltriangle]$ and $\vertexC[\sphericaltriangle]$
(or $\vertexA$, $\vertexB$, $\vertexC$ as short notation). Furthermore, let $\flattriangle$ be the
flat triangle spanned by the vertices $\vertexA$, $\vertexB$ and $\vertexC$, i.e.\,
$\sphericaltriangle = \sphericalprojection{\flattriangle}$ with $
\sphericalprojection{\spatialVariable} = \frac{\spatialVariable}{\norm{\spatialVariable}{2}}~.$ This
is shown exemplarily in \figref{fig:triangles}.

\begin{figure}
\centering
\settikzlabel{fig:sphericaltrianglenotation}
\settikzlabel{fig:gridrefinement}
\externaltikz{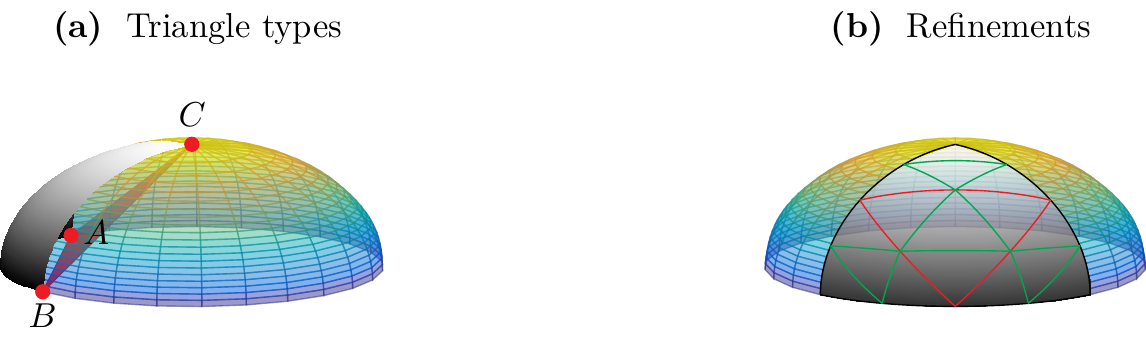}{\input{Images/sphericaltrianglenotation}}
\caption{\textbf{(a)}: definition of the flat (red) and spherical (shaded black to white) triangle on the upper half of the unit sphere.\\
\textbf{(b)}: a sequence of dyadic refinements (black, red, green).}
\label{fig:triangles}
\end{figure}

In the following, we will use a dyadic refinement \cite{baumgardner1985icosahedral} of the quadrants/octants. This is achieved by subdividing every spherical triangle
into four new ones, adding vertices at the midpoints of the triangle edges. This is shown in \figref{fig:gridrefinement} for one quadrant (black) and
two refinements (red and green). Starting from the octants, after $\refinementnumber$ refinements, the triangulation
$\TriangulationSphere(\refinementnumber)$ consists of $\abs{\TriangulationSphere(\refinementnumber)} = 2 \cdot 4^{\refinementnumber+1}$
spherical triangles with $\nvertex(\refinementnumber) = 4^{\refinementnumber+1} + 2$ vertices (see \figref{fig:BCFunctions}
for the case $\refinementnumber = 1$).

\begin{remark}
Previous results indicate that the dyadic refinement is better initialized using the vertices of an icosahedron instead of an octahedron \cite{baumgardner1985icosahedral}.
We decided to use this variant regardless to include the quarter-moment model into our sequence of refinements. However, this is only for investigative reasons.
Any practical application should start from the icosahedron.
\end{remark}

\subsection{Barycentric-coordinate basis functions}
We consider basis functions defined using spherical barycentric coordinates \cite{Buss2001,Langer2006,Rustamov2010}. These basis functions are associated with vertices of the triangulation $\TriangulationSphere$.
\begin{definition}
	\label{def:BCfuns}
Let the spherical triangle $\sphericaltriangle$ be given. The functions
$\hfbasiscomp[{\vertexA[\sphericaltriangle]}]$,
$\hfbasiscomp[{\vertexB[\sphericaltriangle]}]$ and
$\hfbasiscomp[{\vertexC[\sphericaltriangle]}]$
are defined by
\begin{enumerate}
	\item $\hfbasiscomp[{\vertexA[\sphericaltriangle]}]\left(\SC\right)+\hfbasiscomp[{
		\vertexB[\sphericaltriangle]}]\left(\SC\right)+\hfbasiscomp[{\vertexC[
		\sphericaltriangle]}]\left(\SC\right) = 1$ for every point $\SC\in\sphericaltriangle$ (partition of unity).
	\item At the vertices $\vertex\in\{\vertexA[\sphericaltriangle],\vertexB[\sphericaltriangle],\vertexC[\sphericaltriangle]\}$ of $\sphericaltriangle$:
	\begin{align}
	\label{eq:BC1}
	\hfbasiscomp[{\vertexA[\sphericaltriangle]}]\left(\vertex\right) = \begin{cases}
	1 & \text{ if } \vertex = \vertexA[\sphericaltriangle],\\
	0 & \text{ else},
	\end{cases}
	\end{align}
	and likewise for $\hfbasiscomp[{\vertexB[\sphericaltriangle]}]$ and $\hfbasiscomp[{\vertexC[\sphericaltriangle]}]$, respectively (Lagrange property).
	\item $\SC\in\sphericaltriangle$ is the
	\emph{Riemannian center of mass} with weights
	$\hfbasiscomp[{\vertexA[\sphericaltriangle]}]\left(\SC\right)$,
	$\hfbasiscomp[{\vertexB[\sphericaltriangle]}]\left(\SC\right)$,
	$\hfbasiscomp[{\vertexC[\sphericaltriangle]}]\left(\SC\right)$ and nodes
	$\vertexA[\sphericaltriangle]$, $\vertexB[\sphericaltriangle]$,
	$\vertexC[\sphericaltriangle]$, respectively.
\end{enumerate}

\end{definition}\label{def:sphericalbarycentric}
\begin{lemma}[\cite{Rustamov2010}]
\begin{enumerate}[(a)]
\item The functions $\hfbasiscomp[{\vertexA[\sphericaltriangle]}]$,
$\hfbasiscomp[{\vertexB[\sphericaltriangle]}]$ and
$\hfbasiscomp[{\vertexC[\sphericaltriangle]}]$ from \defnref{def:BCfuns} exist uniquely.
\item For every interior point $\SC\in\interior{\sphericaltriangle}$ and every $\vertex\in\{\vertexA[\sphericaltriangle],\vertexB[\sphericaltriangle],\vertexC[\sphericaltriangle]\}$ it holds that $\hfbasiscomp[\vertex]\left(\SC\right)>0$ and $\hfbasiscomp[\vertex]\left(\SC\right)\geq 0$ for every $\SC\in\sphere$.
\end{enumerate}
\label{lem:BCproperties}
\end{lemma}
Details on how to evaluate the spherical barycentric coordinates are given in \cite{Rustamov2010}. For the theoretical investigation of the resulting moment models we
only need the properties stated above.

Numbering all vertices of $\TriangulationSphere$ as $\vertex[0],\ldots,\vertex[\nvertex-1]$, the full set of basis functions is given as
\begin{equation*}
  \hfbasis[\TriangulationSphere] = \left(\hfbasiscomp[0],\ldots,\hfbasiscomp[\nvertex-1]\right).
\end{equation*}
with $h_i|_{\sphericaltriangle} = h_{\vertexA[\sphericaltriangle]}$ if $v_i$ corresponds to vertex $\vertexA$ in $\sphericaltriangle$ (and likewise for $\vertexB$, $\vertexC$)
  and $h_i|_{\sphericaltriangle} \equiv 0$ on spherical triangles that do not contain $v_i$.
  Since the basis functions form a partition of unity, it follows that $\density = \sum\limits_{\basisind=0}^{\nvertex-1} \momentcomp{\basisind}$. We show one example of such a basis function in \figref{fig:BCFunctions}.

\begin{figure}
\externaltikz{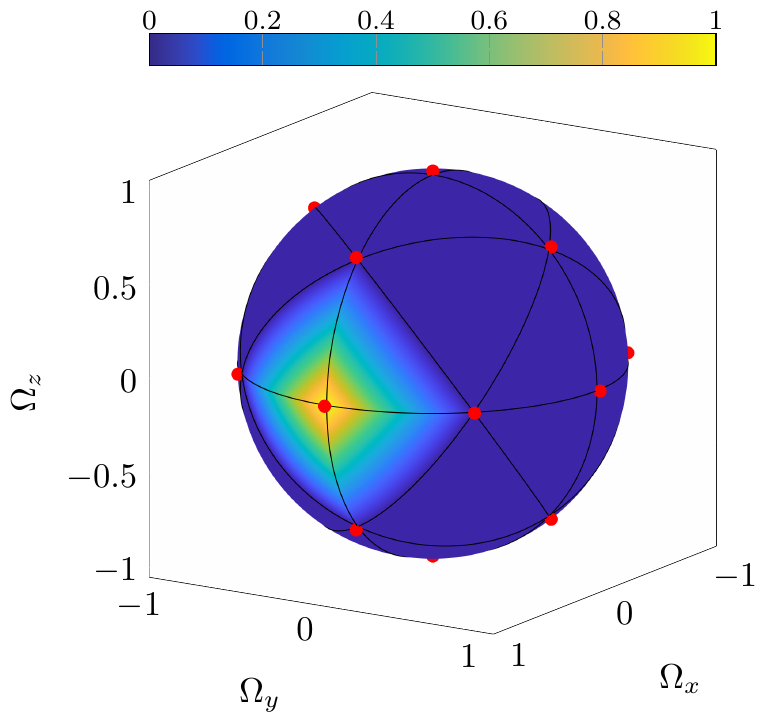}{\input{Images/HatFunctionsBC}}
 \centering
 \caption{One basis function of $\hfbasis[18]$ (one level of refinement). The function value is encoded in the color scale.}
 \label{fig:BCFunctions}
\end{figure}

As above (see \defnref{def:HFPNHFMN}), the corresponding moment models will be called $\HFPTh$ and $\HFMTh$, respectively.
If $\TriangulationSphere = \TriangulationSphere(\refinementnumber)$ is the triangulation obtained by $\refinementnumber$ dyadic refinements of
the octants, the models will be called $\HFPN[\nvertex]$ and $\HFMN[\nvertex]$, where $\nvertex = 4^{\refinementnumber+1} + 2$ is the number of vertices of the triangulation.

\subsection{Partial moments on the unit sphere}
Partial-moment models on the sphere have been introduced in reference \cite{Frank2006} (although the
authors restricted their investigation to quarter moments). We will only investigate first-order partial moments. Let $\sphericaltriangle$ be a spherical triangle. We then define
\begin{align*}
\pmbasis[\sphericaltriangle](\SC) = \begin{cases} \left(1,\SCx,\SCy,\SCz\right)^T &\text{ if } \SC \in \interior{\sphericaltriangle}, \\
    \mathbf{0}                               &\text{ if } \SC \in \sphere\setminus\sphericaltriangle,
                                    \end{cases}
\end{align*}
where $\interior{\sphericaltriangle}$ is the interior of $\sphericaltriangle$. As in one dimension, the value on the boundary of $\sphericaltriangle$ remains unspecified. We will denote the moments with respect to
$\pmbasis[\sphericaltriangle]$ as
$\moments[\sphericaltriangle] = \ints{\pmbasis[\sphericaltriangle]\distribution}$.
If $\moments[\sphericaltriangle] = {(\momentcomp{0}, \momentcomp{1}, \momentcomp{2}, \momentcomp{3})}^T$ we define $\momentstens[\sphericaltriangle]{0} = \momentcomp{0}$
          and $\momentstens[\sphericaltriangle]{1} = {(\momentcomp{1}, \momentcomp{2}, \momentcomp{3})}^T$.

The angular basis $\pmbasis[\TriangulationSphere]$ now consists of all element bases $\pmbasis[\sphericaltriangle]$, $\sphericaltriangle\in\TriangulationSphere$.
The local particle density is then given by $\density =
\sum\limits_{\sphericaltriangle\in\TriangulationSphere}\momentstens[\sphericaltriangle]{0}$.
As above (see \defnref{def:PPnPMn}), the corresponding moment models will be
called $\PMPTh$ and $\PMMTh$, respectively. For \TriangulationSphere = \TriangulationSphere(\refinementnumber), we will use the
notation $\PMPN$ and $\PMMN$, where $\momentnumber = 4 \cdot \abs{\TriangulationSphere(\refinementnumber)} = 2 \cdot 4^{\refinementnumber+2}$
 is the number of moments.

\begin{remark}
  As in one dimension (see Remark \ref{rem:basisspans1d}), our bases satisfy the following relation:
\begin{align*}
  \spanop(\fmbasis[1])\subseteq \spanop(\hfbasis[\TriangulationSphere]) \subseteq \spanop(\pmbasis[\TriangulationSphere]).
\end{align*}
\end{remark}

%% file: Sections/realizabilitySlab.tex
\section{Realizability}
\label{sec:Realizability}

Since the underlying kinetic density to be approximated is
non-negative, a moment vector only makes sense physically if it can be associated with a
non-negative density function. In this case, the moment vector is called \emph{realizable}.

For higher-order numerical methods, realizability becomes a major challenge. Generally, the system of moment
equations only evolves on the set of realizable moments. This property, though desirable since it is consistent with the original kinetic density,
can cause problems for numerical methods.
Standard high-order numerical solutions to the Euler equations, which indeed are an entropy-based moment closure, have been observed to have negative
local densities and pressures \cite{Zhang2010}. This can be overcome by using a realizability-preserving limiter
(see, e.g., \cite{Zhang2010,Schneider2015a,Chidyagwai2017,Olbrant2012,Schneider2016a}). However, this limiter requires exact
knowledge of the set of realizable moments which often poses severe problems.
In particular, for high moment orders in the full-moment context, calculating the limited approximation
is numerically intractable \cite{Schneider2015b}.

In this section, we will investigate the structure of the realizable set and derive realizability conditions for our models, showing that they are particularly simple in all cases,
which will later allow for an efficient high-order implementation in space and time for our minimum-entropy models $\HFMN$ and $\PMMN$.

\subsection{Structure of the realizable sets}
As already mentioned above, a moment vector is called realizable if it can be associated with a non-negative integrable density function.
The set of realizable moments is called the \emph{realizable set}.
\begin{definition}
\label{def:RealizableSet}
The \emph{(}$\basis$\emph{-)}\emph{realizable set} $\RD{\basis}{}$ is
\begin{equation}\label{eq:defrealizableset}
  \RD{\basis}{} = \left\{\moments \in \R^{\momentnumber} ~:~\exists \text{ integrable }\distribution(\SC) \ge 0, \, \density = \ints{\distribution} > 0,\,
 \text{ such that } \moments =\ints{\basis\distribution} \right\}.
 \end{equation}
 If $\moments\in\RD{\basis}{}$, then $\moments$ is called \emph{(}\basis\emph{-)}\emph{realizable}. Any $\distribution$ such that $\moments =\ints{\basis \distribution}$ is
 called a \emph{(}$\basis$\emph{-)}\emph{representing density}.
\end{definition}
\begin{remark}\begin{minipage}[t]{0.86\linewidth}
\begin{itemize}
\item[(a)] The realizable set is a convex cone, and
\item[(b)] representing densities are not necessarily unique.
\end{itemize}
\end{minipage}
\end{remark}
Unfortunately, the definition of the realizable set is not constructive, making it hard to check if a given moment vector is realizable or not.
In the minimum-entropy context, realizability is further complicated by the fact that the optimization problem \eqref{eq:primal} only has a
solution if the moment vector can be represented by a density of the form \eqref{eq:psiME}, i.e.\ if it lies in the ansatz set
\begin{align*}
 \AnsatzSpace := \left\{\ints{\basis \ansatz[\moments]}\stackrel{\eqref{eq:psiME}}{=} \ints{\basis \ld{\entropy}'\left(\basis^T\multipliers\right) }
  : \multipliers \in \R^{\momentnumber}  \right\}.
\end{align*}
As the ansatz density $\ld{\entropy}'\left(\basis^T\multipliers\right)$ is always strictly positive (assuming $\basis^T \multipliers(\moments) < 0$ in the Bose-Einstein case),
we immediately get the inclusions
\begin{align*}
  \AnsatzSpace \subset \RDpos{\basis}{} \coloneqq \left\{\moments  \in \R^{\momentnumber} ~:~\exists \text{ integrable }\distribution(\SC) > 0 \,
  \text{ such that } \moments =\ints{\basis\distribution} \right\} \subset \RD{\basis}{}.
\end{align*}
In general, these inclusions might be strict and the ansatz set $\AnsatzSpace$ might be considerably smaller than the realizable set $\RD{\basis}{}$. In that case, the minimum entropy
ansatz would not be a reasonable approximation to the original kinetic density. Luckily, in the case of a bounded angular domain, it was shown in \cite{Junk2000}
(for Maxwell-Boltzmann entropy) that for some bases, including the full moment basis $\fmbasis$, the minimum entropy ansatz does not further restrict the set of realizable moments.
\begin{theorem}[\cite{Junk2000}]
  For the full moment basis we have
\begin{align*}
  \AnsatzSpace[\fmbasis] = \RDpos{\fmbasis}{} = \RD{\fmbasis}{}.
\end{align*}
\end{theorem}
The piecewise linear bases do not satisfy all of the properties required from the basis in \cite{Junk2000}. More precisely, the assumption that
\begin{align}\label{eq:nullsetassumption}
  \{\Omega ~:~ \boldsymbol{\lambda} \cdot \basis(\Omega) = 0\} \text{ is a null set for any } \boldsymbol{\lambda} \in \R^n\setminus\{\mathbf{0}\}.
\end{align}
is not valid. However, in the following we will show that this is not a problem as the ansatz set is still at least the interior of the realizable set.
\begin{theorem}\label{thm:RealizableSetInclusions}
  Let $\basis = \hfbasis$ or $\basis = \pmbasis$. Then we have
\begin{align*}
  \AnsatzSpace = \RDpos{\basis}{} = \interior{\RD{\basis}{}}
\end{align*}
\end{theorem}
In particular, the ansatz set is dense in the realizable set, so each realizable moment vector can be accurately approximated by vectors in the ansatz set. However, note
that numerical difficulties may occur when approaching the boundary of the realizable set \cite{Alldredge2014}.

Before proving \thmref{thm:RealizableSetInclusions}, we will introduce a slightly more general realizable set based on distributions \cite{Schwartz1950}
instead of integrable functions which is often easier to work with for theoretical considerations. In the following, $\angularDomain$ will always denote the angular
domain ($[-1, 1]$ in slab geometry and $\sphere$ in three dimension).
\begin{definition}
  A \emph{distribution} on $V$ is a continuous linear functional mapping from the space
  $\mathcal{C}_c^{\infty}(\angularDomain)$ of compactly supported smooth functions on $\angularDomain$ to $\R$.
\end{definition}
For a more precise definition and details on the topologies involved see, e.g., \cite{Hoermander2015, Friedlander1998}. In our case, $\angularDomain$ is compact,
so all functions on $\angularDomain$ are compactly supported. Thus, for a smooth basis $\basis$ (e.g.\ the full moment basis), evaluation of the
distribution $\distribution$ at the components of $\basis$ is well-defined. The hatfunction and partial moment bases, however, are only piecewise smooth.
We thus need the notion of a \emph{piecewise distribution}.

\begin{definition}\label{def:piecewisedistribution}
  Let $\generalpartition = \{\generalpart[0], \ldots, \generalpart[\hankelhalfind-1]\}$ be a partition of the angular domain $\angularDomain$
  with closed subsets $\generalpart[\cellind]$ such that
  $\bigcup\limits_{\cellind=0}^{\hankelhalfind-1} \generalpart[\cellind] = \angularDomain$ and
  $\volume{\generalpart[i] \cap \generalpart[j]} = 0 \ \forall i, j$.
  A \emph{piecewise distribution on $\generalpartition$} is a collection $\distribution = (\distribution[0], \ldots, \distribution[\hankelhalfind-1])$ such that
  $\distribution[\cellind]$ is a distribution on $\generalpart[\cellind]$, $\cellind = 0, \ldots, \hankelhalfind-1$. We define the restriction of $\distribution$ to
  $\generalpart[\cellind]$ as $\distribution|_{\generalpart[\cellind]} \coloneqq \distribution[\cellind]$.
\end{definition}

In the following we will use $\generalpartition = \partition$ and $\generalpartition = \TriangulationSphere$ for the piecewise linear bases in one
and three dimension, respectively, and $\generalpartition = \{\angularDomain\}$ for the full moment basis.

We will abuse notation and write, for a distribution $\distribution$ on $\angularDomain$ and a smooth basis $\basis$, the dual pairing as
\begin{align*}
  \ints{\basis \distribution} = \ints{\distribution \basis} \coloneqq {\left(\ints{\distribution \basiscomp[0]}, \ldots, \ints{\distribution \basiscomp[\momentnumber]}\right)}^T
  \coloneqq {\left(\distribution(\basiscomp[0]), \ldots, \distribution(\basiscomp[\momentnumber])\right)}^T.
\end{align*}
Note that if $\distribution$ is an integrable function rather than a distribution, a distribution is defined by
\begin{align}\label{eq:funcisdistribution}
  b \mapsto \ints{\distribution b} = \int_{\angularDomain} \distribution b \ \ \text{ for } \ \ b \in \mathcal{C}_c^{\infty}(V),
\end{align}
which shows that the notation $\ints{\psi \basis}$ for distributions is just a generalization of the notation for integrable functions.
If $\distribution$ is a piecewise distribution and $\basiscomp[]$ is piecewise (with respect to $\generalpartition$)
a smooth function on the angular domain, we further generalize the notation and define $\ints{\psi\basiscomp[]} \coloneqq \sum_{\cellind=0}^{\hankelhalfind-1}
\distribution|_{\generalpart[\cellind]}(\basiscomp[]|_{\generalpart[\cellind]})$.

\begin{remark}\label{rem:pmbasisextension}
    Remember that we did not define the values of the partial moment bases on the boundary between two elements of the partition.
    If $\distribution$ is an integrable function, these values can be chosen arbitrarily without changing the integrals as the
    boundary is a null set. When working with (piecewise) distributions, however, pointwise values have to be
    taken into account. We thus define the restriction $\pmbasis[]|_{\generalpart[\cellind]}$ of $\pmbasis[]$ to
    $\generalpart[\cellind]$ as the continuous extension of $\pmbasis[]|_{\interior{\generalpart[\cellind]}}$ to $\generalpart[\cellind]$. In particular,
      $\pmbasis[]|_{\generalpart[i]}$ and $\pmbasis[]|_{\generalpart[j]}$ might differ on $\generalpart[i] \cap \generalpart[j]$.
\end{remark}
\begin{definition}
  A distribution $\distribution$ is called \emph{non-negative}, written $\distribution \geq 0$, if
  \begin{align*}
    \ints{\distribution b} \geq 0 \text{ for all non-negative } b \in \mathcal{C}_c^{\infty}(\angularDomain).
  \end{align*}
  A piecewise distribution is called non-negative if all its restrictions to the elements of the associated partition are non-negative.
\end{definition}

We are now ready to define the generalized realizable set.
\begin{definition}
  The \emph{generalized realizable set} \RDdist{\basis} is the set of moment vectors representable by non-negative (piecewise) distributions, i.e.,
\begin{align}\label{eq:realizablesetdist}
  \RDdist{\basis} \coloneqq \left\{\moments \in \R^n ~:~\exists \text{ (piecewise) distribution }\distribution \ge 0 \,
 \text{ such that } \moments =\ints{\basis\distribution} \right\}.
\end{align}
\end{definition}

Note that by \eqref{eq:funcisdistribution} it is clear that $\RD{\basis}{} \subset \RDdist{\basis}$.
\begin{example}
  Consider the partial moment basis $\pmbasis[4] = {(\indicator{[-1,0]}, \SCheight \indicator{[-1,0]}, \indicator{[0,1]}, \SCheight \indicator{[0,1]})}^T$
  and the moment vector $\moments = {(2, 0, 1, 0)}^T$. Then $\moments = \ints{\distribution \pmbasis[4]}$ with $\distribution|_{[-1,0]} = 2\dirac$ and $\distribution|_{[0,1]} = \dirac$
  and thus $\moments \in \RDdist{\pmbasis[4]}$. However, $\moments \not\in \RD{\pmbasis[4]}{}$ because a non-negative integrable function $\distribution$
with $\ints{\distribution}_{[-1,0]} = 2$ will always yield $\ints{\SCheight\distribution}_{[-1,0]} < 0$.
\end{example}

\begin{remark}
	For notational convenience, we assume that the Dirac distribution $\dirac$ has full mass also on the boundary of integration.
\end{remark}

If a (piecewise) distribution $\distribution$ is (piecewise) a linear combination of Dirac deltas \cite{Hassani2009, Kuo2006}, it is called \emph{atomic} \cite{Curto1991}.
For the monomial basis in one dimension, every moment $\moments \in \RDdist{\basis}$ realizable by a non-negative distribution can be represented by an atomic density \cite{Curto1991}.
Below, we will see that this is also true for our piecewise linear bases. We can use this to prove \thmref{thm:RealizableSetInclusions}.
\begin{proof}[Proof of \thmref{thm:RealizableSetInclusions}]
  The proofs for the identity $\AnsatzSpace = \RD{\basis}{}$ in \cite{Junk2000}, in particular the proof of \cite[Lemma 6.2]{Junk2000} using the null set assumption \eqref{eq:nullsetassumption},
  also work for the piecewise linear bases if $\RD{\basis}{}$ is replaced by $\RDpos{\basis}{}$, which gives $\AnsatzSpace = \RDpos{\basis}{}$. Further it was shown in \cite{Junk2000} that
\begin{align}\label{eq:momentmapdiffeo}
  \moments: \R^{\momentnumber} \to \AnsatzSpace,\ \multipliers \mapsto \moments(\multipliers) \coloneqq \ints{\basis \ld{\entropy}'\left(\basis^T\multipliers\right) }
  \ \text{ is a diffeomorphism}
\end{align}
(the proofs also hold for the piecewise linear bases). In particular, $\RDpos{\basis} = \AnsatzSpace = \moments(\R^{\momentnumber})$ is open, so it remains to show that
$\RDpos{\basis}$ is dense in $\RDdist{\basis}$ (then, by the inclusion $\RD{\basis}{} \subset \RDdist{\basis}$, $\RDpos{\basis}$ is also dense in $\RD{\basis}{}$).
   Let $\moments \in \RDdist{\basis}$. Then it has an atomic representing density $\distribution(\SC) = \sum_{i} a_i \dirac(\SC - \SC_i)$
   (see \cite{Curto1991}, \thmref{thm:realizabilityhf1d}, \thmref{thm:realizabilitypm1d} and the
   proofs of \thmref{thm:realizabilityhf3d} and \thmref{thm:realizabilitypm3d} for the $\fmbasis$, $\hfbasisIk$, $\pmbasisIk$, $\hfbasisTh$ and $\pmbasisTh$ basis, respectively).
   Approximating each Dirac delta by a Dirac sequence of positive integrable functions $g_{i,j}$ converging to $\dirac(\SC - \SC_i)$ (in the sense of distributions),
    we obtain the sequence of positive densities $\distribution[j] = \sum_i a_i g_{i,j}$ whose corresponding moments
  $\moments[j] = \ints{\basis \distribution[j]}$ converge to $\moments$.
  \end{proof}

  \begin{remark}
    The strictly positive Dirac sequences $g_{i,j}$ required in the proof of \thmref{thm:RealizableSetInclusions} can be obtained by choosing any of the
  compactly supported sequences given in \cite{Dang2012}, adding a constant $1/j$ on the domain of integration and renormalizing such that $\ints{g_{i,j}} = 1$ still
holds. By \cite[Theorem 1]{Dang2012} (originally from \cite{Kanwal1998}), the resulting sequence is again a Dirac sequence.
  \end{remark}

  To summarize, we have the following structure of the realizable sets:
  \begin{align}\label{eq:fullrealizablesetinclusions}
    \AnsatzSpace = \RDpos{\basis} = \interior{\RD{\basis}{}} \subset \RD{\basis}{} \subset \RDdist{\basis} \subset \closure{\RD{\basis}{}},
\end{align}
where $\closure{\RD{\basis}{}}$ is the closure of $\RD{\basis}{}$. For the bases regarded here, the last inclusion also is an equality,
as can be seen by the explicit descriptions provided in the following (see Theorems \ref{thm:FullMomentRealizability}, \ref{thm:realizabilityhf1d},
\ref{thm:realizabilitypm1d}, \ref{thm:realizabilityhf3d} and \ref{thm:realizabilitypm3d}).

Note that, in general,
$\RD{\basis}{}$ is neither open nor closed which makes it hard to give an explicit description of the realizable set.
Instead, realizability conditions for its interior $\RDpos{\basis}$ and its closure $\RDdist{\basis}$ will be given in the next subsections.

\subsection{Slab geometry}
Although completely solved, realizability conditions for the classical full-moment models
turn out to be rather complicated \cite{Curto1991,Laurent2005}.
\begin{theorem}[\cite{Curto1991}]
\label{thm:FullMomentRealizability}
Let $\moments\in\R^{\momentorder+1}$. Define the \emph{Hankel matrices}
\begin{align*}
\hankelA(\hankelhalfind):=\left(\momentcomp{i+j}\right)_{i,j=0}^\hankelhalfind, \quad
\hankelB(\hankelhalfind):=\left(\momentcomp{i+j+1}\right)_{i,j=0}^\hankelhalfind, \quad
\hankelC(\hankelhalfind):=\left(\momentcomp{i+j}\right)_{i,j=1}^\hankelhalfind.
\end{align*}
and the \emph{rank} of $\moments$ with respect to the full-moment basis $\fmbasis$ as
\begin{align*}
  \rank{\fmbasis}{\moments} =
\begin{cases}
  \min\left(\left\{ j ~:~ 1 \leq j \leq \hankelhalfind \text{ and } \hankelA(j) \text{ is singular } \right\} \cup \{\hankelhalfind+1\}\right)
  &\text{if } \momentorder = 2\hankelhalfind, \\
\rank{\fmbasis}{\left(\momentcomp{0}, \ldots, \momentcomp{2\hankelhalfind}\right)^T} &\text{if } \momentorder = 2\hankelhalfind+1.
\end{cases}
\end{align*}

Using the definition $\hankelA\geq \hankelB$ if and only if $\hankelA-\hankelB$ is positive semi-definite, the following are equivalent.
\begin{enumerate}
  \item $\moments \in \RDdist{\fmbasis}$
  \item $\begin{cases}
 \hankelA(\hankelhalfind)\geq \mathbf{0}, \hankelA(\hankelhalfind-1)\geq \hankelC(\hankelhalfind) &\text{if } \momentorder=2\hankelhalfind, \\
 \hankelA(\hankelhalfind)\geq \hankelB(\hankelhalfind),~\hankelA(\hankelhalfind)\geq -\hankelB(\hankelhalfind) &\text{if } \momentorder=2\hankelhalfind+1.
 \end{cases}$
\item There exists an atomic representing density with $\rank{\fmbasis}{\moments}$ atoms.
\end{enumerate}
\end{theorem}
``$1. \Leftrightarrow 2.$'' also holds for the realizable set $\RD{\fmbasis}{} = \RDpos{\fmbasis}$ if ``positive semi-definite'' is replaced by ``positive definite''.

Thus, to check whether a given moment vector is realizable with respect to $\fmbasis$, the definiteness of several Hankel matrices has to be tested. This is quite expensive numerically.
In addition, a severe drawback of full-, partial- and mixed-moment, minimum-entropy models of more than first order is that the resulting integrals in
the moment equations cannot be expressed in terms of elementary functions (not only in slab but also in the full geometry). This means that numerical quadrature is strictly
necessary to solve these equations. Unfortunately, this has a strong impact on the realizable set and therefore also on the solution
of \eqref{eq:MomentSystemUnclosed1D} \cite{Alldredge2012,Schneider2015a}, which further complicates testing realizability.

Fortunately, this does not hold for the first-order moment models where the integrals can always be evaluated. Moreover, even if the integrals are computed numerically,
the realizable set does not change if a suitable quadrature is chosen (except for the partial moments in three dimensions, see \secref{subsec:numericallyrealizableset}).
This allows to always use the analytical realizability conditions which will be derived in the following and which will turn out to be very easy to check
(see Corollaries \ref{cor:hf1drealizabilitypos}, \ref{cor:pm1drealizabilitypos}, \ref{cor:hf3drealizabilitypos}, \ref{cor:pm3drealizabilitypos}).

\subsubsection{Continuous piecewise-linear angular basis (Hat functions)}
In order to derive a theorem similar to \thmref{thm:FullMomentRealizability} for the hat function basis, we first have to introduce the rank of a moment vector with
respect to this basis. We will start by introducing a decomposition of non-negative vectors into vector which have a single block of positive entries.

\begin{definition}
   Given a non-negative vector $\moments \in {(\Rpos)}^{\momentnumber}$, define the two index sets (start and end indices, respectively)
    \begin{align*}
      \posblockstartset(\moments) ~&= \left\{j \in \{0, \ldots, \momentnumber-1\} ~|~ \momentcomp{j-1} = 0 \land \momentcomp{j} > 0\right\}, \\
      \posblockendset(\moments)   ~&=\left\{j \in \{0, \ldots, \momentnumber-1\} ~|~ \momentcomp{j} > 0 \land \momentcomp{j+1} = 0\right\},
    \end{align*}
    where $u_{-1} = u_{\momentnumber} \coloneqq 0$. Let $\posblockstartset(\moments) = \{s_0, \ldots, s_{L-1}\}$ and $\posblockendset(\moments) = \{e_0, \ldots, e_{L-1}\}$
    such that the indices $s_l$ and $e_l$ are sorted in ascending order. The \emph{positive blocks} of $\moments$ are then defined as
  \begin{align*}
    \momentsposblockcomp{l}{i} = \begin{cases} \momentcomp{i} & \text{if } s_l \leq i \leq e_l \\
    0 & \text{else} \end{cases}
  \end{align*}
  for $l = 0, \ldots, L-1$. We define the \emph{order} of the positive block $\momentsposblock{l}$ as $\posblockorder(\momentsposblock{l}) = e_l - s_l + 1$. The positive blocks give a unique
  decomposition
  \begin{align*}
    \moments = \sum_{l=0}^{L-1} \momentsposblock{l}
  \end{align*}
  of $\moments$ (where the empty sum is defined as $\mathbf{0}$).
\end{definition}

\begin{example}
  For $\moments = (0, 1, 1, 0, 1, 0, 1, 1)$ we have $\posblockstartset(\moments) = \{1, 4, 6\}$, $\posblockendset(\moments) = \{2, 4, 7\}$, $L=3$,
    $\momentsposblock{0} = (0, 1, 1, 0, 0, 0, 0, 0)$, $\momentsposblock{1} = (0, 0, 0, 0, 1, 0, 0, 0)$ and $\momentsposblock{2} = (0, 0, 0, 0, 0, 0, 1, 1)$
    with $\posblockorder{\momentsposblock{0}} = \posblockorder{\momentsposblock{2}} = 2$, $\posblockorder{\momentsposblock{1}} = 1$.
\end{example}

Based on this decomposition, we can now define the $\hfbasisIk$-rank of a moment vector.

\begin{definition}
  	Let $\moments = \left(\momentcomp{0},\ldots,\momentcomp{\hankelhalfind}\right)^T$ be given.
        If $\momentcomp{\basisind} \geq 0$ for all $\basisind = 0, \ldots, \hankelhalfind$ we define the \emph{rank} of $\moments$ with respect to the hat function basis by
	\begin{align*}
          \rank{\hfbasisIk}{\moments} = \sum_{l=0}^{L-1} \ceil{\frac{\posblockorder(\momentsposblock{l})}{2}}.
	\end{align*}
\end{definition}

We are now ready to state the hat function basis analogue of \thmref{thm:FullMomentRealizability}.

\begin{theorem}
\label{thm:realizabilityhf1d}
Let $\moments = \left(\momentcomp{0},\ldots,\momentcomp{\hankelhalfind}\right)^T$ be given. Then the following are equivalent:
\begin{enumerate}
	\item $\moments\in\RDdist{\hfbasisIk}$
	\item $\moments$ satisfies
	\begin{align}
	\label{eq:HFRealizabilityCondition}
	\momentcomp{\basisind}\geq 0,\quad \text{ for all }\basisind=0,\ldots,\hankelhalfind.
	\end{align}
      \item There exists an atomic $\hfbasisIk$-representing density with $\hankelrank \coloneqq \rank{\hfbasisIk}{\moments}$ atoms.
    \end{enumerate}
\end{theorem}
\begin{proof}
\begin{itemize}
\item[$1.\Rightarrow 2.$] Let $\moments\in\RDdist{\hfbasisIk}$. Thus there exists a representing distribution $\distribution\geq 0$.
Since by construction $\hfbasiscomp[\basisind]\geq 0$,
it holds that $\ints{\hfbasiscomp[\basisind]\distribution}\geq 0$. Therefore, \eqref{eq:HFRealizabilityCondition} is necessary.
        \item[$2.\Rightarrow 3.$] Let $\momentnumber = \hankelhalfind + 1$. We first handle the case that $\momentcomp{\basisind} > 0$ for all $\basisind$ and $\momentnumber = 2 m$
          is even. Thus, $\moments = \momentsposblock{0}$ with $\posblockorder(\momentsposblock{0}) = \momentnumber$. Define the atomic density
          \begin{align}\label{eq:hatfunctionsatomicdensity}
          \distribution(\SCheight) = \sum\limits_{\basisind=0}^{m-1} \left(\momentcomp{2\basisind+1}+\momentcomp{2\basisind}\right)
\dirac(\SCheight-\frac{\momentcomp{2\basisind+1}\SCheight_{2\basisind+1}+\momentcomp{2\basisind}\SCheight_{2\basisind}}{\momentcomp{2\basisind+1}+\momentcomp{2\basisind}}).
	\end{align}
	It is easy to check that
      $\frac{\momentcomp{2\basisind+1}\SCheight_{2\basisind+1}+\momentcomp{2\basisind}\SCheight_{2\basisind}}{\momentcomp{2\basisind+1}+\momentcomp{2\basisind}}\in
      \left(\SCheight_{2\basisind},\SCheight_{2\basisind+1}\right)$.
   Additionally, $\ints{\hfbasiscomp[\basisind]\distribution} = \momentcomp{\basisind}$, i.e.\ $\distribution$ is a realizing distribution with $m$ atoms.
The odd case $\momentnumber = 2m-1$ is a degenerated version of the even case (just by adding an interior node $\SCheight_{\basisind}$, $0<\basisind<\hankelhalfind$, twice). This results in
a non-uniqueness of the representing distribution (as the location of the additional node is free within the set of inner nodes).

Let now $\moments = \sum_l \momentsposblock{l}$. If $\posblockorder(\momentsposblock{l}) = 1$, a representing distribution with a
  single atom is given by $\momentcomp{s_l} \dirac(\SCheight -\SCheight_{s_l})$. For $\posblockorder(\momentsposblock{l}) \geq 2$, we can get a representing distribution
  of the form \eqref{eq:hatfunctionsatomicdensity} by ignoring the zero entries and treating $(\momentcomp{s_l}, \ldots, \momentcomp{e_l})$ as a vector with strictly positive
entries on the subpartition $\partition[\posblockorder(\momentsposblock{l})-1](\SCheight_{s_l}, \ldots, \SCheight_{e_l}) \subset \partition$.
A representing distribution for $\moments$ is then given by adding the representing distributions of its positive blocks.
	\item[$3.\Rightarrow 1.$] Trivial. \qedhere
\end{itemize}
\end{proof}
\begin{remark}
  The definition of $\RDdist{\basis}$ uses piecewise distributions (see \defnref{def:piecewisedistribution}), while we used distributions on the whole angular domain
  in the proof of \thmref{thm:realizabilityhf1d}. However, for atoms in the interior of an interval
  it is obvious how to get a piecewise distribution from the atomic distribution. Atoms on the interval boundaries can be arbitrarily assigned to one
  of the two distributions in the adjacent intervals due to the continuity of the hatfunction basis.
  In fact, since non-negative distributions are always of order 0 (see, e.g, \cite[Theorem 2.1.7]{Hoermander2015}) and thus can be evaluated on continuous functions,
  the definition of $\RDdist{\basis}$ by piecewise distributions is actually only required for the (discontinuous) partial moment bases.
\end{remark}
\begin{corollary}\label{cor:hf1drealizabilitypos}
  $\moments\in\RDpos{\hfbasisIk}$ if and only if $\momentcomp{\basisind} > 0$ for all $\basisind=0,\ldots,\hankelhalfind$.
\end{corollary}
\begin{proof}
  Follows directly from \thmref{thm:realizabilityhf1d} using the the fact that $\RDpos{\hfbasisIk}$ is the interior of $\RDdist{\hfbasisIk}$ (see \ref{eq:fullrealizablesetinclusions}).
\end{proof}

\begin{remark}
Another $\hfbasisIk$-representing distribution for $\moments$ is
\begin{align}
\label{eq:HFRepresentingDistribution}
\distribution(\SCheight) = \sum\limits_{\basisind=0}^{\hankelhalfind} \momentcomp{\basisind}\dirac(\SCheight-\SCheight_\basisind).
\end{align}
\end{remark}

\subsubsection{First-order partial moments}
For the partial moment basis, the rank is obtained by counting the intervals in the partition $\partition$ where the representing distribution has non-zero mass.
\begin{definition}
Let $\moments[{\pmbasisIk}]=\left(\momentcomp{0,0},\momentcomp{1,0},\ldots,\momentcomp{0,\hankelhalfind-1},\momentcomp{1,\hankelhalfind-1}\right)^T$ be given. We define its \emph{rank} with respect to the partial moment basis by
	\begin{align*}
          \rank{\pmbasisIk}{\moments} \coloneqq \abs{\left\{\cellind\in\{0,\ldots,\hankelhalfind-1\}~:~ \momentcomp{\cellind,0}>0 \right\}}.
	\end{align*}
\end{definition}
The partial moment analogue of \thmref{thm:FullMomentRealizability} is then easily obtained.
\begin{theorem}\label{thm:realizabilitypm1d}
Let $\moments = \left(\momentcomp{0,0},\momentcomp{1,0},\ldots,\momentcomp{0,\hankelhalfind-1},\momentcomp{1,\hankelhalfind-1}\right)^T \in \R^{2\hankelhalfind}$ be given. Then the following are equivalent:
\begin{enumerate}
	\item $\moments\in\RDdist{\pmbasisIk}$.
	\item $\moments$ satisfies
	\begin{align}
	\momentcomp{0,\cellind}>0 \quand \frac{\momentcomp{1,\cellind}}{\momentcomp{0,\cellind}}\in\ccell{\cellind}
        \quad \text{ or } \quad \momentcomp{0,\cellind} =\momentcomp{1,\cellind} = 0 \qquad \cellind=0,\ldots,\hankelhalfind-1.
	\end{align}
    \item There exists an atomic $\pmbasisIk$-representing density with $\hankelrank : = \rank{\pmbasisIk}{\moments}$ atoms.
  \end{enumerate}
\end{theorem}
\begin{proof}
  As the partial moments are just full moments of order 1 on each interval, this follows immediately from the results in \cite[Thm. 4.1+4.3]{Curto1991}.
\end{proof}
\begin{corollary}\label{cor:pm1drealizabilitypos}
  $\moments\in\RDpos{\pmbasisIk}$ if and only if
  \begin{align*}
    \momentcomp{0,\cellind}>0 \quand \frac{\momentcomp{1,\cellind}}{\momentcomp{0,\cellind}}\in\interior{\cell{\cellind}}, \qquad
       \cellind=0,\ldots,\hankelhalfind-1.
   \end{align*}
\end{corollary}
\begin{remark}
  Possible $\pmbasisIk$-representing densities for $\moments\in\RDpos{\pmbasisIk}$ include
\begin{align*}
\distribution|_{\cell{\cellind}} = \momentcomp{0,\cellind}\dirac(\SCheight-\frac{\momentcomp{1,\cellind}}{\momentcomp{0,\cellind}}).
\end{align*}
and
\begin{align}
\label{eq:PMRepresentingDistribution}
\distribution|_{\cell{\cellind}} = \frac{\SCheight_{\cellind+1}\momentcomp{0,\cellind} - \momentcomp{1,\cellind}}{\SCheight_{\cellind+1}-\SCheight_{\cellind}}
\dirac(\SCheight-\SCheight_{\cellind})
+ \frac{\momentcomp{1,\cellind} - \SCheight_{\cellind}\momentcomp{0,\cellind}}{\SCheight_{\cellind+1}-\SCheight_{\cellind}} \dirac(\SCheight-\SCheight_{\cellind+1})
\end{align}
\end{remark}

\subsection{Three dimensions}
We will now turn to the fully three-dimensional setting. Since the rank of a moment vector with respect to the piecewise linear bases in three dimensions
is much harder to define than in one dimension (among others, it depends non-trivially on the partition $\TriangulationSphere$ of the sphere)
we will restrict ourselves to deriving realizability conditions.

\subsubsection{Spherical barycentric}
\begin{theorem}
\label{thm:realizabilityhf3d}
Let $\moments = \left(\momentcomp{0},\ldots,\momentcomp{\nvertex-1}\right)^T\in\R^{\nvertex}$ be a vector of moments.
Then it is necessary and sufficient for the existence of a non-negative distribution $\distribution$ which realizes $\moments$ with respect to $\hfbasisTh$ that
\begin{equation}
\label{eq:RealizabilityBCa}
\momentcomp{\basisind}\geq 0 \qquad \text{ for all } \basisind = 0,\ldots,\nvertex-1.
\end{equation}
\end{theorem}
\begin{proof}
The proof works similar to the one-dimensional setting (compare \thmref{thm:realizabilityhf1d}).
We first show that equation \eqref{eq:RealizabilityBCa} is necessary. Let
$\distribution\geq 0$ be arbitrary. By \lemref{lem:BCproperties}(b) it holds
that $\hfbasiscomp[\basisind]\geq 0$ which implies that $\momentcomp{\basisind}
= \ints{\hfbasiscomp[\basisind]\distribution}\geq 0$ and thus
\eqref{eq:RealizabilityBCa}.

Assume that $\moments$ satisfies \eqref{eq:RealizabilityBCa}. A representing distribution is given by
\begin{equation}
\distribution = \sum\limits_{\basisind=0}^{\nvertex-1} \momentcomp{\basisind}\diracmd\left(\SC-\vertex[\basisind]\right).
\end{equation}
Due to \eqref{eq:BC1} it follows that $$\ints{\hfbasiscomp[\basisindvar]\distribution} = \sum\limits_{\basisind=0}^{\nvertex-1} \momentcomp{\basisind}\hfbasiscomp[\basisindvar]\left(\vertex[\basisind]\right) \stackrel{\eqref{eq:BC1}}{=} \sum\limits_{\basisind=0}^{\nvertex-1} \momentcomp{\basisind} \dirac_{\basisind\basisindvar} = \momentcomp{\basisindvar}.$$
\end{proof}

\begin{remark}
We want to note that the above proof is not limited to barycentric coordinates but also works for other bases with values in $[0,1]$ and the Lagrange interpolation property \eqref{eq:BC1}.
\end{remark}

\begin{corollary}\label{cor:hf3drealizabilitypos}
  $\moments\in \RDpos{\hfbasisTh}$ if and only if
  \begin{align*}
    \momentcomp{\basisind} > 0 \qquad \text{ for all } \basisind = 0,\ldots,\nvertex-1.
\end{align*}
\end{corollary}

\subsubsection{Partial moments}
\begin{theorem}
\label{thm:realizabilitypm3d}
Let $\moments = \left(\momentstens[\sphericaltriangle]{0},\momentstens[\sphericaltriangle]{1}\right)^T\in\R^4$ be a vector of moments. Then it is necessary and sufficient for the existence of a
non-negative distribution $\distribution\not\equiv 0$ which realizes $\moments$ with respect to $\pmbasis[\sphericaltriangle]$ that
\begin{subequations}
\label{eq:RealizabilityPartialMoments}
\begin{equation}
\label{eq:QuarterMomentsRealizableU0}
\momentstens[\sphericaltriangle]{0}>0
\end{equation}
and the normalized first moment
\begin{equation}
\normalizedmomentstens[\sphericaltriangle]{1} := \frac{\momentstens[\sphericaltriangle]{1}}{\momentstens[\sphericaltriangle]{0}}\label{eq:NormalizedMoments}
\end{equation}
\end{subequations}
satisfies $\normalizedmomentstens[\sphericaltriangle]{1} \in \convexhull{\sphericaltriangle}$, where $\convexhull{\cdot}$ denotes the convex hull.
\end{theorem}
\begin{proof}
The necessity of \eqref{eq:RealizabilityPartialMoments} follows immediately, since by assumption
$\distribution\geq 0$ is supported in $\sphericaltriangle$ and integration is a linear operation.

To show the sufficiency of \eqref{eq:RealizabilityPartialMoments} we give a realizing distribution
function under this assumption. The boundary of the convex hull $\convexhull{\sphericaltriangle}$ is
given by $\sphericaltriangle$ and $\flattriangle$. First, we give realizing distributions on these
boundary parts. Due to the convexity of the convex hull and the linearity of the integral, every
interior point can then be reproduced by a suitable convex combination of these two candidates.

Assume that $\norm{\normalizedmomentstens[\sphericaltriangle]{1}}{2} = 1$, i.e. $\normalizedmomentstens[\sphericaltriangle]{1}\in \sphericaltriangle$. A realizing distribution is given by
\begin{align}
\label{eq:psiUp}
\distribution[\sphericaltriangle] = \momentstens[\sphericaltriangle]{0}\diracmd\left(\SC-\normalizedmomentstens[\sphericaltriangle]{1}\right),
\end{align}
where $\diracmd$ denotes the multi-dimensional Dirac-delta distribution. By
assumption, $\distribution[\sphericaltriangle]$ is supported in $\sphericaltriangle$ (this
is no longer true if $\norm{\normalizedmomentstens[\sphericaltriangle]{1}}{2}< 1$).  Thus
\begin{align*}
\ints{\distribution[\sphericaltriangle]} = \momentstens[\sphericaltriangle]{0}\quand \ints{\SC\distribution[\sphericaltriangle]} = \momentstens[\sphericaltriangle]{0}\normalizedmomentstens[\sphericaltriangle]{1} \stackrel{\eqref{eq:NormalizedMoments}}{=} \momentstens[\sphericaltriangle]{1}.
\end{align*}

If $\normalizedmomentstens[\sphericaltriangle]{1}\in \flattriangle$ (the opposite boundary of the convex hull), a realizing distribution is
\begin{align}
\label{eq:psiLow}
\distribution[\flattriangle] = \momentstens[\sphericaltriangle]{0}\left(\bcA\diracmd\left(\SC-\vertexA\right)+\bcB\diracmd\left(\SC-\vertexB\right)+\bcC\diracmd\left(\SC-\vertexC\right)\right),
\end{align}
where $\bcA$, $\bcB$, $\bcC$ are the barycentric coordinates of $\normalizedmomentstens[\sphericaltriangle]{1}$ with respect to the
vertices $\vertexA$, $\vertexB$ and $\vertexC$ on $\flattriangle$, i.e. $\bcA+\bcB+\bcC = 1$ and $\bcA\vertexA+\bcB\vertexB+\bcC\vertexC = \normalizedmomentstens[\sphericaltriangle]{1}$.
It immediately follows that $\distribution[\flattriangle]$ is supported in $\sphericaltriangle$ and realizes $\moments$.

If $\normalizedmomentstens[\sphericaltriangle]{1}$ is in the interior of the convex hull, we find
that $\frac{\normalizedmomentstens[\sphericaltriangle]{1}}{\norm{\normalizedmomentstens[\sphericaltriangle]{1}}{2}}$ can be realized by \eqref{eq:psiUp} and the intersection point of
the line from the origin $(0,0,0)$ and $\frac{\normalizedmomentstens[\sphericaltriangle]{1}}{\norm{\normalizedmomentstens[\sphericaltriangle]{1}}{2}}$ with $\flattriangle$ can be
realized by \eqref{eq:psiLow}. Certainly, $\normalizedmomentstens[\sphericaltriangle]{1}$ lies on the same line and can thus be realized by a convex combination of these two distribution functions.
\end{proof}
\begin{remark}
	Note that, similarly to the one-dimensional case, the $\pmbasisTh$-representing non-negative piecewise distribution for a moment vector $\moments$ might have
zero mass on some of the spherical triangles. In this case $\momentstens[\sphericaltriangle]{0} = 0$ and $\momentstens[\sphericaltriangle]{1} = \mathbf{0}$.
\end{remark}

\begin{corollary}\label{cor:pm3drealizabilitypos}
  $\moments\in \RDpos{\pmbasisTh}$ if and only if
  \begin{align*}
    \momentstens[\sphericaltriangle]{0}>0  \, \text{ and } \,
    \normalizedmomentstens[\sphericaltriangle]{1} \in \interior{\convexhull{\sphericaltriangle}}
    \, \text{ for all } \, \sphericaltriangle \in \TriangulationSphere.
\end{align*}
\end{corollary}

\subsection{Numerically realizable set}
\label{subsec:numericallyrealizableset}
In general, we cannot solve the integrals in \eqref{eq:defrealizableset} analytically and have to approximate them by a numerical quadrature
$\angularQuadrature$. For the full moments, this can have a severe impact on the realizable set \cite{Alldredge2012,Schneider2015a}. In the following,
we will show that for the piecewise linear bases, except for the partial moments in three dimensions, these problems do not occur
if a suitable quadrature is chosen.

Let $\generalpartition = (\generalpart[0], \ldots, \generalpart[\hankelhalfind-1])$ be a partition of the
angular domain $\angularDomain$ in closed subsets as in \defnref{def:piecewisedistribution}
and let $\angularQuadrature = (\angularQuadrature_{0}, \ldots, \angularQuadrature_{\hankelhalfind-1})$ such that
$\angularQuadrature_{\cellind} =
\{(\angularQuadratureWeights{\angularQuadratureIndex}^{\cellind}, \angularQuadratureNodesS{\angularQuadratureIndex}^{\cellind})\}_{\angularQuadratureIndex=0}^{\angularQuadratureNumber_{\cellind}-1}$
is a quadrature on $\generalpart[\cellind]$, i.e.,
for an integrable function $\basiscomp[]$
\begin{align*}
  \intQ{\basiscomp[]} \coloneqq
  \sum_{\cellind=0}^{\hankelhalfind-1} \sum_{\angularQuadratureIndex=0}^{\angularQuadratureNumber_{\cellind}-1}
  \angularQuadratureWeights{\angularQuadratureIndex}^{\cellind}\basiscomp[]|_{\generalpart[\cellind]}(\angularQuadratureNodesS{\angularQuadratureIndex}^{\cellind})\approx\ints{\basiscomp[]}.
\end{align*}
is the approximation of the corresponding integral $\ints{\cdot}$ with the quadrature rule $\angularQuadrature$.

We define the numerically realizable set
\begin{equation}\label{eq:numericallyrealizableset}
\RQ{\basis} = \left\{\moments~:~\exists \distribution(\SC)\ge 0 \, \text{ such that } \moments =\intQ{\basis\distribution} \right\} \subset
\RDdist{\basis},
\end{equation}
Note that pointwise values of $\basis$ have to be available for \eqref{eq:numericallyrealizableset} to be well-defined. For the partial moment bases,
we thus use the continuous extension from $\interior{\generalpart[\cellind]}$ to $\generalpart[\cellind]$ on each element of the partition (see Remark \ref{rem:pmbasisextension}).

In general, $\RQ{\basis}$ is a strict subset of \RDdist{\basis}, i.e., there are some moments that are realizable analytically but cannot be realized with $\angularQuadrature$. For
the hat functions and one-dimensional partial moments, however, the numerically realizable set and the generalized realizable set agree for suitable quadratures.

\begin{theorem}
	\label{thm:RQeqRD}
	Let $\basis$ be an angular basis in one dimension with $\momentnumber$ piece-wise linear (possibly discontinuous) basis functions, i.e., there
        exists a partition $\partition(\SCheight_0, \ldots, \SCheight_{\hankelhalfind})$ of $[-1,1]$ (see \secref{sec:HatFunctions1D}) such
        that the restriction on the interval $\cell{\basisindvar}$
        satisfies $\left.\basiscomp[\basisind]\right|_{\ccell{\basisindvar}}\in\SpaceOfPolynomials{1}(\ccell{\basisindvar})$
          for all $\basisindvar=0,\ldots,\hankelhalfind-1$ and $\basisind=0,\ldots,\momentnumber-1$.
          Additionally, let $\SCheight_\cellind$ and $\SCheight_{\cellind+1}$ be part of the node set of the numerical quadrature $\angularQuadrature_{\cellind} =
          \{(\angularQuadratureWeights{\angularQuadratureIndex}^{\cellind},
        \angularQuadratureNodes{\angularQuadratureIndex}^{\cellind})\}_{\angularQuadratureIndex=0}^{\angularQuadratureNumber_{\cellind}-1}$
on the interval $\cell{\cellind}$,
i.e., $\SCheight_\cellind, \SCheight_{\cellind+1} \in\{\angularQuadratureNodes{\angularQuadratureIndex}^{\cellind}~|~\angularQuadratureIndex=0,\ldots,\angularQuadratureNumber_{\cellind}-1\}$.
Then we have
	\begin{align}\label{eq:QRealizable}
	\RQ{\basis} = \RDdist{\basis}.
    \end{align}
\end{theorem}
\begin{proof}
  It is straight-forward to show that (generally) $\RQ{\basis}\subset \RDdist{\basis}$.
	Thus let $\moments\in \RDdist{\basis}$. Since the basis is piecewise linear, we can
        find a representing density $\ansatz$ such that $\ansatz|_{\ccell{\cellind}}(\SCheight) = \multiplierscomp{0,\cellind}\dirac(\SCheight-\SCheight_\cellind)
        + \multiplierscomp{1,\cellind}\dirac(\SCheight-\SCheight_{\cellind+1})$ with suitably
chosen $\multiplierscomp{0,\cellind}, \multiplierscomp{1,\cellind}\in\Rpos$ (compare \eqref{eq:HFRepresentingDistribution} and \eqref{eq:PMRepresentingDistribution}). Setting
	\[
          \distribution[\angularQuadratureIndex]|_{\ccell{\cellind}}(\SCheight) = \begin{cases}
            \frac{\multiplierscomp{0,\cellind}}{\angularQuadratureWeights{\angularQuadratureIndex}}
            &\text{ if } \SCheight = \SCheight_\cellind$ (where $\basisind$ is chosen such that $\SCheight_\cellind=\angularQuadratureNodes{\angularQuadratureIndex}),\\
            \frac{\multiplierscomp{1,\cellind}}{\angularQuadratureWeights{\angularQuadratureIndex}}
            &\text{ if } \SCheight = \SCheight_{\cellind+1}$ (where $\basisind$ is chosen such that $\SCheight_{\cellind+1}=\angularQuadratureNodes{\angularQuadratureIndex}),\\
	0 &\text{ else},
	\end{cases}
	\]
      shows that $\moments\in \RQ{\basis}$.
\end{proof}

With a very similar proof, it can be shown that $\RQ{\hfbasisTh}=\RDdist{\hfbasisTh}$ also in
three dimensions if the quadrature contains the vertices $\vertex[\basisind]$ of the triangulation.

For the partial moment basis in three dimensions, there is a numerical
equivalent to the analytical realizability conditions \eqref{eq:RealizabilityPartialMoments}:

\begin{theorem}
	\label{thm:realizabilitypm3dnumerical}
	Let $\moments =
	\left(\momentstens[\sphericaltriangle_\cellind]{0},\momentstens[\sphericaltriangle_\cellind]{1}
        \right)^T\in\R^4\setminus \{\mathbf{0}\}$ be a vector of moments, and let
        $\angularQuadrature_{\cellind}$ be a quadrature on
	$\sphericaltriangle_\cellind$.
        Then $\moments \in \RQ{\pmbasis[\sphericaltriangle_{\cellind}]}$ if and only if
        $\momentstens[\sphericaltriangle_\cellind]{0}>0$
	and
	\begin{equation}
	\normalizedmomentstens[\sphericaltriangle_\cellind]{1} \coloneqq
	\frac{\momentstens[\sphericaltriangle_\cellind]{1}}{\momentstens[\sphericaltriangle_\cellind]{0}}
	\in \convexhull{\{\SC_{\angularQuadratureIndex} ~:~ \SC_{\angularQuadratureIndex} \in \angularQuadrature_\cellind \}} .
	\end{equation}
\end{theorem}

\begin{proof}
  In \cite[Proposition 1]{Alldredge2014} it was shown that the set of numerically realizable moments with density $\density=1$ is the convex hull of the basis evaluations,
\begin{equation}
\RQ{\basis}|_{\density = 1} =
\convexhull{\{\basis(\SC_\angularQuadratureIndex)\}_{
		\angularQuadratureIndex=0}^{ \angularQuadratureNumber-1}}.
\end{equation}
Applying this to $\frac{\moments}{\momentstens[\sphericaltriangle_\cellind]{0}}$ and using the definition of $\pmbasis[\sphericaltriangle_\cellind]$ gives the result.
\end{proof}

Note that \thmref{thm:realizabilitypm3dnumerical} shows that the analytical and numerical realizable
set differ for the three-dimensional partial moments as the quadrature would have to contain all
points on $\sphericaltriangle$ to reproduce the analytical realizability conditions.

%% file: Sections/results.tex
\section{Results}

\label{sec:results}
To test the approximation properties of the regarded models, we investigate convergence against
some prescribed density function $\distribution$.

For each basis $\basis$, the moment vector $\moments$ is calculated as $\moments = \ints{\basis \distribution}$. Then the
ansatz densities are computed by \eqref{eq:psiME}. We use the backtracking Newton scheme from
\cite{Alldredge2014} (without change of basis). As an initial guess, we use the multipliers for the
isotropic moment. Integrals are calculated using very fine quadratures to avoid numerical errors. In
one dimension, for the $\MN$ models, the domain $[-1, 1]$ is partitioned in $200$ equally spaced
intervals and a Gauss-Lobatto quadrature of order 197 is used on each interval. The $\HFMN$ and
$\PMMN$ models already come with the partition $\partition$. We further subdivide $\partition$
such that at least $200$ subintervals are available, then we use the same Gauss-Lobatto
quadrature on each of the subintervals. In three dimension, we use a Fekete quadrature
\cite{Taylor2000} of order $18$ on each spherical triangle on the most refined spherical
triangulation for all models (including the $\MN$ models).
\vspace{-5pt}

We follow the FAIR guiding principles for scientific research \cite{Wilkinson2016} and publish the
code that generates the following results in \cite{SchneiderLeibner2019DataversePart1}. {The
code allows to choose between Maxwell-Boltzmann and Bose-Einstein entropy. As the results
with both entropies are qualitatively the same, we are only presenting the Maxwell-Boltzmann
results in the following.}
\vspace{-6pt}
\subsection{Slab geometry}
In one dimension, we prescribe a Gaussian function
\begin{align}
 \distribution[\text{Gauss}](\mu) &= \frac{1}{\sqrt{2\pi\sigma^2}}
\exp\left(\frac{(\mu-\bar{\mu})^2}{-2 \sigma^2}\right) \text{ where } \sigma = 0.5,
\bar{\mu} = 0
\end{align}
as a smooth test case and a Heaviside function
\begin{align}
 \distribution[\text{Heaviside}](\mu) &= \begin{cases} \psi_{vac}(\mu) &\text{if } \mu
< 0 \\ 1 & \text{else} \end{cases} \hspace{0.5cm} \text{ where } \hspace{0.5cm} \psi_{vac}(\mu)
= \frac{10^{-8}}{2}
\end{align}
as a discontinuous test. Note that the minimum-entropy ansatz \eqref{eq:psiME} is
  always strictly positive. The minimum entropy ansatz thus cannot represent density functions that are
  exactly zero in parts of the domain. Moreover, for very small (positive) values the multipliers $\multipliers$ tend to
infinity (in absolute values) which makes it numerically infeasible to solve the
dual optimization problem \eqref{eq:dual}. We thus prescribe an isotropic ``vacuum'' density
$\psi_{vac}$ (with local particle density $\density = 10^{-8}$) as a minimum instead of zero.

{As a third test case, we use}
\begin{align}
{ \distribution[\text{CrossingBeams}](\mu)} &= { \sqrt{\frac{a}{\pi}}
\left(\exp\left(-a(\mu+1)^2\right) + \exp\left(-a(\mu-0.5)^2\right) \right) \text{ where }
a = 10^3}
\end{align}
{which describes two crossing beams of particles with velocities close to $-1$ and $0.5$,
respectively. This is a hard test case for moment models, as it approaches the sum of two
Dirac distributions. In particular, full-moment models of low order typically fail (e.g. $\MN[1]$
produces the isotropic solution) in this specific setup} \cite{Frank2006,Schneider2014}.

Exemplary ansatz functions are shown in Figures \ref{fig:GaussApproximations},
\ref{fig:HeavisideApproximations} and \ref{fig:CrossingBeams1dApproximations}. The convergence
results can be found in \figref{fig:DensityApproximations1d}. Further plots for
all tested models and tabulated errors are provided in the supplementary materials.

\begin{figure}
	\centering
\externaltikz{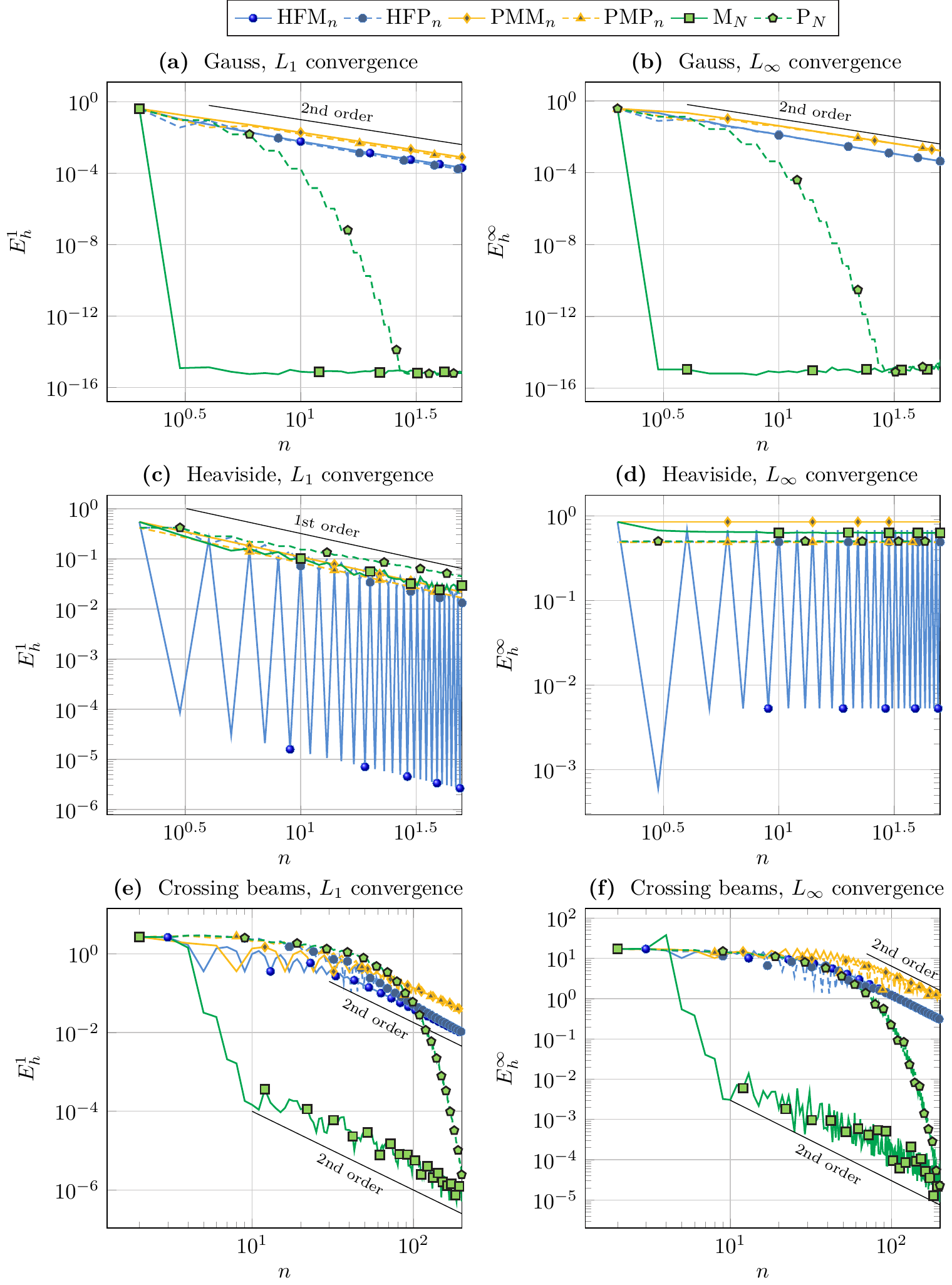}{\input{Images/DensityApproximations}}
\caption{Approximation errors for prescribed kinetic densities in slab geometry. $\PMMN[4k]$
and $\PMPN[4k]$ models are not included in the Heaviside plots because they are exact.}
	\label{fig:DensityApproximations1d}
\end{figure}

For the Gauss function, as expected \cite{Hesthaven2008}, the $\PN$ models show exponential
convergence whereas the $\MN$ models are exact (up to numeric errors) starting from order $2$. This
is due to the fact that the Gaussian can be written as \eqref{eq:psiME} if $\momentorder\geq 2$. For
the piecewise linear models, second-order convergence is expected \cite{Boor1968, Boor1979}, which
is confirmed by the results both in $\Lp{1}$ norm and $\Lp{\infty}$ norm. The $\HFMN$ and
$\PMMN$ approximations show similar shapes, dependent on the location of the ansatz intervals. For
models with an odd number of intervals the peak of the Gauss function lies in the middle of an
ansatz interval and is approximated by a constant value in that interval (due to its symmetry).
These models thus show a plateau around zero. In contrast, models using an even number of intervals
show a maximum at zero. In this test case, the linear $\HFPN$ and $\PMPN$ models do not differ much
from their non-linear counterparts.
\begin{figure}[htpb]
	\centering
	\externaltikz{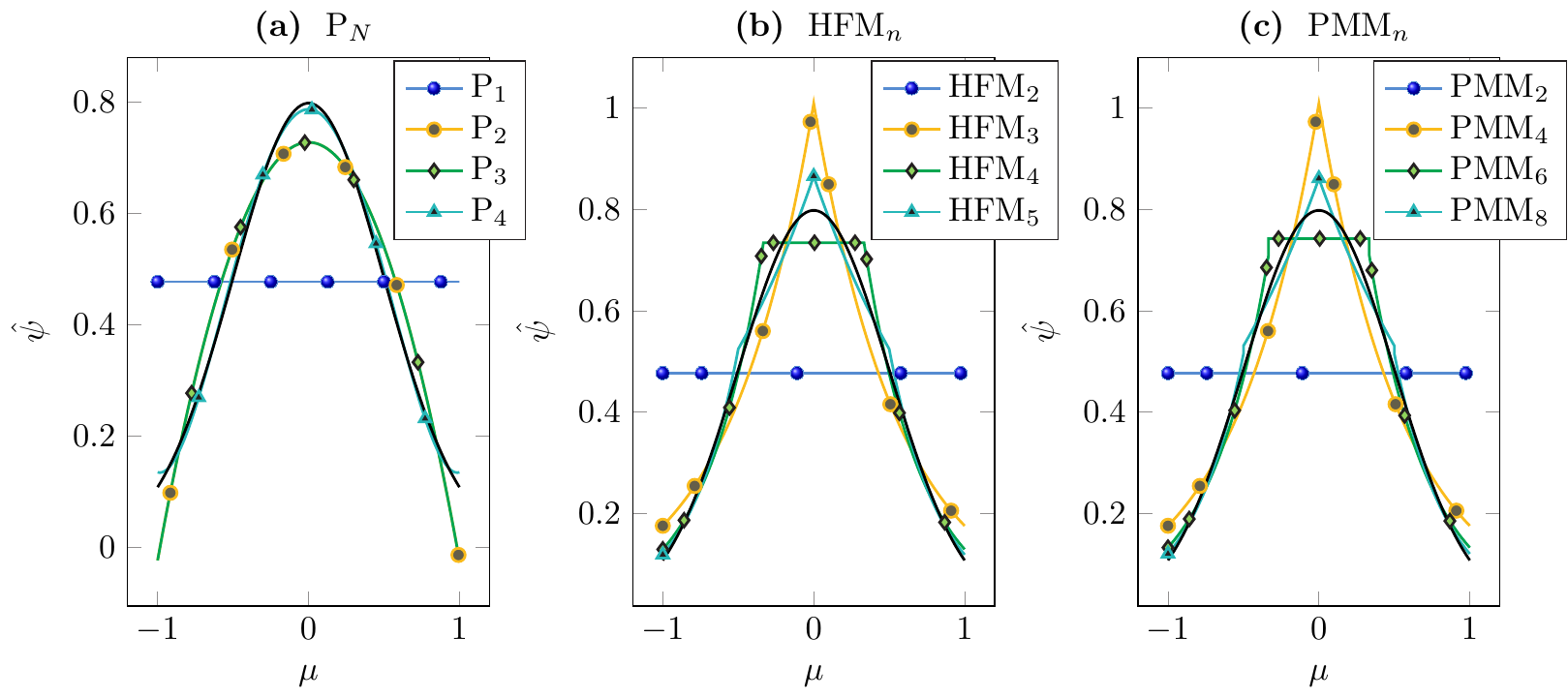}{\input{Images/GaussApproximations}}
        \caption{Ansatz densities of selected models for the slab geometry Gauss test. $\MN[1]$ is equal
to $\PMMN[2]$, higher order $\MN$ models are exact. The linear $\PMPN$ and $\HFPN$ models give
results similar to their non-linear counterparts. Reference is plotted as solid black line.}
	\label{fig:GaussApproximations}
\end{figure}

In the Heaviside test, the $\PMPN$ and $\PMMN$ models are exact
if an even number of intervals is used (such that the jump occurs between two ansatz intervals).
All other tested models and the $\PMPN$ and $\PMMN$ models for an odd number of intervals show the
expected first-order convergence in $\Lp{1}$ norm \cite{kuznetsov1977stable,Tadmor1991}. Notably,
the $\HFMN$ models with an even number of intervals give errors that are several orders of magnitude
lower than the other models though the rate of convergence is not improved. In $\Lp{\infty}$ norm,
no convergence can be observed in this testcase, due to the well-known Gibbs phenomenon (see
\figref{fig:HeavisideApproximations}). The $\MN$ and $\PN$ models both show strong
oscillations around the reference solution. For higher orders, the frequency of the oscillations
increases while their amplitude decreases, except at the discontinuity where the overshoot
approximately stays the same. While the $\PN$ models are symmetric, the $\MN$ models mainly
oscillate around the positive part of the Heaviside function due to their inherent positivity.
The $\PMMN$ and $\PMPN$ models with an odd number of intervals are exact on all intervals except the one containing the discontinuity. In
that interval, the $\PMMN$ models show a strong overshoot and the $\PMPN$ models show a symmetric
under-/overshoot at both sides of the discontinuity. The $\HFMN$ and $\HFPN$ models show similar
behavior, except that the under-/overshoots give rise to small oscillations in the other intervals
due to the models' continuity.
\begin{figure}[htbp]
	\centering
	\externaltikz{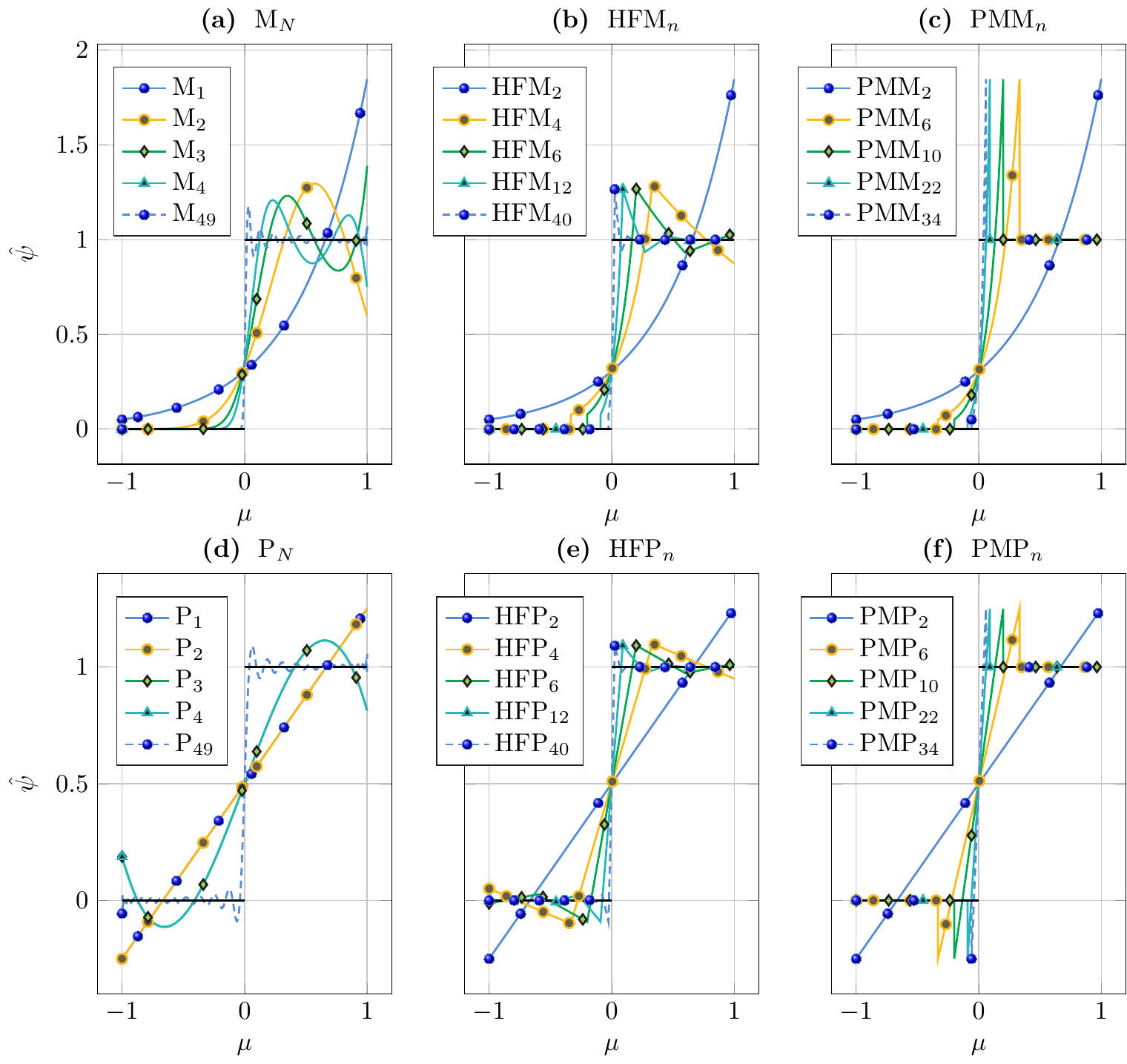}{\input{Images/HeavisideApproximations}}
        \caption{Ansatz densities of selected models for the slab geometry Heaviside test. Reference is
plotted as solid black line.}
	\label{fig:HeavisideApproximations}
\end{figure}

For the two crossing beams, approximation quality of the $\MN$ models initially rapidly
increases with the moment order. $\MN[1]$ and $\MN[2]$ fail, giving (nearly) isotropic
ansatz densities, while the $\MN[3]$ solution starts to show features of the reference solution.
$\MN[4]$ is already quite close to the reference solution. For even higher orders, the approximation
error decreases with second order in $n$. The first-order models struggle in this test case.
Depending on the location of the ansatz intervals, the beam centered at $\SCheight = 0.5$ is
approximated either flattened out, sharpened or skewed to the left or right. The beam around
$\SCheight = -1$ always is at the boundary of an interval and thus at least qualitatively met.
Only for high moment numbers (starting at around $\momentnumber = 50$), the approximation quality
significantly increases (with second order). Again, the linear $\HFPN$ and $\PMPN$ exhibit similar
behavior. The $\PN$ models show significant oscillations and perform even worse than the
piecewise first-order models for a large range of $\momentnumber$. For very high orders, given
that the crossing beams density regarded here is anisotropic but still smooth, the
exponential convergence finally kicks in and the approximation errors rapidly decrease.

\begin{figure}[htpb]
	\centering
	\externaltikz{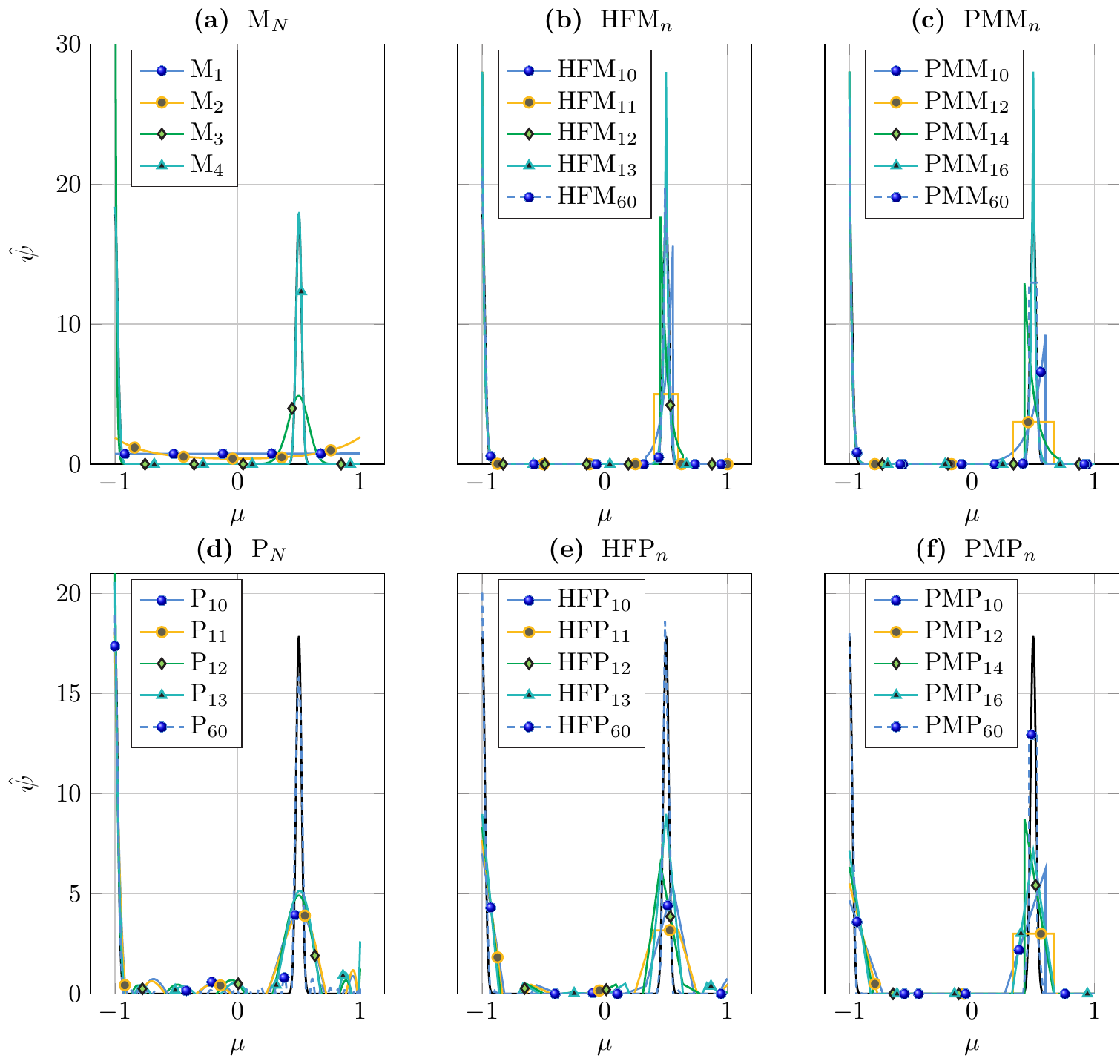}{\input{Images/CrossingBeams1dApproximations}}
        \caption{Ansatz densities of selected models for the slab geometry crossing beams test.}
	\label{fig:CrossingBeams1dApproximations}
\end{figure}

\pagebreak

\subsection{Three dimensions}
In three dimensions, the tests again include a smooth Gaussian
\begin{align}
 \distribution[\text{Gauss}](\SC) &= \frac{1}{2 \pi \sigma^2 \left( 1 -
\exp\left(-2 \sigma^{-2}\right)\right)}
\exp\left(\frac{(\SCx-\overline{\SCx})^2 + \SCy^2 + \SCz^2}{-2 \sigma^2}\right) \text{ where }
\sigma = 0.5,
\overline{\SCx} = 1,
\end{align}
a discontinuous indicator function
\begin{align}
 \distribution[\text{Square}](\SC)
 &=
 \begin{cases}
 1 &\text{if } \SCx > 0 \land \abs{\SCy} < 0.5 \land \abs{\SCz} < 0.5 \\
 \distribution[vac](\SC)  &\text{else,}
\end{cases}
\hspace{0.3cm} \text{ where } \hspace{0.3cm} \distribution[vac](\Omega) = \frac{10^{-8}}{4\pi},
\end{align}
{and a density function modelling two crossing beams of particles}
\begin{align}
{ \distribution[\text{CrossingBeams}](\SC) }
 &=
 { \max\left(\frac{a}{\pi}
 \left(\exp(-a \norm{\SC - (1, 0, 0)^T}{2}^2) + \exp(-a \norm{\SC - (0, 1, 0)^T}{2}^2)\right),
\distribution[vac](\SC) \right),}
\end{align}
{where $a = 100$.}

The convergence results are depicted in \figref{fig:DensityApproximations3d}. For completeness, we
also show a selection of ansatz functions for Gauss (\figref{fig:DensityApproximations3dGauss}),
Square (\figref{fig:DensityApproximations3dSquare}) and Crossing Beams test
(\figref{fig:DensityApproximations3dCrossingBeams}). Ansatz functions for all tested models and
tabulated convergence results can be found in the supplementary materials.

The smooth Gauss function is reproduced exactly by all $\MN$ models. This is
different to the slab geometry case since $\SC_x^2+\SC_y^2+\SC_z^2 = 1$ and therefore this Gaussian
on the unit sphere is actually just a first-order ansatz. Likewise, the $\PMMN$ models are exact.
Due to quadrature errors and errors in solving the optimization problem, the actual computed error
is slightly higher than numerical accuracy, which is especially visible for $\MN[7]$ and $\MN[8]$
($64$ and $81$ moments, respectively). The $\HFMN$ ansatz is not able to model the Gaussian
function exactly and results in a rather square-shaped ansatz density for the lower order
models. The $\PN$ models show exponential convergence. For the piecewise linear models, quadratic
convergence in the grid width is expected \cite{Hesthaven2008}. This corresponds to linear
convergence in $\momentnumber$, as the grid width approximately halves with each dyadic refinement,
but $\momentnumber = 2 \cdot 4^{k+2}$ after $k$ refinements for the partial moment basis and
$\momentnumber = 2 + 4^{k+1}$ for the hatfunction basis. Both models converge with the expected
order.

For the discontinuous square function, as in one dimension, first-order convergence in the grid
width (for the piecewise linear models) is expected in the $\Lp{1}$-norm \cite{Hesthaven2008}. This
corresponds to a convergence rate of $\frac{1}{2}$ in $\momentnumber$. The expected convergence rate
for the piecewise-linear models is achieved on average, with slightly varying values (from $0.3$ to
$0.8$) between different refinement iterations. As in one dimension, on triangles
containing or touching the discontinuity, artificial maxima are produced. On triangles away from
the discontinuity, the exact solution is reproduced, with slight deviations in the case of the
hatfunction models due to the continuity requirements. As it takes a very fine triangulation to
resolve the discontinuity even the high-order models still show visible differences compared to the
reference solution. The full moment models ($\MN$ and $\PN$) converge
(on average) slightly slower, with an observed convergence rate of approximately $0.38$.
Here, the low-order models cannot reproduce the square shape and are radially symmetric.
Higher-order models closely mimic the square but show a regular oscillation pattern within the
square domain. As in one dimension, the number of oscillations increases with the ansatz order.
In $\Lp{\infty}$ norm, none of the models converge due to the famous Gibbs phenomenon.

\begin{figure}[htbp]
	\centering
\externaltikz{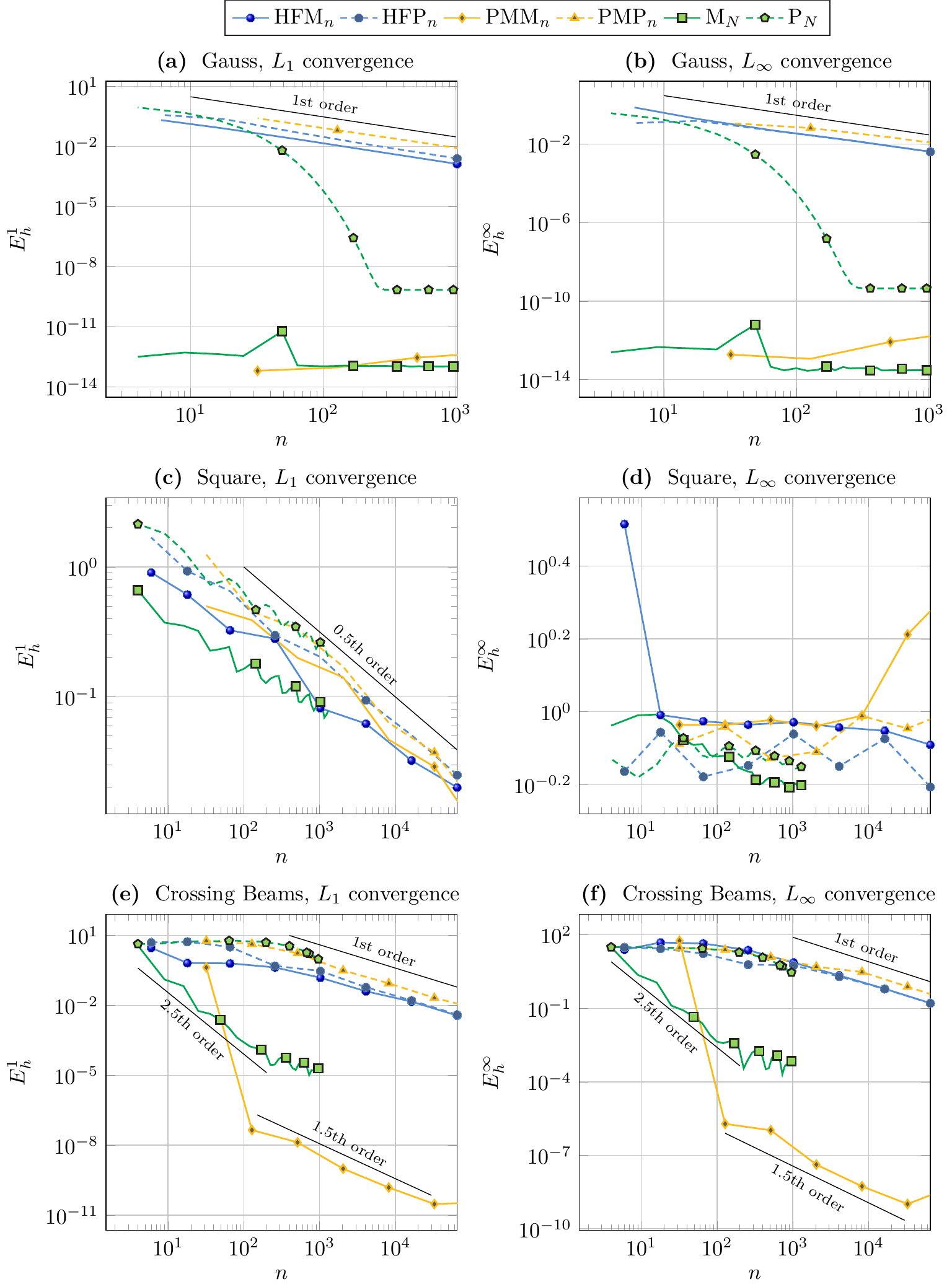}{\input{Images/DensityApproximations3d}}
\caption{Approximation errors for prescribed kinetic densities in three dimensions.}
	\label{fig:DensityApproximations3d}
\end{figure}

\begin{figure}[htbp]
	\centering
	\externaltikz{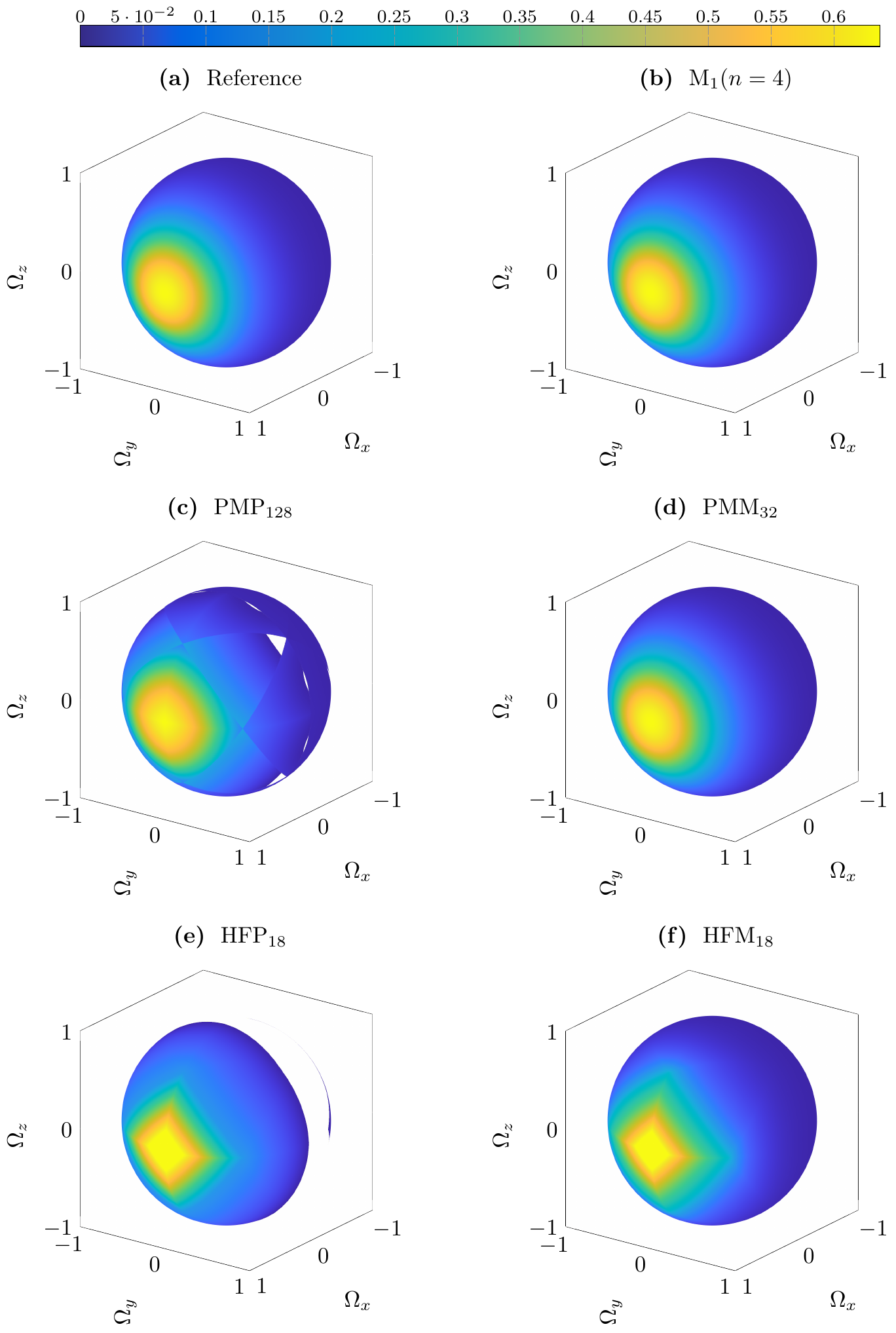}{\input{Images/DensityApproximations3dGauss}}
        \caption{Ansatz densities of selected models for the three-dimensional Gauss test. The function values are
          represented by the color, restricted to the interval $[0,\frac{2}{\pi}]$.
Unphysical negative values are shown in white.}
	\label{fig:DensityApproximations3dGauss}
\end{figure}

\begin{figure}[htbp]
	\centering
	\externaltikz{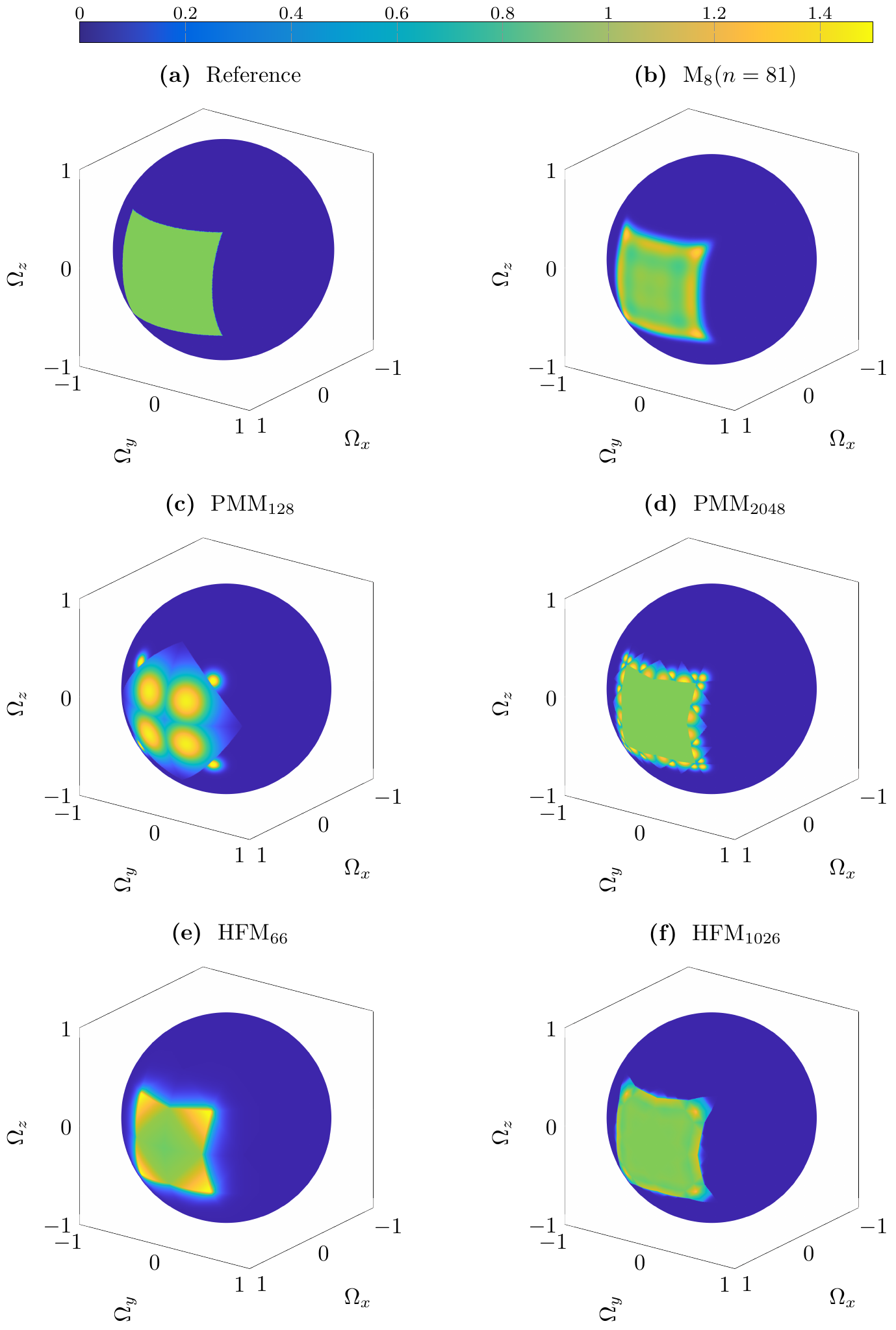}{\input{Images/DensityApproximations3dSquare}}
        \caption{Ansatz densities of selected models for the three-dimensional square test. The function values
        are represented by the color, restricted to the interval $[0,\frac32]$.}
	\label{fig:DensityApproximations3dSquare}
\end{figure}

\begin{figure}[htbp]
	\centering
\externaltikz{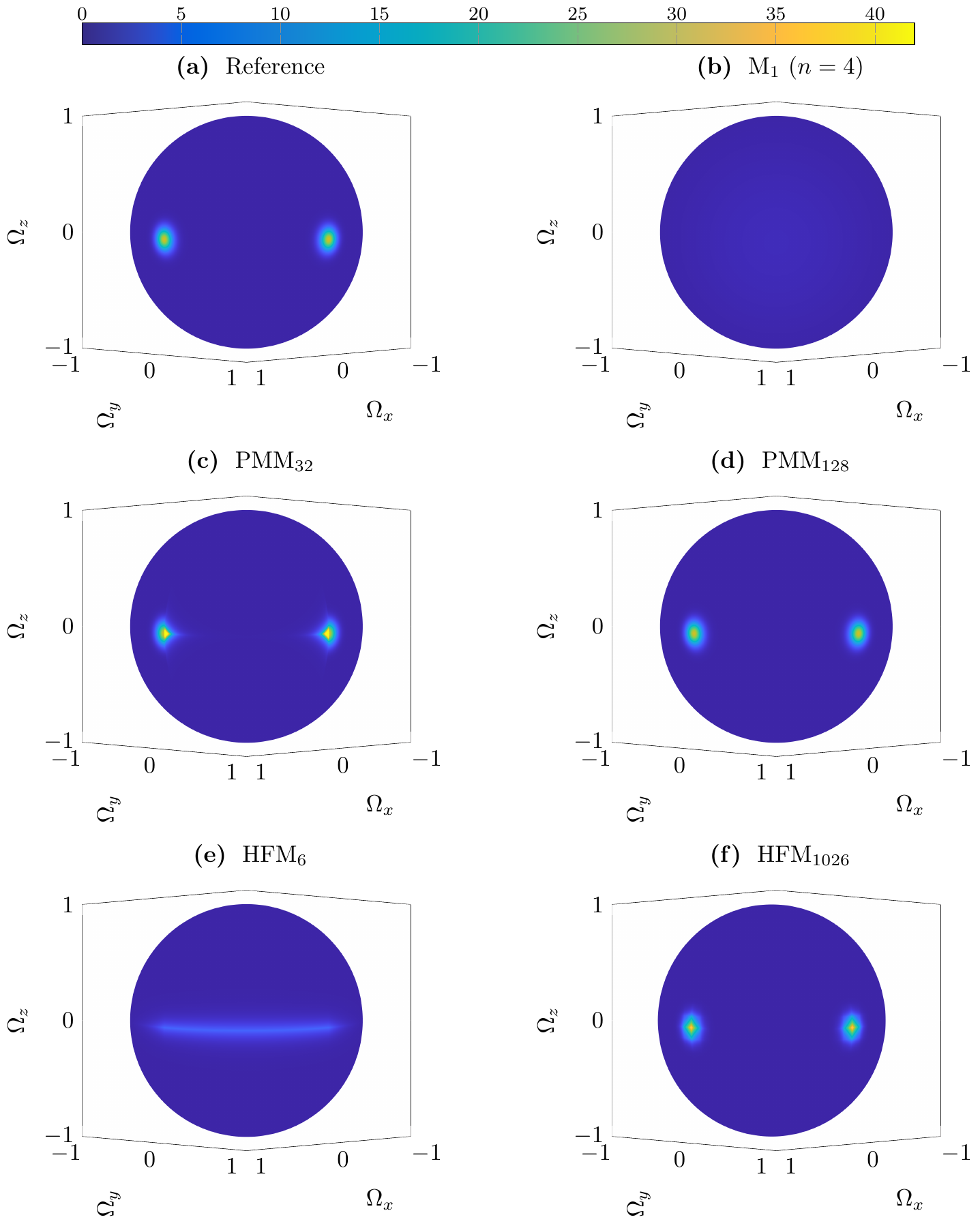}{
\input{Images/DensityApproximations3dCrossingBeams}}
\caption{Ansatz densities of selected models for the three-dimensional crossing beams
          test. The function values are represented by the color, restricted to the interval
$[0,42]$.}
	\label{fig:DensityApproximations3dCrossingBeams}
\end{figure}

\subsection{Computational complexity}

Solving the minimum entropy moment model \eqref{eq:MomentSystemClosed} requires the (numerical)
solution of the non-linear minimum entropy optimization problem \eqref{eq:dual} at every point on
the space-time grid. Though inherently parallelizable (the optimization
problems at a given time point are independent of each other) the computational cost often is
prohibitively high for practical applications. Moreover, the numerical scheme has to guarantee that
the moment vectors stay in the realizable set at all times as the optimization problem is not
solvable otherwise. Given the complicated realizability conditions for the $\MN$ models (see
\thmref{thm:FullMomentRealizability}), this is no easy task. Finally, the integrals needed during
the optimization usually cannot be evaluated analytically for the $\MN$ models. Hence, a numerical
quadrature has to be used which further complicates realizability.

The piece-wise linear models avoid many of these problems. The local support allows to use sparse
or blocked containers in the implementation of the optimization algorithm. The realizability
conditions are much simpler and, apart from the partial moment models in three dimensions, the
numerically realizable set agrees with the analytically realizable set (see
\secref{sec:Realizability}) which means that using a quadrature does not pose additional
problems. In addition, analytical formulas for the integrals are available.

A thorough numerical investigation of the piece-wise linear minimum entropy moment models
\eqref{eq:MomentSystemClosed} will be done in a follow-up to this paper. Here, we only analyse
the performance of solving a single optimization problem. We focus on the two crossing beams
test cases which are representative of all our numerical tests. In general, the optimization algorithm
converges faster for all models in the Gauss test cases and slower in the discontinuous ones
but the relative behavior stays the same.

The results can be found in \figref{fig:Timings}. For the $\MN$ models, the computation times increase
rapidly with higher order. The increase is quadratic in the number of moments $\momentnumber$ in
three dimensions and approximately of order $1.5$ in one dimension. In contrast, the execution times
are basically independent of $n$ for the $\HFMN$ and $\PMMN$ models which are thus several orders of
magnitude faster for large moment numbers. Here, due to the local support, only a
fixed number of basis functions is non-zero at each point and thus the computational
effort to, e.g., calculate the hessian matrix of the objective function only depends on the number
of quadrature points. The zig-zag pattern in \figref{fig:Timings}a) results from our choice of
quadrature intervals (see the beginning of \secref{sec:results}) which does not
yield exactly the same number of quadrature points for each model. A practical implementation
would probably choose the quadrature suitable to the model, thus using a coarser quadrature for models with
fewer moments. In that case, an additional linear increase in computational time with the number of
quadrature points would be expected for all models.

\begin{figure}[htbp]
	\centering
\externaltikz{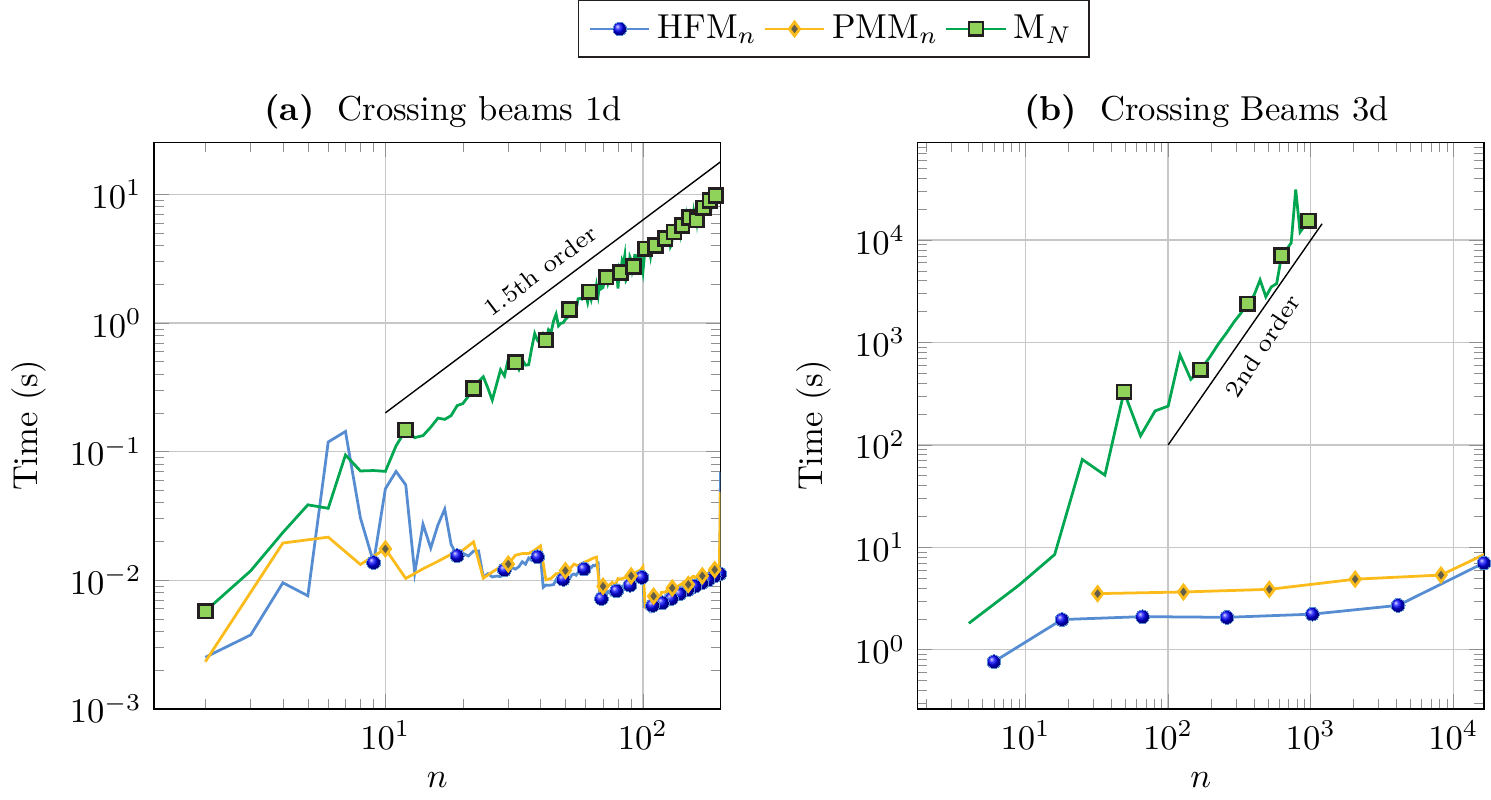}{\input{Images/TimingsPart1}}
	\caption{Timings for the minimum entropy optimization problem in the crossing beams
testcases.}
	\label{fig:Timings}
\end{figure}

%% file: Sections/outlook.tex
\section{Conclusions and outlook}
We have derived minimum-entropy models for continuous and discontinuous piece-wise linear basis
functions in one dimension and on the unit sphere. Additionally, we gave the corresponding
realizability conditions, which also provide the set of moments for which these minimum-entropy
models are hyperbolic and well-defined. Finally, numerical tests show the expected second- and
first-order convergence to smooth and discontinuous kinetic densities, respectively.

For smooth tests, the standard minimum-entropy $\MN$ models show much better approximation
properties than the new $\HFMN$ and $\PMMN$ models at equal moment number $\momentnumber$. They
should thus be preferred if smooth kinetic densities are expected. However, in practical applications,
it is rarely known in advance that the kinetic densities are smooth, and non-smooth densities
frequently occur. In that case, the new models give competitive
approximations and are several orders of magnitude faster to compute. If in addition the ansatz
triangulation is fitted to the discontinuities, the piece-wise linear models are even able to
outperform the full moment models also with respect to approximation quality. Thus, approaches
to adaptively choose the ansatz triangulation will be investigated further.

Comparing the two new model classes, continuous and discontinuous piece-wise linear basis functions
generally perform very similar, though a direct comparison in three dimensions unfortunately is
impossible due to the difference in the degrees of freedom (e.g.\ $\PMMN[32]$
often yields a better result than $\HFMN[18]$ but is outperformed by the next refinement
$\HFMN[66]$).

The sequel of this paper will deal with higher-order realizability-preserving numerical schemes for
this type of moment models and an extensive numerical analysis for several classical benchmark
problems. Additionally, we investigate the effects of numerical quadrature on the realizability set.

In future work, generalizations of the continuous piece-wise linear basis function to higher-order
splines on the unit sphere \cite{alfeld1996bernstein} could be considered, as well as higher-order
partial moments. However, realizability theory even for second order is challenging and has yet to
be found, except for very few special cases like the full moment basis \cite{Ker76}.
Additionally, the development of Kershaw closures as in \cite{Ker76,Monreal,Schneider2015,Schneider2016a} for these
new bases would be particularly interesting, as they provide a realizable moment method while avoiding the
costly non-linear optimization problem of the minimum-entropy model.